\documentclass{amsart}

\usepackage{graphicx}
\usepackage{stmaryrd}
\usepackage{tikz-cd}
\usepackage{hyperref}
\usepackage{comment}
\usepackage{enumitem}

\newtheorem{theorem}{Theorem}[section]
\newtheorem{corollary}{Corollary}[theorem]
\newtheorem{proposition}[theorem]{Proposition}
\newtheorem{lemma}[theorem]{Lemma}

\theoremstyle{definition}
\newtheorem{definition}[theorem]{Definition}

\newtheorem{remark}[theorem]{Remark}

\numberwithin{equation}{section}

\newcommand{\bX}{\mathbb{X}}
\newcommand{\bx}{{\bf x}}
\newcommand{\bu}{{\bf u}}
\newcommand{\bk}{{\bf k}}
\newcommand{\bV}{\mathbb{V}}
\newcommand{\bY}{\mathbb{Y}}
\newcommand{\bG}{\mathbb{G}}
\newcommand{\bP}{\mathbb{P}}
\newcommand{\by}{\mathbf{y}}
\newcommand{\bA}{\mathbb{A}}

\newcommand{\bL}{\mathbb{L}}
\newcommand{\R}{\mathbb{R}}

\newcommand{\Q}{\mathbb{Q}}
\newcommand{\Z}{\mathbb{Z}}

\newcommand{\bM}{\mathbb{M}}
\newcommand{\C}{\mathbb{C}}

\renewcommand{\H}{\mathbb{H}}
\newcommand{\F}{\mathbb{F}}

\newcommand{\Oo}{\mathcal{O}}

\newcommand{\cZ}{\mathcal{Z}}
\newcommand{\cH}{\mathcal{H}}

\newcommand{\cY}{\mathcal{Y}}

\newcommand{\cG}{\mathcal{G}}
\newcommand{\cN}{\mathcal{N}}
\newcommand{\cW}{\mathcal{W}}

\newcommand{\cM}{\mathcal{M}}
\newcommand{\cF}{\mathcal{F}}
\newcommand{\Exc}{\mathrm{Exc}}

\newcommand{\Kra}{\mathrm{Kra}}
\newcommand{\Int}{\mathrm{Int}}
\newcommand{\Pap}{\mathrm{Pap}}
\newcommand{\GL}{\mathrm{GL}}

\newcommand{\ord}{\mathrm{ord}}
\newcommand{\loc}{\mathrm{loc}}
\newcommand{\Blow}{\mathrm{Blow}}
\newcommand{\Lie}{\mathrm{Lie}\, }
\newcommand{\Nilp}{\mathrm{Nilp}\, }
\newcommand{\Proj}{\mathrm{Proj}\, }

\newcommand{\Spec}{\mathrm{Spec}\, }
\newcommand{\Spf}{\mathrm{Spf}\, }
\newcommand{\SpfOF}{{\mathrm{Spf}\,\mathcal{O}_{\breve{F}} }}

\newcommand{\Herm}{\mathrm{Herm}}
\newcommand{\Hom}{\mathrm{Hom}}
\newcommand{\End}{\mathrm{End}}
\newcommand{\rU}{\mathrm{U}}
\newcommand{\barchi}{\bar{\chi}}
\newcommand{\barpsi}{\bar{\psi}}
\newcommand{\q}{q}

\begin{document}

\title{Special cycles on unitary Shimura curves at ramified primes}

\author{Yousheng Shi}
\address{Department of Mathematics, University of Wisconsin Madison, Van Vleck Hall, Madison, WI 53706, USA}
\email{shi58@wisc.edu, shiyousheng0318@gmail.com}


%    General info
%\subjclass[2000]{Primary 54C40, 14E20; Secondary 46E25, 20C20}

\date{\today}

%\keywords{Differential geometry, algebraic geometry}

\begin{abstract}
In this paper, we study special cycles on the Kr\"amer model of $\mathrm{U}(1,1)(F/F_0)$-Rapoport-Zink spaces where $F/F_0$ is a ramified quadratic extension of $p$-adic number fields with the assumption that the $2$-dimensional hermitian space of special quasi-homomorphisms is
anisotropic. We write down the decomposition of these special cycles and prove a version of Kudla-Rapoport conjecture in this case. We then apply the local results to compute the intersection numbers of special cycles on unitary Shimura curves and relate these intersection numbers to Fourier coefficients of central derivatives of certain Eisenstein series.
\end{abstract}

\maketitle

\maketitle

\setcounter{tocdepth}{1}% hide subsections in table of contents
\tableofcontents

\section{Introduction}
The classical Siegel-Weil formula expresses certain Siegel Eisenstein series as generating series of representation numbers of quadratic forms. Starting from  \cite{Kudla97}, Kudla initiated an influential program that seeks to establish the arithmetic Siegel-Weil formula, which expresses central derivatives of these series as arithmetic intersection numbers of $n$ special divisors in the Shimura varieties associated to $\mathrm{GSpin}(n-1,2)$ or $\mathrm{GU}(n-1,1)$. 

In \cite{KR1} and \cite{KR2}, Kudla and Rapoport defined integral models of unitary Shimura varieties associated to $\mathrm{GU}(n-1,1)$ over an imaginary quadratic field $\bk$ and special cycles associated to a certain class of special endomorphisms on them. They showed that the intersection of $n$ independent special divisors is supported on the supersingular locus of the Shimura variety over a finite set of primes of $\Q$ inert or ramified in $\bk$. Over an inert prime $p$, they formulated a local version of the arithmetic Siegel-Weil formula (the local Kudla-Rapoport conjecture), which expresses intersection number of $n$ special divisors on the corresponding Rapoport-Zink spaces as derivative of local densities of hermitian forms. They then explained how the nonarchimedean component of the arithmetic Siegel-Weil formula (also known as the global Kudla-Rapoport conjecture) follows from the local Kudla-Rapoport conjecture.  

In \cite{KR1} and \cite{KR2}, Kudla and Rapoport proved their conjectures when the support of the intersection has zero dimension, which essentially reduces to the case when $n=1,2$. Recently, Li and Zhang \cite{LZ} proved the conjecture for general $n$.

Over a prime $p$ ramified in $\bk$, it has been speculated that there should be analogues of Kudla-Rapoport conjectures. In this case, one can consider the integral model of Shimura variety and its corresponding Rapoport-Zink spaces proposed by Kr\"amer \cite{Kr}. The Kr\"amer model is regular and has semi-stable reduction over $p$. 

In this paper and \cite{HSY}, we prove the analogue of
Kudla–Rapoport conjecture for the Kr\"amer models when $n = 2$ at odd ramified primes.
These results shed light on how  Kudla-Rapoport type of conjectures should be formulated for the Kr\"amer model at ramified primes for general $n$. 

\subsection{Special cycles on Rapoport-Zink spaces}
Let $p$ be an odd prime integer. Let $F_0$ be a finite extension of $\Q_p$ and $F$ be a ramified quadratic extension of $F_0$. Denote by $\Oo_F$ and $\Oo_{F_0}$ their rings of integers respectively. Assume the residue field of $\Oo_{F_0}$ is $\F_\q$ and let $k$ be its algebraic closure. We denote by $\bar{a}$ the Galois conjugate of $a\in F$ over $F_0$. We fix a uniformizer $\pi\in F$ such that $\pi_0=\pi^2\in \Oo_{F_0}$. Let $\breve{F}_0$  be the completion of a maximal unramified extension of $F_0$ and $\breve{F}:=F\otimes_{F_0} \breve{F}_0$. Let $\Oo_{\breve F}$ and $\Oo_{\breve{F}_0}$ be the rings of integers of $\breve{F}$ and $\breve{F}_0$ respectively. Let $\Herm_n$ be the group scheme over $\Oo_{F_0}$ defined by 
\begin{equation}\label{eq:Herm_n local}
    \Herm_n(R)=\{X\in M_{n\times n} (R\otimes_{\Oo_{F_0}} \Oo_F)\mid X=X^*\}\
\end{equation}
for an $\Oo_{F_0}$-algebra $R$ where $X^*=\bar{X}^t$.

We consider the formal moduli space $\cN^\Pap$  which is the  functor over $\Spf \Oo_{\breve{F}}$ whose $S$-point is the groupoid of isomorphism classes of quadruples $ (X,\iota,\lambda,\rho)$ over $S$ for an $\Oo_{\breve{F}_0}$-scheme $S$. Here $X$ is a strict formal $\Oo_{F_0}$-module (which is a particular kind of formal $p$-divisible group, see Section \ref{sec:modulispaces}) over $S$ of dimension $n$ and $F_0$-height $2n$. The action $\iota:\Oo_{F}\rightarrow \End(X)$ is a ring homomorphism such that the induced action on $\Lie X$ satisfies certain signature conditions. 
$\lambda:X\rightarrow X^\vee$ is a principal polarization compatible with $\iota$. Finally, $\rho$ is an $\Oo_F$-linear quasi-isogeny of height $0$ from $X\times_S \bar{S} $ to $\bX\times_{\Spec k} \bar{S}$ where $\bar{S}=S\times \Spec k$ and $(\bX,\iota_\bX,\lambda_\bX)$ is a fixed framing object. $\cN^\Pap$ is representable by a formal scheme over $\Spf \Oo_{\breve{F}}$ which is normal, flat of relative dimension $n-1$ over $\Spf \Oo_{\breve{F}}$ and has isolated singularities (see \cite[Theorem 4.5]{P}). We refer to Section \ref{sec:modulispaces} for a detailed definition of $\cN^\Pap$ and $\cN^\Kra$ below.

In order to resolve the singularities, \cite{Kr} proposes a new moduli problem $\cN^\Kra$ which is flat over $\Spf \Oo_{\breve F}$, has semi-stable reduction and has a natural morphism $\Phi:\cN^\Kra \rightarrow \cN^\Pap$. It is known to the experts that $\cN^\Kra$ is in fact the blow-up of $\cN^\Pap$ along its singular locus. But due to the lack of a precise reference, we provide a proof of this fact in Appendix \ref{sec:Kramer model is blow up}.
The exceptional divisor of this blow-up is denoted by $\Exc$ and is isomorphic to $\bP^{n-1}/k$. 

In the rest of this subsection we specialize to the case $n=2$.
Let $\bY$ be the unique strict formal $\Oo_{F_0}$-module of dimension $1$ and $F_0$-height $2$. It admits an $\Oo_F$-action $\iota_\bY:\Oo_F\rightarrow \End(\bY)$ and a compatible polarization $\lambda_\bY:\bY\rightarrow \bY^\vee$. The space of special quasi-homomorphisms $\bV=\Hom_{\Oo_F}(\bY,\bX)\otimes_{\Z} \Q $ is a $2$ dimensional $F$-vector space with the natural hermitian form
\[h(x,y)=\iota_\bY^{-1} (\lambda_\bY^{-1}\circ y^\vee \circ \lambda_\bX \circ x)\]
where $y^\vee$ is the dual homomorphism of $y$. Let us briefly recall that $n$ dimensional hermitian spaces associated to the quadratic extension $F/F_0$ are classified by the Hasse invariant which is the class 
\begin{equation}\label{eq:Hasse invariant}
    (-1)^{\frac{n(n-1)}{2}} d \in F_0^\times/\mathrm{Nm}(F^\times)\cong \{\pm 1\}
\end{equation}
where $d$ is the determinant of a given hermitian form.
When $n=2$, we say a hermitian space is anisotropic if it does not contain a line of length zero elements, or equivalently its Hasse invariant is $-1$. Otherwise we say it is isotropic.
When $(\bV,h(,))$ is anisotropic, by \cite[Section 8]{RSZ} there is an isomorphism
\[\cN^\Pap\cong \cM_{\Gamma_0(\pi)}=\Spf \Oo_{\breve{F}}[[X,Y]]/(XY-\pi_0),\]
where $\cM_{\Gamma_0(\pi)} $ is the deformation space of degree $\q$ isogenies between Lubin-Tate formal groups.
When $(\bV,h(,))$ is isotropic, then by \cite{KR3} (see also \cite{RSZ}), $\cN^\Pap$ is isomorphic to the formal scheme whose generic fiber is the Drinfeld $p$-adic half plane. Our key assumption in this paper is 
\begin{equation}\label{eq:anisotropicassumption}
    (\bV,h(,))\text{ is anisotropic},
\end{equation}
which we make in the rest of the introduction.
The isotropic case is treated in \cite{HSY}.

For any $\bx\in\bV^m$ ($1\leq m \leq 2$), we define $\cZ^\Kra(\bx)$  (resp. $\cZ(\bx)$) to be the closed subscheme of $\cN^\Kra$ (resp. $\cN^\Pap$) where the quasi-homomorphism $\rho^{-1}\circ \bx:\bY^m \rightarrow \bX$ deforms to a homomorphism (see Definition \ref{def:localspecialcycle}).
By \cite[Proposition 4.3]{Ho2}, $\cZ^\Kra(x)$ for a single $x\in \bV$ is a Cartier divisor. 
In order to decompose $\cZ^\Kra(x)$, we define Weil divisors $\cY^\kappa_{s,\pm}$ (see Section \ref{subsec:quasi canonical lifting on NPap}) on $\cN^\Pap$ and Cartier divisors $\cZ^\kappa_{s,\pm}$ (see Definition \ref{def:quasicanonicaldivisoronKramer}) on $\cN^\Kra$ for an integer $s\geq 0$. Here $\kappa$ is an embedding of $\Oo_F$ into the quaternion algebra with center $F_0$ determined by $\bx$, see \eqref{eq:kappa x}.
The pullback of the universal formal $\Oo_{F_0}$-module over $\cN^\Pap$ (resp. $\cN^\Kra$) to these divisors are quasi-canonical lifts of level $s$ in the sense of \cite{G}. Moreover $\cZ^\kappa_{s,\pm}$ is the strict transform of $\cY^\kappa_{s,\pm}$ under the blow-up $\Phi:\cN^{\Kra}\rightarrow \cN^\Pap$.
When $s=0$, $\cZ^\kappa_{0,+}$ and $\cZ^\kappa_{0,-}$ are the same and is denoted by $\cZ^\kappa_0$.
Our first major result is then the following decomposition theorem (Theorem \ref{decompositionofspecialcycles}).
\begin{theorem}
For any $\bx \in \bV\setminus\{0\}$,
we have the following equality of Cartier divisors on $\cN^\Kra$.
\[\cZ^\Kra(\bx)=\cZ^\kappa_0+\sum_{s=1}^a \cZ^\kappa_{s,-}+\sum_{s=1}^{a} \cZ_{s,+}^\kappa+(a+1) \Exc,\]
where $a$ is the $\pi_0$-adic valuation of $h(\bx,\bx)$.
\end{theorem}

The reduced locus  of $\cN^\Kra$ is its exceptional divisor. In particular it is proper over $\Spec k$.
For a full rank lattice $\bL\subset \bV$ and a basis $\{x_1,x_2\}$ of $\bL$, we define the intersection number 
\begin{equation}\label{IntL}
    \Int(\bL)=\chi(\cN^\Kra, \Oo_{\cZ^\Kra(x_1)}\otimes^{\mathbb{L}} \Oo_{\cZ^\Kra(x_2)}),
\end{equation}
where $\Oo_{\cZ^\Kra(x)}$ is the structure sheaf of $\cZ^\Kra(x)$, $\otimes^{\mathbb{L}}$ is the derived tensor product of coherent sheaves on $\cN^\Kra$ and $\chi$ denotes the Euler-Poincar\'e characteristic. By \cite{Ho2}, this quantity is independent of the choice of $\{x_1, x_2\}$, hence is an invariant of $\bL$. We say two hermitian matrices $T_1$ and $T_2$ in $\Herm_n(F_0)$ are equivalent if there is a $g\in \GL_n(\Oo_F)$ such that $g^* T_1 g=T_2$.
We then prove the following theorem (Theorem \ref{thm:maintheorem1}).
\begin{theorem}\label{thm:IntrothmA}
Suppose $(\bV,h(,))$ is anisotropic. 
Let $\bL$ be a hermitian $\Oo_F$-lattice in $\bV$ whose hermitian form is represented by a Gram matrix equivalent to $T=\left(\begin{array}{cc}
    u_1 (-\pi_0)^a & 0 \\
    0 & u_2 (-\pi_0)^b
\end{array}\right)$,where $u_1,u_2\in \Oo_{F_0}^\times$ and $a,b\geq 0$. Then 
\[\mathrm{Int}(\bL)=2 \sum_{s=0}^{min\{a,b\}}\q^s(a+b+1-2s)-a-b-2.\]
\end{theorem}

\subsection{Local density}
Fix an additive character $\psi$ of $F_0$ with conductor $\Oo_{F_0}$.
For $S\in\Herm_m(F_0)$ and $T\in\Herm_n(F_0)$, define
\begin{equation}
    \alpha(S, T) = \int_{\Herm_n(F_0)} \int_{M_{m,n}(\Oo_F)} \psi(\mathrm{Tr}( Y(S[X]-T))) dX dY
\end{equation}
where $dX$ and $dY$ are Haar measures on $M_{m,n}(\Oo_F)$ and $\Herm_n(\Oo_{F_0})$ respectively such that $\mathrm{vol}(M_{m,n}(\Oo_F), dX) =\mathrm{vol}(\Herm_n(\Oo_{F_0}), dY)=1$ and $\mathrm{Tr}$ stands for the matrix trace. The quantity $\alpha(S,T)$ is independent of the choice of $\psi$ and is called the representation density of $T$ by $S$, or simply (hermitian) local density. We refer to Lemma \ref{lem:definitionoflocaldensity} to justify the name. Define 
\[S_r=S\oplus \frac{1}{\pi}\left(\begin{array}{cc}
    0 & I_r \\
    -I_r & 0
\end{array}\right).\]
It can be shown that there is a polynomial $\alpha(S,T,X)\in \Q[X]$ such that $\alpha(S_r,T)=\alpha(S,T,\q^{-2r})$, see \cite[Corollary 5.4]{Hi}. We define the derivative of the local density $\alpha(S,T)$ to be 
\begin{equation}\label{eq:alphaprime}
  \alpha'(S,T)=-\frac{\partial}{\partial X} \alpha(S,T,X)|_{X=1}.  
\end{equation}

Now go back to the situation of Theorem \ref{thm:IntrothmA}.
We take $S$ to be a Gram matrix of the unique self-dual hermitian lattice (up to hermitian isometries) of rank $2$ which has Hasse invariant $1$.  We can compute
$\alpha(S,T,X)$ and $\alpha(S,S)$ following the idea of \cite{Hi}. Then we obtain the following theorem (see Theorem \ref{thm:maintheorem2}).
\begin{theorem}\label{thm:IntrothmB}
Let $\bL,T$ be as in Theorem \ref{thm:IntrothmA} and 
$S$ be as above. Then 
\[   \mathrm{Int}(\bL)=2\cdot \frac{\alpha'(S,T)}{\alpha(S,S)}.\]
\end{theorem}
\begin{remark}
The factor $2$ in the theorem is essentially the ramification index of $F/F_0$ which comes naturally from global considerations, see \eqref{eq:degcZ} below. One can compare with the inert case, see \cite[Theorem 3.4.1]{LZ}, where this number is $1$. 
\end{remark}
\begin{remark}
In general when the Rapoport-Zink space has bad reduction, one needs to modify the analytic (right hand) side of Theorem \ref{thm:IntrothmB}. This is first discovered in \cite{KRshimuracurve}. See also \cite{San3}, \cite[Section 10]{LZ} and \cite{HSY}.
\end{remark}
We also obtain formulas of $\alpha(S,T,X)$ for some other choices of $S$ and $T$ which might be of independent interest, see Theorem \ref{thm:localdensity calculation}.

\subsection{Global moduli problem and Eisenstein series}
In the second part of the paper, we apply the local results above to the global intersection problem proposed by \cite{KR2}. 
Let $\bk$ be an imaginary quadratic field.
As in the local case, let $\Herm_n$ be the group scheme over $\Z$ defined by
\begin{equation}\label{eq:Herm_n global}
   \Herm_n(R)=\{X\in M_{n\times n} (R\otimes_{\Z} \Oo_\bk)\mid X=X^*\} 
\end{equation}
for any algebra $R$.
Let $\cM^\Kra_{(n-1,1)}$ be the moduli functor over $\Spec \Oo_\bk$ which parametrizes principally polarized abelian varieties $A$ with an action $\iota:\Oo_\bk \rightarrow \End(A)$, a polarization $\lambda: A\rightarrow A^\vee$ compatible with $\iota$ and a filtration $\cF_A\subset \Lie A$ which satisfies the signature $(n-1,1)$ condition (see Section \ref{sec:globalintersecionnumber}).
In particular $\cM_{(1,0)}:=\cM^\Kra_{(1,0)}$ is the moduli space of CM elliptic curves with $\Oo_{k}$-action.
Let $V$ be a  hermitian vector space over $\bk$ of signature $(1,1)$ containing a self-dual lattice $L$. The lattice $L$ determines an open and closed substack
\[\cM\subset\cM_{(1,0)}\times_{\Spec \Oo_\bk} \cM^\Kra_{(1,1)}\]
which is an integral model of a unitary Shimura curve.
For a point in $\cM(S)$ ($S$ an $\Oo_\bk$-scheme), i.e., 
a pair of data $(E,\iota_0,\lambda_0)\in \cM_{(1,0)}(S)$ and $(A,\iota,\lambda,\cF_A)\in \cM^\Kra_{(1,1)}(S)$, define the projective $\Oo_\bk$-module of finite rank
\[V'(E,A)=\Hom_{\Oo_\bk}(E,A).\]
It is equipped with the hermitian form $h'(x,y)=\iota_0^{-1}(\lambda_0^{-1}\circ y^\vee \circ \lambda \circ x)$. For a $m \times m$ positive-definite hermitian matrix $T$ with values in $\Oo_\bk$, let $\cZ^\Kra(T)$ be the stack of collections $(E,\iota_0,\lambda_0, A,\iota, \lambda,\cF_A,\bx)$ where $ (E,\iota_0,\lambda_0, A,\iota, \lambda,\cF_A)\in \cM(S)$ and $\bx \in V'(E,A)^m$ with $h'(\bx,\bx)=T$. The moduli space $\cZ^\Kra(T)$ is a representable by a Deligne-Mumford stack which is finite and unramfied over $\cM$. When $t\in \Z_{>0}$, $\cZ^\Kra(t)$ is a divisor by \cite[Proposition 3.2.3]{Ho1}. 

Let $T\in \Herm_2(\Z)$ and assume it is positive definite. Let $V_T$ be the hermitian space with Gram matrix $T$. Let $V_p=V\otimes_\Q \Q_p$. Define
\begin{equation}\label{eq:Diff}
    \mathrm{Diff}(V_T,V):=\{\ell \text{ a finite prime of } \Q \mid V_\ell \text{ is not isomorphic to } (V_T)_\ell\}
\end{equation}
where isomorphic means isomorphic as hermitian spaces. 
Then $\cZ^\Kra(T)$ is empty if $|\mathrm{Diff}(V_T,V)|>1$. If $\mathrm{Diff}(V_T,V)=\{p\}$ for a prime $p$ inert or ramified in $\bk$,
it is proved in \cite[Proposition 2.22]{KR2} that 
the support of $\cZ^\Kra(T)$ is on the supersingular locus of $\cM$ over the residue field of $\bk_p$. Let $e$ be the ramification index of $\bk_p/\Q_p$.
Define
\begin{equation}\label{eq:degcZ}
    \widehat{\mathrm{deg}}(\cZ^\Kra(T))=\chi(\cZ^\Kra(T),\Oo_{\cZ^\Kra(x_1)}\otimes^{\mathbb{L}} \Oo_{\cZ^\Kra(x_2)})\cdot \log p^{2/e},
\end{equation}
where $t_1, t_2$ are the diagonal entries of $T$.

Let $E(z,s,L)$ (see section \ref{sec:Eisensteinseries}) be the incoherent Eisenstein series considered by \cite{KR2}. It can be analytically continued to $s\in \C$ and satisfies a functional equation centered at $s=0$. Let $E'_T(z,0,L)$ be the $T$-th Fourier coefficient of its central derivative. Then we have the following theorem (Theorem \ref{globalmainthm}).

\begin{theorem}
Let $T\in \Herm_2(\Z)$ be positive definite and assume $\mathrm{Diff}(V_T,V)=\{p\}$ where $p$ is an odd prime ramified in $\bk$ such that $V_{p}$ is isotropic. Then
\[E'_T(z,0,L)=C_1 \cdot \widehat{\mathrm{deg}}(\cZ^\Kra(T)) \cdot q^T, \]
where $C_1$ is a constant that only depends on $\bk$ and $L$ and $q=\exp(2\pi i \mathrm{Tr}(Tz))$.
\end{theorem}
\begin{remark}
The hermitian space $V_p$ in the theorem has the opposite Hasse invariant as the hermitian space $\bV$ with hermitian form $h(,)$.
\end{remark}

\noindent
{\bf Acknowledgements.}
First we would like to thank Michael Rapoport for suggesting studying special cycles at ramified primes. We would also like to thank Qiao He, Brian Smithling, Tonghai Yang and Wei Zhang for helpful discussions. Finally, we would like thank the referee for valuable suggestions.

\noindent
{\bf Data availability statement:}
The paper has no associated data.

\part{Local theory}
Denote by $\sigma$ the Frobenius of $k/\F_\q$ and $\breve{F}_0/F_0$. Let $\mathbb{H}$ be the unique quaternion algebra with center $F_0$ and $\Oo_{\mathbb{H}}$ be its ring of integers. Let $\mathrm{val}_{\pi_0}$ ($\mathrm{val}_{\pi}$ resp.) be the valuation on $\breve F$ such that $\mathrm{val}_{\pi_0}(\pi_0)=1$ ($\mathrm{val}_{\pi}(\pi)=1$ resp.).
There exists a $\delta\in \Oo_{\mathbb{H}}^\times$ such that $\delta^2\in F_0$ and $\delta \pi=-\pi \delta$.
For a local $p$-adic ring $\Oo$, let $\Nilp \Oo$ be the category of $\Oo$-schemes $S$ such that $p\cdot \Oo_S$ is a locally nilpotent ideal sheaf. For $S\in \Nilp \Oo$, let $\bar{S}$ be its special fiber. Unless otherwise specified, all tensor products are defined over $\SpfOF$.

\section{Moduli spaces}\label{sec:modulispaces}
\subsection{Rapoport-Zink spaces}\label{subsec:RZ spaces}
For a scheme $S$ over $\Spec \Oo_{F_0}$, a formal $\Oo_{F_0}$-module over $S$ is a formal $p$-divisible group $X$ over $S$ with an action $i:\Oo_{F_0}\rightarrow \End(X)$. We say $X$ is a strict formal $\Oo_{F_0}$-module if $\Oo_{F_0}$ acts on $\Lie X$ via the structure morphism $\Oo_{F_0}\rightarrow \Oo_S$.
Let $(\bX,\iota_\bX)$ be a strict formal $\Oo_{F_0}$-module of dimension $n$ and $F_0$-height $2n$ over $k$ with an action $\iota_\bX:\Oo_{F}\rightarrow \End(X)$ that extends the action of $\Oo_{F_0}$. Let $\lambda_{\bX}$ be a principal polarization whose associated Rosati involution induces on $\Oo_F$ the nontrivial Galois automorphism over $F_0$. 

We consider the Rapoport-Zink space $\cN^\Pap$, a functor over $\Spf \Oo_{\breve{F}}$ which represents the moduli problem such that $\cN^\Pap(S)$ is the groupoid of isomorphism classes of quadruples $ (X,\iota,\lambda,\rho)$ over $S$ for $S\in \Nilp \Oo_{\breve F}$. Here $X$ is a strict formal $\Oo_{F_0}$-module over $S$ of dimension $n$ and $F_0$-height $2n$ with an $\Oo_F$-action $\iota:\Oo_{F}\rightarrow \End(X)$ that extends the action of $\Oo_{F_0}$ satisfying the Kottwitz signature condition
\begin{equation}\label{eq:Kottwitzcondition}
    \mathrm{char}(\iota(\pi)|\Lie X)=(T-\pi)^{n-1}(T+\pi),
\end{equation}
and the Pappas wedge condition proposed in \cite{P}
\begin{equation}\label{eq:wedgecondition}
    \wedge^n(\iota(\pi)+\pi|\Lie X)=0, \ \wedge^2(\iota(\pi)-\pi|\Lie X)=0.
\end{equation}
$\lambda:X\rightarrow X^\vee$ is a principal polarization whose associated Rosati involution induces on $\Oo_F$ the nontrivial Galois automorphism over $\Oo_{F_0}$ where $X^\vee$ is the dual of $X$ in the category of strict formal $\Oo_{F_0}$-modules.
Finally $\rho:X\times_S \bar{S} \rightarrow \bX \times_{\Spec k} \bar{S}$ is an $\Oo_F$-linear quasi-isogeny of height $0$ such that $\lambda$ and $\rho^*(\lambda_\bX)$ differ locally on $\bar{S}$ by a factor in $\Oo_{F_0}^\times$. An isomorphism between two quadruples $(X,\iota,\lambda,\rho)$ and $(X',\iota',\lambda',\rho')$ is given by an $\Oo_F$-linear isomorphism $\alpha:X\rightarrow X'$ such that $\rho'\circ (\alpha \times_S \bar{S})=\rho$ and $\alpha^*(\lambda')$ is a $\Oo_{F_0}^\times$-multiple of $\lambda$.

By \cite[Proposition 2.1]{RTW}, $\cN^\Pap$ is representable by a separated formal scheme, locally formally of finite type over $\SpfOF$ of relative dimension $n-1$. Moreover it is normal and has isolated singularities. 
Let $(\bY,\iota_\bY,\lambda_\bY)$ be the unique strict formal $\Oo_{F_0}$-module of dimension $1$ and $F_0$-height $2$ over $\Spec k$, together with an action $\iota_\bY:\Oo_F\rightarrow \End(\bY)$ that lifts the $\Oo_{F_0}$-action and a principal polarization $\lambda_\bY: \bY\rightarrow \bY^\vee$. Let $\cN_{(1,0)}$ be the Rapoport-Zink space defined similarly as above using the framing object $(\bY,\iota_\bY,\lambda_\bY)$.
Then $\cN_{(1,0)}\cong \SpfOF$ and the universal strict $\Oo_{F_0}$-module over $\cN_{(1,0)}$ is the canonical lift $\cG$ of $\bY$ with respect to the $\Oo_F$ action (see \cite{G} or \cite[Section 8]{ARGOS}).

Let $\cN^\Kra$ be the functor over $\Spf \Oo_{\breve{F}}$ whose $S$-point is the groupoid of isomorphism classes of quintuples $ (X,\iota,\lambda,\rho,\mathcal{F}_X)$ over $S\in \Nilp \Oo_{\breve F}$. Here the quadruple $(X,\iota,\lambda,\rho)$ has exactly the same meaning as in the definition of $\cN^\Pap$ and the additional data $\cF_X$ is a $\Oo_F\otimes_{\Oo_{F_0}}\Oo_S$-stable local direct summand of rank $n-1$ of $\Lie X$. We require that $\Oo_F$ acts on $\cF_X$ via the structure map $\Oo_F\rightarrow\Oo_S$ and acts on $\Lie X/\cF_X$ via the Galois conjugate of the structure map. An isomorphism between two quintuples $(X,\iota,\lambda,\rho,\cF_X)$ and $(X',\iota',\lambda',\rho',\cF_{X'})$ is given by an $\Oo_F$-linear isomorphism $\alpha:X\rightarrow X'$ such that $\rho'\circ (\alpha \times_S \bar{S})=\rho$, $\alpha^*(\lambda')$ is a $\Oo_{F_0}^\times$-multiple of $\lambda$ and $\alpha^*(\cF_{X'})=\cF_X$.
The following is essentially due to \cite{Kr}, see Theorem 2.3.2 and 2.3.3 of \cite{BHKRY1}.
\begin{proposition}\label{prop:Kramermodel}
$\cN^\Kra$ is representable by a formal scheme flat over $\SpfOF$ which is regular and has semi-stable reduction. The natural forgetful map
\[\Phi:\cN^\Kra\rightarrow \cN^\Pap,(X,\iota,\lambda,\rho,\cF_X)\mapsto (X,\iota,\lambda,\rho) \]
is an isomorphism outside the singular locus $\mathrm{Sing}$ of $\cN^\Pap$ where $\iota(\pi)$ acts on $\Lie X$ by $0$. Moreover $\Phi^{-1}(\mathrm{Sing})$ is a union of exceptional divisors each of which is isomorphic to the projective space $\bP^{n-1}/k$.
\end{proposition}

The following fact should be well-known to experts. However due to the lack of a precise reference, we prove it in Appendix \ref{sec:Kramer model is blow up}.
\begin{proposition}\label{prop: Kra is blow up of Pappas}
$\cN^\Kra$ is the blow-up of $\cN^\Pap$ along its singular locus Sing.
\end{proposition}

\subsection{An exceptional isomorphism of moduli spaces}\label{subsec:exceptional iso}
From now on we assume $n=2$. In this case the right hand side of \eqref{eq:Kottwitzcondition} is $T^2-\pi_0$, hence is defined over $\Oo_{F_0}$. Furthermore the wedge condition \ref{eq:wedgecondition} is implied by \eqref{eq:Kottwitzcondition} so can be omitted. 

Let $\mathbb{M}:=M(\bX)$ be the relative covariant Dieudonn\'e module of $\bX$ over $\Oo_{F_0}$ and $N:=\bM \otimes_{\Z} \Q$. Recall that $N$ is a $4$-dimensional $\breve{F}_0$-vector space with a $\sigma$-linear operator $F_\bX$ and a $\sigma^{-1}$-linear operator $V_\bX$ (see \cite[Appendix B.8]{fargues2006isomorphisme}), both commuting with the $F$-action which we still denote as $\iota_\bX$. The polarization $\lambda_\bX$ induces a skew symmetric $\breve{F}_0$ valued form $\langle,\rangle_\bX$ on $N$ satisfying 
\begin{align*}
    \langle F_\bX (x),y\rangle_\bX=&\langle x,V_\bX (y)\rangle_\bX^\sigma, \\
    \langle\iota_{\bX}(a)x,y\rangle_\bX=&\langle x,\iota_{\bX}(\bar{a}y)\rangle_\bX, a\in F.
\end{align*}
We denote by $\Pi_\bX$ the endomorphism $\iota_{\bX}(\pi)$ of $N$ and we often write $\Pi$ if there is no ambiguity. Following \cite{KR3}, we can define a  form $(,)_\bX$ on $N$ by
\begin{equation}\label{eq:definitionof(,)}
(x,y)_\bX=\delta[\langle \Pi x, y\rangle_\bX+\pi \langle x, y\rangle_\bX].   
\end{equation}
Let $\tau=\Pi V_\bX^{-1}$. Then $C:=N^\tau$ is a hermitian $F$-vector space such that $C \otimes_F \breve F= N$. The key assumption of this paper is that 
\begin{equation}\label{eq:keyassumption}
    \text{the hermitian form } (,)_\bX \text{ is anisotropic.}
\end{equation}
Equivalently, $-d \notin \mathrm{Nm}_{F/F_0}(F^\times)$ where $d$ is the determinant of $(,)_\bX$. This is equivalent to the condition \eqref{eq:anisotropicassumption} in the introduction, see \eqref{eq:relationsonhermitianforms} below. Such a framing object $(\bX,\iota_\bX,\lambda_\bX)$ is unique up to isogenies which preserve the polarization up to a scalar in $\Oo_{F_0}^\times$.

\begin{comment}
Suppose $\pi_0=\epsilon p$ and let $\eta \in W$ such that $\eta \cdot \eta^\sigma=\epsilon^{-1}$. Define the $\sigma$-linear operator $\tau=\eta \Pi V^{-1}$ on $N$. Set $C=N^\tau$,  we obtain a $4$ dimensional $\Q_p$-vector space with an isomorphism
\[C\otimes_{\Q_p} W_\Q \simeq N.\]
Then for $x,y \in C$
\begin{align*}
    <x,y>=&<\tau x,\tau y>=<\eta \Pi V^{-1}x,\eta \Pi V^{-1}y> \\
    =& \eta^2 (-p \epsilon) <V^{-1}x,V^{-1}y> \\
    =& \eta^2 (-p \epsilon) p^{-1} <x,y>^\sigma \\
    =&-\frac{\eta}{\eta^\sigma} <x,y>^\sigma
\end{align*}
Choose $\delta\in\Z_{p^2}^{\times} \backslash \Z_p^{\times}$ such that $\delta^2\in \Z_p^\times$. Then the restriction of the form 
\begin{equation}
    <,>'=\eta^\sigma \delta <,> 
\end{equation}
to $C$ defines a skew symmetric form with values in $\Q_p$, and
\begin{equation}\label{definitionof(,)}
    (x,y)=<\Pi x,y>'+<x,y>'\pi
\end{equation}
defines a hermitian form on $C$ with
\[<x,y>'=\frac{1}{2} \mathrm{tr}_{F/F_0}(\pi^{-1}(x,y)).\]
\end{comment}

In fact, $\cN^\Pap$ can be descended to $\Spf \Oo_{\breve{F}_0}$ (see \cite[Section 8]{RSZ}). We denote by $\cN$ the corresponding moduli functor defined over $\Spf \Oo_{\breve{F}_0}$. In other words we have
\begin{equation}
   \cN^\Pap= \cN \times_{\Spf \Oo_{\breve{F}_0}} \SpfOF. 
\end{equation}
where $\times$ is the complete tensor product for formal schemes.
By assumption \eqref{eq:keyassumption} we have the following.
\begin{proposition}(\cite[Section 8]{RSZ})
$\cN$ has a unique reduced point on its special fiber.
As formal schemes,
\[\cN \cong \Spf \Oo_{\breve{F}_0} [[x,y]]/(x y-\pi_0).\]
In particular $\cN$ has semi-stable reduction and is regular.
\end{proposition}
\begin{remark}
In contrast, $\cN^\Pap$ is not regular. But it is normal and Cohen-Macaulay, see \cite[Theorem 4.5]{P}.
\end{remark}

Moreover by \cite[Section 8]{RSZ}, $\cN$ is isomorphic to a Lubin-Tate deformation space with Iwahori level structure. We recall the construction now.
In the following, we define every moduli functor over $\SpfOF$ as this is all we need.  
Let $\cM_{\Gamma_0(\pi)}$ represent the functor over $\SpfOF$ that associate each scheme $S\in \Nilp \Oo_{\breve{F}}$ the set of isomorphism classes of quadruples 
\[(Y,Y',\phi:Y \rightarrow Y',\rho_Y),\]
where $Y$ and $Y'$ are strict formal $\Oo_{F_0}$-modules over $S$ of $F_0$-height $2$ and dimension one, $\phi$ is an isogeny of degree $\q$ and $\rho_Y: Y\times_{S} \bar{S}\rightarrow \bY \times_{\Spec k} \bar{S}$ is a quasi-isogeny of formal $\Oo_{F_0}$-modules of height $0$. We can define a map $ \cM_{\Gamma_0(\pi)}\rightarrow \cN^\Pap$ as follows. Let $(Y,Y',\phi, \rho_Y)$ be in $ \cM_{\Gamma_0(\pi)}(S)$, set
\[X:=Y\times Y'.\]
We can define an $\Oo_F$-action on $X$ by 
\begin{equation}\label{eq:iota X in cN}
   \iota_X(\pi)=\left(\begin{array}{cc}
     & \phi' \\
    \phi & 
\end{array} \right), 
\end{equation}
where $\phi'$ is the dual isogeny of $\phi$ ($ \phi'\circ \phi=\pi_0$). Define the framing map
\begin{equation}\label{eq:rhoY'}
  \rho_{Y'}: Y'\times_{S} \bar{S}\overset{\phi^{-1}}{\longrightarrow} Y\times_{S} \bar{S} \overset{\rho_Y}{\longrightarrow} \bY \times_{\Spec k} \bar{S} \overset{\iota_{\bY}(\pi)}{\longrightarrow} \bY \times_{\Spec k} \bar{S}.  
\end{equation}
Then $\rho_Y^*(\lambda_{\bY}) $ and $\rho_{Y'}^*(\lambda_{\bY}) $ lift to principal polarization $\lambda_Y$ of $Y$ and $\lambda_{Y'}$ of $Y'$. Let 
\[\lambda_X=\left(\begin{array}{cc}
    \lambda_Y &  \\
     & \lambda_{Y'}
\end{array} \right).\]
Then we get the object $(X,\iota_X, \lambda_X,\rho_X)\in \cN^\Pap(S)$. This defines a map $ \cM_{\Gamma_0(\pi)}\rightarrow \cN^\Pap$ which is obviously functorial in $S$.
\begin{proposition}(\cite[Proposition 8.2]{RSZ})\label{prop:exceptionaliso}
 the map $ \cM_{\Gamma_0(\pi)}\rightarrow \cN^\Pap$ is an isomorphism.
\end{proposition}

Let $\cM$ be the Lubin-Tate moduli functor over $\SpfOF$ that associates to each scheme $S$ the set of isomorphism classes of pairs
\[(Y,\rho_Y)\]
where $Y$ is a strict formal $\Oo_{F_0}$-module of $F_0$-height $2$ and dimension $1$ over $S$ and $\rho_Y:Y\times_S \bar{S}\rightarrow \bY \times_{\Spec k} \bar{S}$ a quasi-isogeny of formal $\Oo_{F_0}$-modules of height $0$. Then it is well-known that (see for example \cite[Theorem 3.8]{viehmannKonstantin})
\begin{equation}
    \cM\cong \SpfOF[[T]].
\end{equation}
For later use we recall the following lemma from classical literature. 
\begin{lemma}\label{lem:modularequation}
$\cM_{\Gamma_0(\pi)}$ embeds into $\cM\times\cM$ as a divisor. 
In other words, there is a $f\in \Oo_{\breve{F}_0}[[X,Y]]$ such that
\[\cM_{\Gamma_0(\pi)}\cong \Spf \Oo_{\breve{F}}[[X,Y]]/(f(X,Y)).\]
For $i=1,2$, let $p_i$ be the morphism which is the composition of the embedding  $\cM_{\Gamma_0(\pi)}\rightarrow \cM\times \cM$ with the projection of $\cM\times \cM$ to its $i$-th factor.
Then $p_i$ is finite of degree $\q+1$.
\end{lemma}
\begin{proof}
This is well-known so we sketch its proof very quickly.
With the choice of framing maps $\rho_Y$ and $\rho_{Y'}$ as in \eqref{eq:rhoY'}, $\cM_{\Gamma_0(\pi)}$ is exactly the locus of $\cM\times\cM$ where the isogeny $\iota_\bY(\pi)$ lifts. The fact that $\cM_{\Gamma_0(\pi)}$ is a divisor is now an immediate consequence of Grothendieck-Messing deformation theory.

To prove the finiteness of the morphism $\cM_{\Gamma_0(\pi)}\rightarrow \cM$, notice that the morphism $\cM'\rightarrow \cM$ factor through $\cM_{\Gamma_0(\pi)}$ where $\cM'$ is the Lubin-Tate space with Drinfeld level structure:
\[\cM'(S)=\{(X,\rho,z)\mid (X,\rho)\in \cM(S), z\in X \text{ is a divisor of exact order } \q\}.\]
$\cM'$ is in turn a subfunctor of $\cM'$ whose $S$-point is 
\[\cM''(S)=\{(X,\rho,z)\mid (X,\rho)\in \cM(S), z\in X, [\pi_0](z)=0\}\]
where $[\pi_0]$ is the action of $\pi_0$ on $X$.
So it suffices to prove the finiteness of $\cM''\rightarrow \cM$. 
But this is a consequence of the Weierstrass preparation theorem for complete local rings and the fact that $[\pi_0]\times k$ can be represented by a polynomial.

Finally, to calculate the degree of the morphism $\cM_{\Gamma_0(\pi)}\rightarrow \cM$, it is enough to compute it on the generic fiber by studying the Tate module. The degree is then in fact $|\bP^1(\F_\q)|=\q+1$. 
\end{proof}
\begin{remark}\label{rmk:modularequation}
When $F_0=\Q_p$, the equation $f(X,Y)$ is exactly the classical modular polynomial such that
\[f(X,Y)\equiv (X-Y^p)(Y-X^p)\pmod{p}.\]
\end{remark}

\subsection{Dieudonn\'e module of the framing objects}\label{subsec:Dieudonnemodulecalculation}
We know that $\cM_{\Gamma_0(\pi)}$ has a unique $k$-point which is $(\bY,\bY, \iota_{\bY}(\pi),\mathrm{id}_\bY)$. Under the isomorphism  $ \cM_{\Gamma_0(\pi)}\rightarrow \cN^\Pap$, this point goes to $(\bX,\iota_\bX,\lambda_\bX,\mathrm{id}_\bX)\in \cN^\Pap$ where 
\begin{equation}\label{eq:definitionofbX}
    \bX=\bY\times \bY, \ \iota_\bX(\pi)=\left(\begin{array}{cc}
         &  \iota_\bY(\pi)\\
      \iota_\bY(\pi)   & 
    \end{array}\right),\ \lambda_\bX=\left(\begin{array}{cc}
      \lambda_\bY   &  \\
         & \lambda_\bY
    \end{array}\right).
\end{equation}
$(\bX,\iota_\bX,\lambda_\bX,\rho_\bX)$ is the unique reduced $k$-point of $\cN^\Pap$, we denote it by Sing. Here $(\bX,\iota_\bX,\lambda_\bX)$ is actually the framing object for $\cN^\Pap$. We describe its relative Dieudonn\'e module $\mathbb{M}$ explicitly.

First let us describe the relative Dieudonn\'e module $M(\bY)$ of $\bY$ (see \cite[Remark 2.15]{Shi1}). As an $\Oo_{\breve{F}_0}$ lattice, it is of rank 2. We can choose a basis $\{e,f\}$ such that 
$F_\bY e=f,F_\bY f=\pi_0 e,V_\bY e=f,V_\bY f=\pi_0 e$ and $\langle e,f\rangle_\bY=\delta$. With respect to this basis,
$\mathrm{End}(\bY)$ is of the form
\begin{equation}\label{eq:End bY is Oo_H}
    \left\{\left(\begin{array}{cc}
    a & b \pi_0  \\
    b^\sigma & a^\sigma
\end{array}\right) \mid a,b \in F_0[\delta]\cap \Oo_{\mathbb{H}}\right\}\cong \Oo_{\mathbb{H}}.
\end{equation}
One can check that under the pairing $\langle,\rangle_\bY$, the Rosati involution of an element in $\mathrm{End}(\bY)$ is given by the main involution of $\mathbb{H}$
\begin{equation}\label{rosatiinvolutionofbG}
    \left(\begin{array}{cc}
    a & b\pi_0  \\
    b^\sigma & a^\sigma
\end{array}\right)^*=
\left(\begin{array}{cc}
    a^\sigma & -b\pi_0  \\
    -b^\sigma & a
\end{array}\right).
\end{equation}
As in the proof of \cite[Proposition 8.2]{RSZ},
by changing basis using elements in $\mathbb{H}\cap \mathrm{SL}_2(F_0)$ we can assume $F$ and $V$ are of the same matrix form as before and 
\begin{equation}\label{eq:standardpiaction}
  \Pi=\left(\begin{array}{cc}
    0 & \pi_0  \\
    1 & 0
\end{array}\right).  
\end{equation}
Thus $\tau$ is the diagonal matrix $\mathrm{Diag}\{1,1\}$. It fixes the $F_0$-vector space $\mathrm{span}_{F_0}\{ e,f\}$. Then 
\begin{equation}\label{eq:(e,e)}
    (e,e)_\bY=-\delta^2 
\end{equation}
As $\Oo_F$ is a DVR and $M(\bY)\otimes_\Z \Q$ is a one dimensional $F$-space, there is a unique self-dual $\Oo_F$-lattice w.r.t. $(,)_\bY$, i.e., $\mathrm{span}\{e\}$. Let $\rho_{\bY}$ be the identity of $\bY$, then $(\bY,\iota_\bY, \lambda_\bY,\rho_\bY)$ is the unique closed point of $\cN_{(1,0)}(k)$.

Now let $\bX$ be defined in \eqref{eq:definitionofbX}. We denote the first copy of $\bY$ as $\bY_1$ and the second copy as $\bY_2$.
In terms of the Dieudonn\'e module $\bM=M(\bX)$ we assume that $\{e_1, f_1\}$ is a basis of $M(\bY_1)$  and $\{e_2, f_2\}$ is a basis of $M(\bY_2)$ such that 
\[F_\bX e_i=f_i,F_\bX f_i=\pi_0 e_i, V_\bX e_i=f_i,V_\bX f_i=\pi_0 e_i,\langle e_i,f_j\rangle_\bX=\delta \delta_{ij},\  i,j=1,2.\]
Using \eqref{eq:iota X in cN} and \eqref{eq:standardpiaction}, we can see that
\begin{equation}\label{eq:Pimatrixform}
    \Pi(e_1)=f_2, \ \Pi(f_1)=\pi_0 e_2,\ 
    \Pi(e_2)= f_1, \ \Pi(f_2)=\pi_0 e_1.
\end{equation}
It is then easy to see that $\Pi^2=\pi_0$ and $\langle\Pi x,y\rangle_{\bX}=\langle x,-\Pi y\rangle_{\bX} $ for $x,y\in M(X)$.
Define
\begin{equation}\label{eq:v_1 v_2}
   v_1:= e_1+e_2, \ v_2:=\delta (e_1-e_2). 
\end{equation}
Then we have
\[\tau(v_1)=v_1, \tau(v_2)=v_2,\]
where $\tau=\Pi V_\bX^{-1}$ and
\begin{align}\label{eq:Gram matrix of v_1 v_2}
    (v_1,v_1)_\bX=&\delta \langle\Pi v_1,v_1\rangle_\bX+\pi \delta\langle v_1,v_1\rangle_\bX
    = -2\delta^2 , \\ \nonumber
    (v_1,v_2)_\bX=& \delta \langle\Pi v_1,v_2\rangle_\bX+\pi \delta\langle v_1,v_2\rangle_\bX
    = 0, \\ \nonumber
    (v_2,v_2)_\bX=& \delta \langle\Pi v_2,v_2\rangle_\bX+\pi \delta\langle v_2,v_2\rangle_\bX
    = 2\delta^4 .
\end{align}

\begin{comment}
And we have 
\begin{align*}
    \Pi(e_1+e_2)=f_1+f_2, &\  \Pi(f_1+f_2)=\pi_0 (e_1+e_2),\\ 
    \Pi(e_1-e_2)=-(f_1-f_2), &\  \Pi(f_1-f_2)=-\pi_0 (e_1-e_2).
\end{align*}
So the hermitian form of $C$ with respect to the basis $\{v_1,v_2\}$ is represented by the Gram matrix
\[\left(\begin{array}{cc}
   -2\delta^2   & 0 \\
    0 & 2\delta^4 
\end{array}\right).\]
Recall that a two dimensional hermitian form over $F_0$ with determinant $d$ is anisotropic if and only if $(-d)\notin \mathrm{Nm}_{F/F_0}(F^\times)$. Hence the above hermitian form is anisotropic since it has determinant $-\delta^2( 2\delta^2)^2 $ and $\delta^2 \notin \mathrm{Nm}_{F/F_0}(F^\times)$.

For later use, we define an automorphism $s_\delta$ of $(\bX,\iota_\bX, \lambda_\bX)$  by 
\begin{equation}\label{s_delta}
   s_\delta(x,y)=(\delta x, -\delta y),
\end{equation}
where $x,y\in \bY(R)$, $R$ is a $k$-algebra.
We need to check that $s_\delta \circ \iota_\bX(\pi)=\iota_\bX(\pi)\circ s_\delta$. In fact 
\begin{align*}
    s_\delta \circ \iota_\bX(\pi)(x,y)=& s_\delta(\pi y,\pi x)
    = (\delta \pi y,-\delta \pi x),\\
    \iota_\bX(\pi)\circ s_\delta(x,y)=& \iota_\bX(\pi) (\delta x, -\delta y)
    = (-\pi \delta y,\pi \delta x).
\end{align*}
Moreover, the induced action of $s_\delta$ on $\bM$ multiply the bilinear form $\langle,\rangle_\bX$ by $-\delta ^2$. 

\[s_\delta (e_1)=\delta e_1, s_\delta (f_1)=-\delta f_1,s_\delta (e_2)=-\delta e_2,s_\delta (f_2)=\delta f_2,\]
\[s_\delta (v_1)=v_2, s_\delta (v_2)=\delta^2 v_1.\]
\end{comment}

\subsection{Kr\"amer model}\label{subsec:Kramermodel}
We need some information on  the local model $N^\Kra$ of $\cN^\Kra$ in the sense of \cite[Definition 3.27]{RZ}. Again we assume $n=2$ and  \eqref{eq:anisotropicassumption}.
For the complete definition, see Appendix \ref{sec:Kramer model is blow up}.
For now we just need to know that $N^\Kra$ is representable by a regular scheme which is flat over $\SpfOF$ and can be covered by affine charts $U_i$ charts $U_i=\Spec R_i$ ($i=1,2$) where
\begin{equation}\label{eq:Kramer model equation}
    R_i=\Oo_{\breve F}[x_i^1, x_i^2, y_i]/(y_i((x_i^1)^2+(x_i^2)^2)-2\pi, x^i_i=1).
\end{equation}
The following transition relations between coordinate functions hold in $U_i\cap U_j$ where $x_i^j$ and $x_j^i$ are units.
\begin{equation}\label{eq:coordinate change of Kramer model}
    x_i^\ell=x_i^j x_j^\ell,\ y_i=(x_j^i)^2 y_j.
\end{equation}
The special fiber of $N^\Kra$ consists of two divisors $Z_1$ and $Z_2$ where $Z_1 \cong \bP^{1}/k$ is the exceptional divisor of $N^\Kra$ defined by the equation $y_i=0$ in $U_i$:
\begin{equation}\label{eq:Z1}
    Z_1\cap U_i=\Spec (k[x_i^1, x_i^2]/(x_i^i-1)).
\end{equation}
And $Z_2$ is defined by the equation $\sum_{j=1}^2 (x_i^j)^2=0$ in $U_i$:
\begin{equation}\label{eq:Z2}
    Z_2 \cap U_i=\Spec (k[x_i^1,x_i^2,y_i]/(x_i^i-1,(x_i^1)^2+(x_i^2)^2).
\end{equation}

\begin{lemma}\label{lem:Z1 intersect Z2}
$Z_1$ and $Z_2$ intersect transversely at the two singular points of the special fiber $N^\Kra_s$ of $N^\Kra$ and the intersection number is $2$.
\end{lemma}
\begin{proof}
From \eqref{eq:Z1} and \eqref{eq:Z2}, we see that the intersection of $Z_1$ and $Z_2$ in $U_i$ ($1\leq i \leq 2$) is 
\[\Spec (k[x_i^1, x_i^2]/(x_i^i-1,(x_i^1)^2+(x_i^2)^2).\]
Hence in the homogeneous coordinate of $\bP^1/k$, the intersection locus is the union of two points $[(1,\sqrt{-1})]$ and $[(-1,\sqrt{-1})]$, each with multiplicity one. Notice that the two points are exactly the singular locus of $N^\Kra_s$.
\end{proof}

\begin{lemma}\label{lem:cNKra is completion of local model}
There is an isomorphism of formal schemes
\[s:\cN^\Kra \rightarrow (N^\Kra)_{Z_1},\]
where $(N^\Kra)_{Z_1}$ is the formal completion of $N^\Kra$ along $Z_1$.
\end{lemma}
\begin{proof}
By \cite[Proposition 3.33]{RZ}, there is a formally \'etale morphism
\[s_0:\mathcal{U}\rightarrow (N^\Pap)_\pi\]
where $\mathcal{U}$ is an \'etale neighbourhood of $\cN^\Pap$, $N^\Pap$ is the local model of $\cN^\Pap$ (whose definition will be recalled in Appendix \ref{sec:Kramer model is blow up}) and $(N^\Pap)_\pi$ is its $\pi$-adic completion. Since $\cN^\Pap$ is strictly Henselian and local (this for example follows from Proposition \ref{prop:exceptionaliso}), $\mathcal{U}=\cN^\Pap$.
By the proof of \cite[Corollary 4.13]{P}, we have
\[N^\Pap=\Oo_{\breve F}[x_1,x_2]/(x_1^2+x_2^2-\pi_0)\]
and the image of Sing under $s_0$ is the point $pt$ in the special fiber of $N^\Pap$ defined by $x_1=x_2=\pi=0$. Again since $\cN^\Pap$ is strictly Henselian and local, $s_0$ factor through 
\[s_0':\cN^\Pap\rightarrow (N^\Pap)_{pt}\]
which is again formally \'etale. Here $(N^\Pap)_{pt}$ is the completion of $(N^\Pap)_\pi$ along $pt$. Since both structure rings are local and strictly Henselian, the morphism $s'_0$ is in fact an isomorphism. The lemma now follows from Proposition \ref{prop: Kra is blow up of Pappas} and its analogue for local models, namely Lemma \ref{lem:local model isomorphism}.
\end{proof}

We denote the unique exceptional divisor of $\cN^\Kra$ by $\Exc$. By Proposition \ref{prop:Kramermodel}, $\iota_\bX(\pi)$ acts trivially on $\Lie \bX$. Hence we can choose $\cF_\bX$ to be any line in $\Lie \bX$ and $\Exc=\bP^1(\Lie\bX)$.
Under this isomorphism of Lemma \ref{lem:cNKra is completion of local model}, we have, 
\begin{equation}
    s^{-1}(Z_1)=\Exc=\bP^1(\Lie\bX)\cong \bP^1(V\bM/\pi_0\bM),
\end{equation}
where the last isomorphism  sends a line $L$ in $\Lie \bX$ to its perpendicular line $L^\vee\subset V\bM/\pi_0\bM$ under the pairing $\langle,\rangle_\bX$ described in Section \ref{subsec:Dieudonnemodulecalculation}.

We now describe $s^{-1} (Z_1\cap Z_2)$, namely the two singular points of the special fiber of $\cN^\Kra$. Recall that $V_\bX\bM=\Pi \bM=\mathrm{span}\{f_1,f_2\}$ in the coordinates introduced in  Section \ref{subsec:Dieudonnemodulecalculation}. The quadratic form $(x_i^1)^2+(x_i^2)^2$ that shows up in equation \eqref{eq:Kramer model equation} corresponds to the quadratic form (see \cite[page 9]{Kr})
\[\{\Pi v,\Pi w\}:=\langle\Pi v,  w\rangle_{\bX},\]
for $\Pi v,\Pi w \in V \bM$. Now by equation \eqref{eq:Pimatrixform} we see that
\[\{f_1,f_1\}=\langle e_2,f_1\rangle_\bX=0.\]
Similarly $\{f_2,f_2\}=0$. Hence $s^{-1} (Z_1\cap Z_2)$ correspond to the two lines spanned by $f_1$ and $f_2$ respectively in $V\bM/\pi_0\bM$. Taking their perpendicular lines in $\Lie\bX= \bM/V\bM$, we have shown that
\begin{lemma}\label{lem:intersectionofZ1Z2}
$\Exc\cap s^{-1}(Z_2)$ are the two points $P_1=\mathrm{span}_k\{e_1\}$ and $P_2=\mathrm{span}_k\{e_2\}$ in $\bP^1(\Lie\bX)$.
\end{lemma}

The reduced locus of $\cN^\Kra$ is its exceptional divisor. In particular it is proper over $\Spec k$.
For two divisors $\cZ_1$ and $\cZ_2$ of $\cN^\Kra$ such that the support of their intersection is contained in the reduced locus, one can define their intersection number to be 
\begin{equation}
   \cZ_1 \cdot \cZ_2= \chi(\cN^\Kra, \Oo_{\cZ_1}\otimes^{\mathbb{L}} \Oo_{\cZ_2}). 
\end{equation}
where $\otimes^{\mathbb L}$ is the derived tensor product of coherent sheaves on $\cN^\Kra$ and $\chi$ takes the Euler characteristics. It is easy to compute the self intersection number of $\Exc$.
\begin{lemma}\label{lem:Excselfintersection}
\[\Exc\cdot \Exc=-2.\]
\end{lemma}
\begin{proof}
By Lemma \ref{lem:cNKra is completion of local model}, it suffices to compute $Z_1\cdot Z_1$. Since $N^\Kra$ is regular and flat over $\Spf \Oo_{\breve F}$,
Proposition 1.21 of \cite[Chapter 9]{L} implies that the intersection number of $Z_1$ with the special fiber $N^\Kra_s$ of $N^\Kra$ is $0$. But 
\[N^\Kra_s=Z_1+Z_2.\]
The lemma now follows from Lemma \ref{lem:Z1 intersect Z2}.
\end{proof}

\section{Quasi-canonical lifting divisors}\label{sec:quasicanonicalliftingdivisor}
In this section we define sub formal schemes of $\cN^\Pap$ and $\cN^\Kra$ based on the theory of quasi-canonical lifts of Gross (\cite{G}). Our main reference is \cite{W}. For simplicity we often use $\otimes k$ to denote the base change $\times_{\SpfOF} \Spec k$. 
\subsection{Quasi-canonical lifts on $\cM$}\label{subsec:quasicanonicallifts}
Let $\kappa:\Oo_F\hookrightarrow \End(\bY)=\Oo_{\mathbb{H}}$ be an embedding. Via this embedding we obtain a map $\Oo_F\rightarrow \End(\Lie \bY)=k$, which extends the canonical residue map $\Oo_{F_0}\rightarrow k$.
\begin{comment}
determined by
\begin{equation}\label{kappaembedding}
   \kappa(a+b \pi)=\left(\begin{array}{cc}
    a & b \pi_0 \\
    b & a
\end{array}\right), a,b\in \Oo_{F_0} 
\end{equation}
where we have identified $\Oo_\H$ with $\End(\bY)$ via \eqref{eq:End bY is Oo_H}.
\end{comment}

Fix an $s\geq 0$. Let 
\begin{equation}\label{eq:Oo_s}
  \Oo_s=\Oo_{F_0}+\Oo_F\cdot \pi_0^s.  
\end{equation}
Let  $M_s$ be the ring class field of $\Oo_s^\times$, i.e., \ the finite abelian extension of $\breve{F}$ corresponding to the subgroup $\Oo_s^\times$ of $\Oo_F^\times$ under the reciprocity isomorphism of Lubin-Tate theory:
\begin{equation}\label{eq:Artinmap}
   \gamma:\mathrm{Gal}(\bar{\breve{F}}/\breve{F})^{ab}\rightarrow \Oo_F^\times.
\end{equation}
Let $W_s$ be the ring of integers of $M_s$.
Notice that in this notation $M_0=\breve{F}$, $W_0=\Oo_{\breve F}$. We have 
\[[W_s: \Oo_{\breve{F}}]=[\Oo_F^\times:\Oo_s^\times]=\q^s.\]
By its definition, $W_s/\Oo_{\breve{F}_0}$ is totally ramified and integrally closed. 

\begin{definition}(\cite[Definition 3.1]{W})\label{def:quasi canonical lifts}
A quasi-canonical lift of level $s$ with respect to the embedding $\kappa:\Oo_F\hookrightarrow \Oo_{\mathbb{H}}$ is a triple $(\cG',\rho,\iota_{\cG'})$. Here $\cG'$ is a lift of $\bY$ defined over some finite extension $A'/ \Oo_{\breve{F}}$. Meanwhile $\rho: \cG'\times \Spec k \overset{\sim}{\longrightarrow} \bY$ is an isomorphism of strict formal $\Oo_{F_0}$-modules and $\iota_{\cG'}:\Oo_s \overset{\sim}{\longrightarrow} \End(\cG')$ is an  $\Oo_{F_0}$-algebra isomorphism. We require the following conditions. 
\begin{enumerate}
    \item The composition of $\iota_{\cG'}$ together with the natural embedding $\End(\cG')\rightarrow \End (\Lie \cG')=A'$ is the canonical inclusion $\Oo_s\rightarrow A'$.\
    \item The composition of $\iota_{\cG'}$ with the inclusion $\End(\cG')\hookrightarrow \End(\bY)=\Oo_{\mathbb{H}}$ induced by $\rho$ is equal to $\kappa|_{\Oo_s}$.
\end{enumerate}
We call a quasi-canonical lift of level $0$ a canonical lift.
\end{definition}
\begin{remark}
To ease notation, we often omit $\iota_{\cG'}$ from the triple $(\cG',\rho,\iota_{\cG'})$. We also say $\cG'$ is a quasi-canonical lift without mentioning the framing map $\rho$ if condition (1) above is satisfied.
\end{remark}
\begin{remark}\label{rmk: cG independent of kappa}
Let $\kappa,\kappa'$ be two embeddings $\Oo_F\rightarrow \Oo_{\mathbb H}$. By Skolem-Noether Theorem, there is an element $a\in \Oo^\times_{\mathbb H}$ such that $\kappa'=a\kappa a^{-1}$. If $(\cG',\rho)$ is a quasi-canonical lift of level $s$ with respect to $\kappa$, then $(\cG',\iota_\bY(a)\circ \rho)$ is a quasi-canonical lift of level $s$ with respect to $\kappa'$.
\end{remark}

Canonical lift can be defined over $\Oo_{\breve F}$ and is unique up to isomorphism, we fix such a choice and denote it by $\cG$. By Remark \ref{rmk: cG independent of kappa} we can assume $\cG\otimes k=\bY$ and $\iota_\cG\otimes k=\iota_\bY$.
More generally quasi-canonical lifts of level $s$ always exists. Their minimal ring of definition is $W_s$, see Theorem 3.2 and Corollary 4.7 of \cite{W}. A quasi-canonical lift $(\cG_s,\rho)$ with respect to $\kappa$ defines a deformation of $(\bY,\rho_\bY)$ to $\Spf W_s$. Hence we obtain a morphism
\[\Spf W_s\rightarrow \cM.\]
\begin{lemma}\label{lem: definition of cW_s}
The morphism $\Spf W_s\rightarrow \cM$ is a closed embedding and defines a Cariter divisor $\cW_s^\kappa$ of $\cM$. Moreover the resulting divisor $\cW_s^\kappa$ is independent of the choice of the quasi-canonical lift $(\cG_s,\rho)$.
\end{lemma}
\begin{proof}
Since $\cM\cong \Spf\Oo_{\breve{F}}[[X]]$, the first claim follows from the fact that $X$ pullbacks to a uniformizer of $W_s$, which in turn follows from \cite[Corollary 4.7]{W}. 

Assume that $\cW_s^\kappa$ is defined by the degree $\q^s$ monic polynomial $g_s(X)$.
By \cite[Theorem 3.2]{W}, the set of quasi-canonical lifts of level $s$ is acted simply transitively on by the Galois group $\mathrm{Gal}(W_s/\Oo_{\breve F})$ and can be identified with 
\begin{equation}\label{eq:g_s(X)}
    \Hom_{\Oo_{\breve{F}}}(\Oo_{\breve{F}}[[X]]/(g_s(X)),W_s),
\end{equation}
or in other words the set of roots of $g_s(X)$ in $W_s$. This proves the second claim of the lemma.
\end{proof}
\begin{remark}\label{rmk: root of g_s generates W_s}
by \cite[Corollary 4.7]{W}, we know that each root $x_s$ of $g_s(X)$ is a uniformizer of $W_s$. In particular, $W_s=\Oo_{\breve F}[x_s]$.  
\end{remark}

\subsection{Quasi-canonical lifting divisors on $\cN^\Pap$}\label{subsec:quasi canonical lifting on NPap} 
Throughout the rest of the paper we use the notation
\begin{equation}\label{eq:m(s)}
    m(s)=\left\{\begin{array}{cc}
    s-1 & \text{ if } s\geq 1,\\
    0 & \text{ if } s=0.
\end{array}\right.
\end{equation}
\begin{proposition}\label{prop:quasi canonical divisor on NPap}
For all $s\geq 0$ there exist quasi-canonical lifts $(\cG_s,\rho_s)$ of level $s$ with respect to $\kappa$ together with a morphism $\beta_s:\cG_{m(s)}\rightarrow \cG_s$ such that the following diagrams commute
\begin{equation}\label{eq:cXs definition diagram}
\begin{tikzcd}
\cG_{0}\otimes k\arrow{r}{\beta_0\otimes k} \arrow{d}{\rho_0} & \cG_{0}\otimes k \arrow{d}{\rho'_0}\\
\bY \arrow{r}{\iota_\bY(\pi)} & \bY
\end{tikzcd},
\begin{tikzcd}
\cG_{m(s)}\otimes k\arrow{r}{\beta_s\otimes k} \arrow{d}{\rho_{m(s)}} & \cG_{s}\otimes k \arrow{d}{\rho_s}\\
\bY \arrow{r}{\iota_\bY(\pi)} & \bY
\end{tikzcd}
\end{equation}
where $s\geq 1$ in the second diagram.
These define Weil divisors $\cY_{s,+}^\kappa$ of $\cN^\Pap$ isomorphic to $\Spf W_s$. The morphism 
\[\cY_{s,+}^\kappa\rightarrow \cN^\Pap\overset{p_2}{\rightarrow} \cM\]
induces an isomorphism from $\cY_{s,+}^\kappa$ to its image $\cW_s^\kappa$. 
\end{proposition}
\begin{proof}
We proceed by induction on $s$. 
First let $(\cG_0,\rho_0)=(\cG,\rho_\cG)$ be a canonical lift with respect to $\kappa$ and $\beta_0=\iota_{\cG}(\pi)$. Then we can define $\rho'_0$ by \eqref{eq:rhoY'} such that the first diagram in \eqref{eq:cXs definition diagram} commutes.
Recall that by  Lemma \ref{lem:modularequation}, we have $\cM_{\Gamma_0(\pi)}=\Spf R$, where
\begin{equation}\label{eq:coordinateringofcN}
    R:=\Oo_{\breve{F}}[[X,Y]]/(f(X,Y)).
\end{equation}
Let $x_0\in W_0$ be the unique root of the polynomial $g_0(X)$ where $g_0(X)$ is as in \eqref{eq:g_s(X)}. Now assume inductively that we have defined $x_s\in W_s$ for $s\geq 0$ and that $x_s$ is a root of $g_s(X)$. In particular $W_s=\Oo_{\breve{F}}[x_s]$ and the isomorphism
\[\Oo_{\breve{F}}[[X]]/(g_s(X))\rightarrow W_s, \ X\mapsto x_s\]
corresponds to a quasi-canonical lift $(\cG_s,\rho_s)$ of level $s$ (see the proof of Lemma \ref{lem: definition of cW_s}). Consider the equation 
\begin{equation}\label{eq:modularequation}
    f(x_s,Y)=0.
\end{equation}
By the moduli interpretation of $\cM_{\Gamma_0(\pi)}$, a solution $y$ to the above equation corresponds to a strict formal $\Oo_{F_0}$-module $\cG'$ over $W_s[y]$ together with a degree $\q$ isogeny 
\[\beta':\cG_s\rightarrow \cG'.\]
By the proof of \cite[Proposition 4.6]{W}, $\cG'$ is a quasi-canonical lift of level either $s+1$ or $m(s)$ and we can choose $y$ such that $\cG'$ is of level $s+1$. We define $x_{s+1}$ to be such a $y$ and define $\cG_{s+1}:=\cG'$ and $\beta_{s+1}:=\beta'$. Equation \eqref{eq:rhoY'} then gives us a framing map $\rho_{s+1}:\cG_{s+1}\otimes k\rightarrow \bY$ defined by 
\begin{equation}\label{eq:normalizationofbetas}
  \rho_{s+1}=(\beta_s\otimes k)^{-1} \circ \rho_{s} \circ \iota_\bY(\pi).
\end{equation}
This finishes the definition of quasi-canonical lifts $(\cG_s,\rho_s)$ inductively. The commutativity of the second diagram in \eqref{eq:cXs definition diagram} follows from \eqref{eq:normalizationofbetas}.

Define an $\Oo_{\breve{F}}$-algebra homomorphism 
\begin{equation}\label{eq:varphi_s sharp}
    \varphi_s^\sharp: \Oo_{\breve{F}}[[X,Y]]/(f(X,Y))\rightarrow W_s, X\mapsto x_{m(s)}, Y \mapsto x_{s}.
\end{equation}
$\varphi_s^\sharp$ is surjective as $x_s$ is a generator of $W_s$ for $s\geq 0$. Hence it defines a closed sub formal scheme $\cY_{s,+}^\kappa\cong \Spf W_s$ of $\cN^\Pap$. Since $\cY_{s,+}^\kappa$ has codimension $1$ and $W_s$ is an integral domain, it is a Weil divisor.
\end{proof}

Using Tate modules, we can see that there is a unique morphism 
\begin{equation}\label{gamma_s}
    \gamma_s:\cG_s\rightarrow \cG_{m(s)}
\end{equation}
such that 
\[\beta_s\circ \gamma_s=\iota_{\cG_s}(\pi_0).\]
Then it is automatically true that 
\[\gamma_s \circ \beta_s=\iota_{\cG_{m(s)}}(\pi_0),\]
and 
\begin{equation}\label{eq:gamma_s commutative diagram}
  \rho_{m(s)}\circ (\gamma_s\otimes k) \circ \rho_{s}^{-1}= \iota_\bY (\pi).  
\end{equation}
Let $\lambda_s:\cG_s\rightarrow \cG^\vee_s$ be the unique principal polarization that lifts $\bY \rightarrow \bY^\vee$ corresponding to the bilinear form $\langle,\rangle_\bY$ on $M(\bY)$.
According to the dictionary translating between objects in $\cM_{\Gamma_0(\pi)}$ and $\cN^\Pap$ discussed before Proposition \ref{prop:exceptionaliso},
the pullback of the universal object of $\cN^\Pap$ to $\cY_{s,+}^\kappa$ is 
$(X_s^+,\iota_s^+,\lambda_s^+,\rho_s^+)$ where
\begin{equation}\label{eq:Xs+}
    X_s^+=\cG_{m(s)}\times \cG_{s},\ \iota_s^+(\pi)=\left(\begin{array}{cc}
     & \gamma_s \\
    \beta_s & 
\end{array}\right), \lambda_s^+=\lambda_{m(s)}\times \lambda_{s},
\end{equation}
and
\begin{equation}\label{eq:rho_s^+}
    \rho_s^+=\begin{cases}
    \rho_{m(s)}\times \rho_{s} & \text{ if } s>0,\\
    \rho_0\times \rho_0' & \text{ if } s=0.
    \end{cases}
\end{equation}
Similarly for $s\geq 1$ we can define a closed sub formal scheme $\cY_{s,-}^\kappa\cong \Spf W_s $ of $\cN^\Pap$ such that the pullback of the universal object of $\cN^\Pap$ to $\cY_{s,-}^\kappa$ is 
$(X_s^-,\iota_s^-,\lambda_s^-,\rho_s^-)$ where
\begin{equation}\label{eq:Xs-}
    X_s^-=\cG_s\times \cG_{m(s)},\ \iota_s^-(\pi)=\left(\begin{array}{cc}
     & \beta_s \\
    \gamma_s & 
\end{array}\right), \lambda_s^-=\lambda_s\times \lambda_{m(s)},\ \rho_s^-=\rho_{s}\times \rho_{m(s)}.
\end{equation}

\begin{remark}
When $s=0$ we denote $\cY_{0,+}^\kappa$ by $\cY_{0}^\kappa$ as well. We also define $\cY_{0,-}^\kappa$ to be $\cY_{0}^\kappa$ for uniformality of statements.
\end{remark}

Let $\cW_s^\kappa$ be as in Lemma \ref{lem: definition of cW_s}. Then $p_i^*(\cW_s^\kappa)$ ($i=1,2$) is a Cartier (hence also Weil) divisor on $\cN^\Pap$.
\begin{lemma}\label{lem:piW and cY}
For all $s\geq 0$ we have the following equality of Weil divisors on $\cN^\Pap$.
\[p_1^*(\cW_s^\kappa)=\cY_{s+1,+}^\kappa +\cY_{s,-}^\kappa,\ p_2^*(\cW_s^\kappa)=\cY_{s+1,-}^\kappa +\cY_{s,+}^\kappa.\]
\end{lemma}
\begin{proof}
We prove the first equation.
By the moduli interpretation in Proposition \ref{prop:quasi canonical divisor on NPap}, both $\cY_{s+1,+}^\kappa$ and $\cY_{s,-}^\kappa$ are sub divisors of $p_1^*(\cW_s^\kappa)$. Since $p_1$ is finite of degree $\q+1$ and $\cW_s^\kappa$ has degree $\q^s$ over $\Oo_{\breve{F}}$, we know that $p_1^*(\cW_s^\kappa)$ has degree $\q^s(\q+1)$ over $\Oo_{\breve{F}}$. On the other hand $\cY_{s+1,+}^\kappa$ and $\cY_{s,-}^\kappa$ are isomorphic to $\Spf W_{s+1}$ and $\Spf W_s$ respectively. By counting the degrees, we see that the inclusion of $\cY_{s+1,+}^\kappa +\cY_{s,-}^\kappa$ into $p_1^*(\cW_s^\kappa)$ is in fact an equality.
\end{proof}

\begin{corollary}\label{cor: Gal acts transitively}
The Galois group $\mathrm{Gal}(W_s/\Oo_{\breve{F}})$ acts simply transitively on diagrams which satisfy the condition stated in Proposition \ref{prop:quasi canonical divisor on NPap}. In particular, the divisor $\cY_{s,+}^\kappa$ is independent of the choice of such a diagram. Similar statements hold for $\cY_{s,-}^\kappa$.
\end{corollary}
\begin{proof}
Let $x_s$ be the root of $g_s(X)$ corresponding to $(\cG_s,\rho_s)$.
Consider the roots of the following degree $\q+1$ equation
\[f(X,x_s).\]
By  Lemma \ref{lem:piW and cY}, there are exactly $\q$ roots which correspond to quasi-canonical lifts of level $s+1$ and there is a unique root $x_{m(s)}$ which  corresponds to a quasi-canonical lift of level $m(s)$. Since $\mathrm{Gal}(W_s/\Oo_{\breve{F}})$ acts simply transitively on the roots of $g_s(X)$, it also acts simply transitively on such pairs $(x_{m(s)},x_s)$. But such a pair exactly corresponds to a diagram in Proposition \ref{prop:quasi canonical divisor on NPap}.
\end{proof}

\subsection{Quasi-canonical lifting divisors on $\cN^\Kra$}
For a blow-up $\tilde{X}\rightarrow X$ and a morphism $Y\rightarrow X$, the strict transform of $Y$ under $\tilde{X}\rightarrow X$ is defined in the definition below \cite[Corollary 7.15]{hartshorne2013algebraic}.
\begin{definition}\label{def:quasicanonicaldivisoronKramer}
Define $\cZ_{s,+}^\kappa$ (resp. $\cZ_{s,-}^\kappa$) to be the strict transform of $\cY_{s,+}^\kappa$ (resp. $\cY_{s,-}^\kappa$) under the blow-up $\Phi:\cN^\Kra\rightarrow \cN^\Pap$. Notice that $\cZ_{0,+}^\kappa=\cZ_{0,-}^\kappa$, so we also denote it by $\cZ_0^\kappa$.
\end{definition}

\begin{proposition}\label{prop:blowupquasicanonicaldivisor}
$\cZ_{s,+}^\kappa$ is an irreducible Cartier divisor of $\cN^\Kra$.
The blow-up $\Phi$ induces an isomorphism 
\[\cZ_{s,+}^\kappa \overset{\sim}{\rightarrow} \cY_{s,+}^\kappa.\]
In particular $\cZ_{s,+}^\kappa\cong \Spf W_s$.
Similar statements hold for $\cZ_{s,-}^\kappa$.
\end{proposition}
\begin{proof}
As in the proof of Proposition \ref{prop:quasi canonical divisor on NPap}, we identify $\Oo_{\breve F}[Y]/(g_s(Y))$ with $W_s$ by the map 
\[\Oo_{\breve F}[Y]/(g_s(Y))\rightarrow W_s, Y \mapsto x_s\]
where $x_s$ is a root of $g_s(Y)$.
Then the closed embedding $\cY_{s,+}^\kappa\rightarrow\cN^\Pap$ corresponds to the $\Oo_{\breve F}$-algebra homomorphism 
\[\varphi_{s}^\sharp:\Oo_{\breve F}[[X,Y]]/(f(X,Y))\rightarrow W_s, X\mapsto x_{m(s)},Y \mapsto x_s\]
where $x_{m(s)}$ is as in the proof of Corollary \ref{cor: Gal acts transitively}.
$\cN^\Kra$ is the blow-up of $\cN^\Pap$ along the ideal $I=(\pi,X,Y)$. Hence $\cZ_{s,+}^\kappa$ is isomorphic to the blow-up of $\Spf W_s$ along the ideal generated by $\varphi_{s}^\sharp(I)$ in $W_s$.
Recall that any blow-up along a principal ideal is an isomorphism and $W_s$ is a P.I.D.. This proves the first claim of the Proposition. In particular, $\cZ_{s,+}^\kappa$ is an irreducible Weil divisor of $\cN^\Kra$. As $\cN^\Kra$ is regular, $\cZ_{s,+}^\kappa$ is also an irreducible Cartier divisor.
\end{proof}

\begin{proposition}\label{intersectionnumber1withExc}
For all $s\geq 0$ we have
\[\cZ_{s,-}^\kappa \cdot \Exc=\cZ_{s,+}^\kappa \cdot \Exc=1.\]
\end{proposition}
\begin{proof}
$\Exc$ is exactly the supersingular locus of $\cN^\Kra$. To be more precise,
if $X_1\times X_2$ is the universal $\Oo_{F_0}$-module of $\cN^\Kra$ under the identification $\cM_{\Gamma_0(\pi)}\cong\cN^\Pap$, then the supersingular locus is where $X_1=X_2=\bY$, or equivalently the locuas where endomorphism ring of both $X_1$ and $X_2$ is $\Oo_\mathbb{H}$. By applying \cite[Theorem 2.1]{V} to any element in $\Oo^\times_\mathbb{H}\setminus \Oo_{F}$,
the supersingular locus of $\cZ_{s,\pm}^\kappa$ is exactly their reduced locus which is $\Spec k$. This implies the proposition.
\end{proof}

\begin{proposition}\label{prop:cZintersectwithspecialfiber}
For $s\geq 1$, the divisor $\cZ_{s,-}^\kappa$ intersects the special fiber $\cN^\Kra_s$ of $\cN^\Kra$ only at $P_1$ while $\cZ_{s,+}^\kappa$ intersects $\cN^\Kra_s$ only at $P_2$. Here $P_1$ and $P_2$ are the only singular points of $\cN^\Kra_s$ described in Lemma \ref{lem:intersectionofZ1Z2}.  
\end{proposition}
\begin{proof}
We need to use the concepts of formal modulus  of quasi-canonical lifts, see \cite[Section 4]{W}. First we deal with $\cZ_{s,-}^\kappa$.
We know by \cite[Proposition 4.6]{W} that the formal modulus of $\cG_s$ is 
\[v(\cG_s)=\frac{1}{2\q^s}.\]
Recall that we have a degree $\q$ isogeny $\gamma_s: \cG_{s}\rightarrow \cG_{s-1}$. In terms of power series, it can be represented by (see \cite[equation (4.1)]{W}):
\[\gamma_s(X)=c_s X+\text{ higher order terms}.\]
Define $v(\Lie(\gamma_s)):=\mathrm{val}_{\pi_0} (c_s)$. This is independent of the choice of the parameter $X$. Recall that the dual isogeny $\beta_s:\cG_{s-1}\rightarrow \cG_s $ satisfies $\beta_s \circ \gamma_s=\iota_{\cG_s}(\pi_0)$ and is represented by 
\[\beta_s(X)=b_s X+\text{ higher order terms}.\]
Assume that $\{E_1\}$ is a basis of $\Lie\cG_s$ and $\{E_2\}$ is a basis of $\Lie \cG_{s-1}$. We assume that the basis $\{E_1,E_2\}$ of $\Lie X_s^-$ lifts the basis $\{e_1,e_2\}$ of $\Lie \bX$ in Section \ref{subsec:Dieudonnemodulecalculation}. Then with respect to the basis $\{E_1,E_2\}$ of $\Lie X_s^-$, we have 
\[\iota_s^-(\pi)|_{\Lie X_s^-}=\left(\begin{array}{cc}
    0 & b_s \\
    c_s & 0
\end{array}\right).\]
Its $\pi$-eigenspace is then the span of 
\begin{equation}
    e_\pi=E_1+\frac{c_s}{\pi} E_2.
\end{equation}
By \cite[Corollary 4.8]{W}, we have
\[v(\Lie(\gamma_s))=1-v(\cG_s)=1-\frac{1}{2\q^s}.\]

The vector $e_\pi$ reduces to $e_1$ at the special fiber as $\mathrm{val}_{\pi_0}(\frac{c_s}{\pi})=\frac{1}{2}-\frac{1}{2\q^s}>0$. 
Then the filtration $\mathrm{span}\{e_\pi\}\subset\Lie X$ over $\cZ_{s,-}^\kappa$ is $\cF_X$ in the definition of Kr\"amer model. As $\cF_X$ reduces to $\mathrm{span}\{e_1\}$ at the special fiber, $\cZ_{s,-}^\kappa$ intersect $\cN^\Kra_s$ at $P_1=\mathrm{span}\{e_1\}$. 

The case for $\cZ_{s,+}^\kappa$ is the same with above except that one has to switch the roles  of  $e_1$ and $e_2$ in the above calculation. This finishes the proof of the proposition.
\end{proof}

Let $\bar\kappa:\Oo_F\rightarrow \Oo_{\mathbb H}$ be the embedding defined by 
\[\bar\kappa(a)=\kappa(\bar a).\]
\begin{proposition}\label{prop:cZ0intersectspecialfiber}
The divisor $\cZ_0^\kappa$ intersects with $\cN^\Kra_s$ at exactly one point of $\Exc$ which is not $P_1$ or $P_2$. Moreover $\cZ_0^\kappa$ and $\cZ_0^{\bar\kappa}$ do not intersect.
\end{proposition}
\begin{proof}
Since $\cZ_0^\kappa\cong \Spf \Oo_{\breve F}$, its intersection number with the special fiber $\cN^\Kra_s$ is the length of the module $\Oo_{\breve F}/(\pi)$ which is $1$. By Proposition \ref{intersectionnumber1withExc}, $\cZ_0^\kappa$ and $\Exc$ do intersect.
By Lemma \ref{lem:intersectionofZ1Z2}, $P_1$ and $P_2$ are where the divisors $\Exc$ and $s^{-1}(Z_2)$ in $\cN^\Kra_s$ intersect. If $\cZ_0^\kappa$ intersects with $\Exc$ at $P_1$ or $P_2$, the intersection number $\cZ_0^\kappa\cdot \cN^\Kra_s$ would be at least $2$ as $\cZ_0^\kappa$ would intersect both $\Exc$ and $s^{-1}(Z_2)$. This proves the first claim of the proposition.

As in the proof of Proposition \ref{prop:cZintersectwithspecialfiber}, $\cZ_0^\kappa$ intersects with $\Exc$ at the point which corresponds to the $\pi$-eigenspace of $\kappa(\pi)$. Similarly  $\cZ_0^{\bar\kappa}$ intersects with $\Exc$ at the point which corresponds to the $\pi$-eigenspace of $\bar\kappa(\pi)=\kappa(-\pi)$. Since these two eigenspaces are different, $\cZ_0^\kappa$ and $\cZ_0^{\bar\kappa}$ do not intersect.
\end{proof}

\begin{corollary}
\label{cor:strict transform respect summation}
The strict transform of $p_1^*(\cW_s^\kappa)$ under the blow-up $\Phi$ is the disjoint union of $\cZ_{s+1,+}^\kappa$ and $\cZ_{s,-}^\kappa$. The strict transform of $p_2^*(\cW_s^\kappa)$ under the blow-up $\Phi$ is the disjoint union of $\cZ_{s+1,-}^\kappa$ and $\cZ_{s,+}^\kappa$.
\end{corollary}
\begin{proof}
The corollary is a consequence of Lemma \ref{lem:piW and cY} and Proposition \ref{prop:blowupquasicanonicaldivisor}. The fact that  $\cZ_{s+1,+}^\kappa$ (resp. $\cZ_{s+1,-}^\kappa$) and $\cZ_{s,-}^\kappa$ (resp. $\cZ_{s,+}^\kappa$) are disjoint follows from Proposition \ref{prop:cZintersectwithspecialfiber} and \ref{prop:cZ0intersectspecialfiber}.
\end{proof}

\section{Special cycles }\label{specialcyclessection}
In this section, we define special cycles on $\cN^\Kra$ and decompose them into irreducible cycles that are introduced in Section \ref{sec:quasicanonicalliftingdivisor}.
\subsection{Definition of special cycles}
We can define special cycles in the same way as \cite{KR1} and \cite{Shi1}. The definition makes sense for both $\cN^\Pap$ and $\cN^\Kra$ and gives us closed sub formal schemes in both cases. 

First define the space of special quasi-homomorphisms to be the $F$-vector space
\begin{equation}\label{eq:bV}
    \bV=\mathrm{Hom}_{\Oo_F}(\bY,\bX)\otimes_\Z \Q.
\end{equation}
For $x,y \in \bV$, let
\begin{equation}\label{h(,)}
    h(x,y)=\lambda_\bY \circ y^\vee \circ\lambda_\bX \circ x\in \mathrm{End}_{\Oo_F}(\bY)\otimes \Q \xrightarrow[\sim]{\iota_\bY^{-1}} F.
\end{equation}

For $x,y\in \bV$ we abuse notation and denote the induced map between the corresponding relative Dieudonn\'e modules still by $x,y$. Then by \cite[Lemma 3.6]{Shi1} we have 
\begin{equation}\label{eq:relationsonhermitianforms}
    h(x,y)(e,e)_\bY=(x(e),y(e))_\bX,
\end{equation}
where $(,)_\bX$ and $(,)_\bY$ are the hermitian forms defined in \eqref{eq:definitionof(,)} for the rational relative Dieudonn\'e module of $\bX$ and $\bY$ respectively and $\{e\}$ is a basis of $M(\bY)$. In particular, $h(,)$ similar to $(,)_\bX$, so one is anisotropic if and only if the other is.

\begin{definition}\label{def:localspecialcycle}
For a fixed $m$-tuple ($m\leq 2$) $\bx=[x_1,\ldots,x_m]$ of special homomorphisms $x_i\in \bV$, the associated special cycle $\cZ^\Kra(\bx)$ ($\cZ(\bx)$ resp.) is the subfunctor of collections $\xi=(Y,\iota,\lambda_{Y}, \rho_{Y}; {X},\iota,\lambda_{ X},\rho_{ X},\cF)$ in $\cN_{(1,0)} \times \cN^\Kra$ ( $\xi=(Y,\iota,\lambda_{Y}, \rho_{Y}; {X},\iota,\lambda_{ X},\rho_{ X})$ in $\cN_{(1,0)} \times \cN^\Pap$ resp.) such that the quasi-homomorphism 
    \[\rho^{-1}_{X}\circ {\bf x} \circ \rho_{ Y}: {Y}^m\times_S \bar{S}\rightarrow { X}\times_S \bar{S}\]
deforms to a homomorphism from ${Y}^m$ to ${X}$.
\end{definition}
\begin{comment}
\begin{remark}
We know that $\cN_{(1,0)}\cong \SpfOF$ and $Y\cong \cG$, the canonical lift. We can fix such an identification by using the framing map $\rho_\cG=\mathrm{id}_\bG$. Then $\cN_{(1,0)} \times_{\Spf \Oo_{\breve{F}}} \cN^\Kra$  ($\cN_{(1,0)} \times_{\Spf \Oo_{\breve{F}}} \cN^\Pap $ resp.) is identified with $ \cN^\Kra$ ($\cN^\Pap$ resp.). Hence we often drop the first quadruple $(Y,\iota,\lambda_{Y}, \rho_{Y})$ in $\xi$ of Definition \ref{def:localspecialcycle}.
\end{remark}
\end{comment}

\subsection{Quasi-canonical lifting cycles in special cycles}\label{quasicanonicalliftingcycles}
Fix an $\bx\in \bV\setminus\{0\}$. Let $\mathrm{pr}_i:\bX\rightarrow \bY$ be the natural projection to the $i$-th factor where we identify $\bX$ with $\bY\times \bY$ as in \eqref{eq:definitionofbX}.
\begin{lemma}\label{lem:pr_1 pr_2 x}
For any $\bx\in \bV$, we have
\[\iota_\bY(\pi)\circ \mathrm{pr}_2\circ \bx=\mathrm{pr}_1\circ \bx\circ \iota_\bY(\pi),\quad \iota_\bY(\pi)\circ \mathrm{pr}_1\circ \bx=\mathrm{pr}_2\circ \bx\circ \iota_\bY(\pi).\]
\end{lemma}
\begin{proof}
Let $z$ be an $R$-point of $\bY$ where $R$ is a $k$-algebra. Then by  \eqref{eq:definitionofbX}, we know that 
\[\iota_\bX(\pi)\circ \bx (z)=(\iota_{\bY}(\pi)\circ \mathrm{pr}_2 \circ \bx(z),\iota_{\bY}(\pi)\circ \mathrm{pr}_1 \circ \bx(z)).\]
On the other hand, since $\bx$ is $\Oo_F$-linear, we know
\[\iota_\bX(\pi)\circ \bx (z)=\bx\circ \iota_\bY (\pi) (z)=(\mathrm{pr}_1\circ \bx\circ \iota_\bY (\pi) (z),\mathrm{pr}_2\circ \bx\circ \iota_\bY (\pi) (z)).\]
Equate the right hand sides of the above two equations, the lemma is proved.
\end{proof}
\begin{lemma}\label{lem:rho_0 is iso}
If $\bx\in \bV$ and $h(\bx,\bx)\in \Oo_{F_0}^\times$, then the morphism $\mathrm{pr}_1\circ \bx:\bY\rightarrow \bY$ is an isomorphism.
\end{lemma}
\begin{proof}
Assume that $\mathrm{pr}_1\circ \bx:\bY\rightarrow \bY$ is not an isomorphism. Since $\mathbb H$ is a division algebra complete with respect to the $\pi$-adic topology, we know that $\mathrm{pr}_1\circ \bx \circ \iota_\bY(\pi)^{-1}$ and $ \iota_\bY(\pi)^{-1}\circ\mathrm{pr}_1\circ \bx$ are morphisms as well. Then by Lemma \ref{lem:pr_1 pr_2 x}, we know that $\iota_\bY(\pi)^{-1}\circ\mathrm{pr}_2\circ \bx$ is also a morphism. Hence $\iota_\bX(\pi)^{-1}\circ \bx$ defines a morphism from $\bY$ to $\bX$. By the definition of $h(,)$ (see \eqref{h(,)}), we know that 
\[-\pi_0^{-1} h(\bx,\bx)=h(\iota_\bX(\pi)^{-1}\circ \bx,\iota_\bX(\pi)^{-1}\circ \bx)\in \Oo_{F_0}.\]
This contradict the fact that $h(\bx,\bx)\in \Oo_{F_0}^\times$.
\end{proof}

Fix a $\bx\in \bV\setminus\{0\}$, let $a$ be the $\pi_0$-adic valuation of $h(\bx,\bx)$. Define 
\begin{equation}\label{eq:rho_0}
   \bu=\bx\circ\iota_\bY(\pi^{-a}), \quad \rho_0:=\mathrm{pr}_1\circ \bu. 
\end{equation}
Then by Lemma \ref{lem:rho_0 is iso}, $\rho_0:\bY \rightarrow \bY$ is an isomorphism. Moreover $(\cG_0,\rho_0)=(\cG,\rho_0)$ is a canonical lift of $\bY$ with respect to the embedding  $\kappa_\bx:\Oo_F\rightarrow \End(\bY)\cong \Oo_{\mathbb H}$ defined by 
\begin{equation}\label{eq:kappa x}
  \kappa_\bx(a):=\rho_0\circ \iota_\bY(a) \circ \rho_0^{-1}=(\mathrm{pr}_1\circ \bx)\circ \iota_\bY(a) \circ (\mathrm{pr}_1\circ \bx)^{-1}.  
\end{equation}
We can construct a system of quasi-canonical lifts $(\cG_s,\rho_s)$ with respect to $\kappa_\bx$ as in Proposition \ref{prop:quasi canonical divisor on NPap}.
Define $\alpha_s: \cG\rightarrow \cG_{s}$ by 
\begin{equation}\label{eq:alpha_s}
    \alpha_s= \left\{\begin{array}{cc}
      \mathrm{Id}_\cG   & \text{ if } s=0 \\
      \beta_s\circ \beta_{s-1}\circ \ldots \circ \beta_1   & \text{ if } s\geq 1.
    \end{array}\right.
\end{equation}
Define the map $\bx_s^+:\cG\rightarrow X_s^+$ where $X_s^+$ is described in \eqref{eq:Xs+} by 
\begin{equation}\label{OEactiononXs+}
  \bx_s^+=[(\alpha_{m(s)}\circ \iota_\cG(\pi)^{s-m(s)})\times \alpha_s]\circ \Delta,
\end{equation}
where $m(s)$ is defined in \eqref{eq:m(s)} and $\Delta$ is the diagonal embedding $\cG\rightarrow X_0$.
\begin{lemma}\label{lem:OoFlinearality}
We have the following commutative diagram
\begin{equation}\label{eq:xs+ commutative diagram}
\begin{tikzcd}
\cG\otimes k\arrow{r}{\bx_s^+\otimes k} \arrow{d}{\mathrm{Id}_\bY} & X_{s}^+\otimes k \arrow{d}{\rho_s^+}\\
\bY \arrow{r}{\bu\circ \iota_\bY(\pi^s)} & \bX=\bY\times \bY
\end{tikzcd}.
\end{equation}
Moreover $\bx_s^+$ is an $\Oo_F$-module morphism.
\end{lemma}
\begin{proof}
First assume $s\geq 1$. Put the commutative diagrams in \eqref{eq:cXs definition diagram} together, we have
\[
   \rho_s\circ (\alpha_s \otimes k)=\iota_\bY(\pi^s)\circ \rho_0. 
\]
We then have
\begin{align*}
    &\rho_s^+\circ (\bx_s^+\otimes k)\\
    =&(\rho_{s-1}\times \rho_{s})\circ(((\alpha_{s-1}\circ \iota_\cG(\pi))\times \alpha_s)\otimes k)\circ i_1\\
    =&(\rho_{s-1}\circ (\alpha_{s-1} \otimes k)\circ \iota_\bY(\pi))\times (\rho_s\circ (\alpha_s \otimes k))\\
    =&(\iota_\bY(\pi)^{s-1}\circ \rho_0\circ \iota_\bY(\pi))\times (\iota_\bY(\pi)^s\circ \rho_0)\\
    =&\iota_\bX(\pi)^{s-1}\circ((\rho_0\circ \iota_\bY(\pi))\times (\iota_\bY(\pi)\circ \rho_0)).
\end{align*}
By the fact that $\rho_0=\mathrm{pr}_1\circ\bu$ and Lemma \ref{lem:pr_1 pr_2 x}, the above is equal to 
\begin{align*}
    &\iota_\bX(\pi)^{s-1}\circ((\mathrm{pr}_1\circ\bu\circ \iota_\bY(\pi))\times (\iota_\bY(\pi)\circ \mathrm{pr}_1\circ\bu))\\
    =&\iota_\bX(\pi)^{s-1}\circ((\mathrm{pr}_1\circ\bu\circ \iota_\bY(\pi))\times ( \mathrm{pr}_2\circ\bu\circ\iota_\bY(\pi)))\\
    =&\iota_\bX(\pi)^{s-1}\circ \bu \circ \iota_\bY(\pi)\\
    =&\bu \circ \iota_\bY(\pi^s).
\end{align*}
This proves the commutativity of the diagram in the lemma when $s\geq 1$. The proof for $s=0$ is similar and left to the reader.
Now we need to check that $\bx_s^+$ is an $\Oo_F$-module morphism.
Because the reduction map $\Hom(X,Y)\rightarrow \Hom(X\otimes k, Y\otimes k)$ is injective, it suffices to check that
\begin{equation}\label{eq:OoElinearilty}
    \rho_s^+\circ (\bx_s^+\otimes k)\circ \iota_\bY(\pi)=\rho_s^+\circ (\iota_s^+(\pi)\otimes k)\circ (\bx_s^+\otimes k).
\end{equation}
By the commutative of \eqref{eq:xs+ commutative diagram}, the left hand side of the above equation is $\bu\circ \iota_\bY(\pi)^{s+1}$. Recall the definition of $\iota_s^+$ in \eqref{eq:Xs+} and the definition of $\iota_{\bX}$ in \eqref{eq:definitionofbX}.  By the definition of the moduli problem $\cM_{\Gamma_0(\pi)}$, we know that
\[\rho_s^+\circ (\iota_s^+(\pi)\otimes k)=\iota_{\bX}(\pi)\circ \rho_s^+.\]
Combine this with the commutativity of \eqref{eq:xs+ commutative diagram}, we see 
\begin{align*}
    &\rho_s^+\circ (\iota_s^+(\pi)\otimes k)\circ (\bx_s^+\otimes k)\\
    =& \iota_{\bX}(\pi)\circ \rho_s^+ \circ (\bx_s^+\otimes k)\\
    =& \iota_{\bX}(\pi)\circ \bu \circ \iota_\bY(\pi)^s\\
    =& \bu\circ \iota_\bY(\pi)^{s+1}.
\end{align*}
This proves \eqref{eq:OoElinearilty}. The lemma is proved.
\end{proof}

Similarly define the map $\bx_s^-:\cG\rightarrow X_s^-$ by 
\begin{equation}\label{OEactiononXs-}
  \bx_s^-=[ \alpha_s \times (\alpha_{m(s)}\circ \iota_{\cG}(\pi)^{s-m(s)})]\circ \Delta.
\end{equation}
By a similar calculation as in the case of $\bx_s^+$, we know that $\bx_s^-$ is $\Oo_F$-linear and lifts the homomorphism 
$(\rho_{s}^-)^{-1}\circ (\bu \circ \iota_\bY(\pi^s))\circ \rho_\cG:\cG\otimes k\rightarrow X_s^-\otimes k$.

We summarize the above constructions by the following lemma.
\begin{lemma}\label{lem:quasicanonicalliftsinspecialcycle}
For $0\leq s \leq a$, $\cY_{s,+}^{\kappa_\bx}$ and $\cY_{s,-}^{\kappa_\bx}$ are sub formal schemes of $\cZ(\bx)$ while $\cZ_{s,+}^{\kappa_\bx}$ and $\cZ_{s,-}^{\kappa_\bx}$ are sub formal schemes of $\cZ^\Kra(\bx)$.  
\end{lemma}
\begin{proof}
By Lemma \ref{lem:OoFlinearality}, the composition of homomorphisms 
\[\cG\overset{\pi^{a-s}}{\longrightarrow}\cG\overset{\bx_s^+}{\longrightarrow} X_s^+\]
is $\Oo_F$-linear and
lifts $(\rho_s^+)^{-1}\circ \bx \circ \rho_\cG$. Hence $\cY_{s,+}^{\kappa_\bx}\hookrightarrow \cZ(\bx)$ and $\cZ_{s,+}^{\kappa_\bx}\hookrightarrow \cZ^\Kra(\bx)$. 
\end{proof}

\subsection{Decomposition of special divisors}
It is proved in \cite[Proposition 4.3]{Ho2} that for $\bx\in \bV\setminus\{0\}$, $\cZ^\Kra(\bx)$ is a Cartier divisor. We say a Cartier divisor $\cZ$ of $\cN^\Kra$ is irreducible if it is connected and is irreducible in each affine charts of $\cN^\Kra$. We say an irreducible divisor $S$ over $\SpfOF$ is horizontal if $\pi$ is not locally nilpotent on the structure sheaf of $S$. Otherwise we say $S$ is vertical.
Any irreducible divisor is either horizontal or vertical. Any Cartier divisor $\cZ$ of $\cN^\Kra$ that can be decomposed into irreducible ones can thus be decomposed into horizontal part and vertical part
\[\cZ=\cZ_h+\cZ_v\]
where every irreducible component of $\cZ_h$ (resp. $\cZ_v$) is horizontal 
(resp. vertical).

\begin{theorem}\label{decompositionofspecialcycles}
Let $\bx=\bx_0\cdot \pi^a\in \bV$ such that $h(\bx_0,\bx_0)\in \Oo_{F_0}^\times$. Define $\kappa_\bx$ as in \eqref{eq:kappa x}. Then we have the following equality of Cariter divisors on $\cN^\Kra$.
\[\cZ^\Kra(\bx)=\cZ_0^{\kappa_\bx}+\sum_{s=1}^a \cZ_{s,+}^{\kappa_\bx}+\sum_{s=1}^{a} \cZ_{s,-}^{\kappa_\bx}+ (a+1) \Exc.\]
\end{theorem}
\begin{proof}
Step 1.
$\Exc$ is obviously in $\cZ^\Kra(\bx)$ since the point Sing is always in $\cZ(\bx)$. And it follows from Lemma \ref{lem:quasicanonicalliftsinspecialcycle} and that all the other irreducible components on the right hand side of the equation indeed show up in the decomposition of $\cZ^\Kra(\bx)$.

\noindent Step 2.
By Proposition \ref{prop:exceptionaliso}, the universal object over $\cN^\Kra$ is $(X,\iota,\lambda,\rho,\cF_X)$ where $X=G_1\times G_2$ and $\rho=\rho_1\times \rho_2$ where $(G_1,\rho_1)$ and $(G_2,\rho_2)$ are deformations of $(\bY,\rho_\bY)$. Define $\cZ(\bx;\mathrm{pr}_i)$ to be the sub formal scheme of (the $i$-th) $\cM$ where  
\[\rho_i^{-1}\circ \mathrm{pr}_i \circ\bx\]
deforms to homomorphism $\cG\rightarrow G_1$. By \cite[Proposition 7.1]{RSZexotic}, we know that
\[\cZ(\bx;\mathrm{pr}_i)=\sum_{j=0}^a \cW_j^{\kappa_\bx}.\]
By Corollary \ref{cor:strict transform respect summation}, we have the following equality of Cartier divisors on $\cN^\Kra$
\[(\Phi\circ p_1)^*(\cZ(\bx;\mathrm{pr}_1))=\cZ_0^{\kappa_\bx}+\sum_{j=1}^a \cZ_{j,+}^{\kappa_\bx} +\sum_{j=1}^a \cZ_{j,-}^{\kappa_\bx} + \cZ_{a+1,+}^{\kappa_\bx}+a_1 \Exc,\]
\[(\Phi\circ p_2)^*(\cZ(\bx;\mathrm{pr}_2))=\cZ_0^{\kappa_\bx}+\sum_{j=1}^a \cZ_{j,+}^{\kappa_\bx} +\sum_{j=1}^a \cZ_{j,-}^{\kappa_\bx} +\cZ_{a+1,-}^{\kappa_\bx}+a_2 \Exc,\]
where $a_1,a_2\in \cZ_{\geq 0}$.
However by definition $\cZ^\Kra(\bx)$ is a sub formal scheme of both $(\Phi\circ p_i)^*(\cZ(\bx;\mathrm{pr}_i))$. Hence it can be decomposed into irreducible divisors and its horizontal part is also a sub divisor of both $(\Phi\circ p_i)^*(\cZ(\bx;\mathrm{pr}_i))_h$. This shows that the horizontal part of $\cZ^\Kra(\bx)$ is 
\[\cZ_0^{\kappa_\bx}+\sum_{s=1}^a \cZ_{s,-}^{\kappa_\bx}+\sum_{s=1}^{a} \cZ_{s,+}^{\kappa_\bx}.\]

\noindent Step 3.
Let $V$ be a vertical component of $\cZ^\Kra(\bx)$, We claim that $\Phi(V)$ must be supported on the reduced locus of $\cN^\Pap$ which is its unique point Sing. This claim follows from the same proof as that of \cite[Lemma 5.1.1]{LZ}. Alternatively there are isogenies $\rho_i^{-1}\circ \mathrm{pr}_i \circ\bx:\bY\rightarrow G_i$. Hence the pullbacks of $G_1$ and $G_2$ to $\cZ^\Kra(\bx)$ are supersingular. This implies that $V=\Exc$. 
It remains to calculate the multiplicity $m$ of the exceptional divisor $\Exc$ in $\cZ^\Kra(\bx)$. By Proposition \ref{intersectionnumber1withExc} and \ref{prop:cZ0intersectspecialfiber},  we know that 
\[m=\cZ_0^{\bar{\kappa}_\bx} \cdot \cZ^\Kra(\bx).\]
Notice that $\bar{\kappa}_\bx$ can be obtained by 
\[
  \bar\kappa_\bx(a)=\rho_0\circ \iota_\bY(\delta)\circ \iota_\bY(a) \circ \iota_\bY(\delta)^{-1}\circ \rho_0^{-1} 
\]
where $\rho_0$ is as in \eqref{eq:rho_0}.
If $\cZ_0^{\kappa_\bx}$ is determined by the first diagram in \eqref{eq:cXs definition diagram}, then $\cZ_0^{\bar{\kappa}_\bx}$ is determined by the  diagram
\[
\begin{tikzcd}
\cG\otimes k\arrow{r}{\beta_0\otimes k} \arrow{d}{\rho_0\circ \iota_\bY(\delta)} & \cG\otimes k \arrow{d}{\rho'_0\circ \iota_\bY(-\delta)}\\
\bY \arrow{r}{\iota_\bY(\pi)} & \bY
\end{tikzcd}.
\]
Apply \eqref{eq:xs+ commutative diagram} to the case $s=0$, we see that
\[(\rho_0\times \rho_0')^{-1}\circ \bu=\mathrm{Id}_\bY \times \mathrm{Id}_\bY. \]
Hence
\[
    ((\rho_0\circ \iota_\bY(\delta))\times(\rho'_0\circ \iota_\bY(-\delta)))^{-1}\circ \bx=\iota_\bY(\delta^{-1}\cdot \pi^a)\times\iota_\bY(-\delta^{-1}\cdot \pi^a).
\]
By the definition of $\cZ_0^{\bar{\kappa}_\bx}$ and $\cZ^\Kra(\bx)$, their intersection is the locus in $\cZ_0^{\bar{\kappa}_\bx}\cong\SpfOF$ such that the map 
\[((\rho_0\circ \iota_\bY(\delta))\times(\rho'_0\circ \iota_\bY(-\delta)))^{-1}\circ \bx: \bY\rightarrow \bY \times \bY\]
deforms to a homomorphism $\cG\rightarrow \cG\times \cG$. By \cite[Theorem 1.4]{W}, the above map lifts to $\Spf \Oo_{\breve{F}}/(\pi^{a+1})$ but not to $\Spf \Oo_{\breve{F}}/(\pi^{a+2})$. This finishes the proof.
\end{proof}

\section{Intersection of special cycles}
\subsection{Deformations of homomorphisms between quasi-canonical lifts}
Most of this subsection is the same as \cite[Section 7]{KR1}. A crucial difference occurs at \eqref{eq:n_r,sofpsi} (compared to \cite[Lemma 7.4]{KR1}), because we are working on schemes over $\SpfOF$ instead of $\Spf \Oo_{\breve{F}_0}$. As in loc. cit., we can assume $\kappa(a)=\iota_\bY(a)$ and $\rho_s=\mathrm{Id}_\bY$ for all $s\geq 0$. Then by \eqref{eq:normalizationofbetas} we know $\beta_s\otimes k =\iota_{\bY}(\pi)$.

Assume that $A$ is a finite extension of $\Oo_{\breve{F}_0}$ with uniformizer $\lambda$, and let $A_m:=A/ \lambda^{m+1}$. We also let $e$ be the absolute ramification index of $A$ over $\Oo_{\breve{F}_0}$ and denote by $\mathrm{ord}_A$ the discrete valuation on $A$ with $\mathrm{ord}_A(\lambda)=\frac{1}{e}$. Define
\begin{equation}\label{eq:e_s}
    e_s:=[W_s:\Oo_{\breve{F}_0}]=2\q^s
\end{equation}
to be the ramification index of $W_s/\Oo_{\breve{F}_0}$.
Suppose $\cG_r$ and $\cG_s$ are the quasi-canonical lifts defined over $A$. Suppose that a homomorphism 
\[\psi: \cG_r \times_{\Spf W_r} \Spec k \longrightarrow \cG_s \times_{\Spf W_s} \Spec k  \]
is given. Define $n_{r,s} (\psi)$ to be the maximal $m$ such that $\psi$ lifts to a homomorphism $\cG_r \times_{\Spf W_r} \Spf A_m \longrightarrow \cG_s \times_{\Spf W_s} \Spf A_m  $.

Let $H_{r,s}\subset \Oo_{\mathbb H}=\End(\bY)$ be the subset of elements $\phi$ that lift to homomorphisms from $\cG_r$ to $\cG_s$.  By \cite[Proposition 1.1]{W2}, if $s\geq r$, then $H_{r,s}=\pi^{s-r} \Oo_r$.
\begin{comment}
and there is an isomorphism 
\begin{equation}\label{H_r,s}
    H_{r,s} \overset{\sim}{\longrightarrow} H_{r,s+1}, \ \phi \mapsto \pi \phi.
\end{equation}
\end{comment}
Passage to dual isogenies shows that $n_{r,s}(\phi)=n_{s,r}(\phi^\vee)$. Hence we may assume $s\geq r$, as we shall do from now on. For $\psi \in \Oo_{\mathbb{H}} \setminus H_{r,s}$, let 
\begin{equation}\label{l_r,s}
   l_{r,s}(\psi)=\mathrm{max} \{v(\psi+\phi)\mid \phi\in H_{r,s}\},
\end{equation}
where $v$ is the valuation on $\mathbb H$ with $v(\pi)=1$. More explicitly, $l=l_{r,s}(\psi)$ is the positive integer such that 
\begin{equation}
    \psi \in (\pi^{s-r} \Oo_r +\pi^l  \Oo_{\mathbb{H}})\setminus(\pi^{s-r} \Oo_r +\pi^{l+1}  \Oo_{\mathbb{H}}).
\end{equation}
Notice that if $v(\psi)<s-r$, then $l_{r,s}(\psi)=v(\psi)$.
What we need is $n_{0,s}(\psi)$ which is determined in the following proposition.
\begin{proposition}\label{lengthofliftingsinquasicanonicallifts}
Let $l=l_{0,s}(\psi)$. Then
\[n_{0,s}(\psi)=e/e_s\cdot\left\{\begin{array}{ccc}
    &\frac{\q^{l+1}-1}{\q-1} & \text{ if } l<s, \\
     &\frac{\q^{s}-1}{\q-1}+ \frac{1}{2}(l+1-s) e_s & \text{ if } l\geq s.
\end{array}\right.\]
\end{proposition}
\begin{proof}
The proof is almost the same with that of \cite[Proposition 7.2]{KR1} with slight changes in numerology. So we only sketch the key steps here with emphasize on where the numerology differs.
First by \cite[Theorem 2.1]{V}, if $\psi\in \Oo_{\mathbb{H}}\setminus H_{r,r}$ with $l_{r,r}(\psi)=l$, then
\[n_{r,r}(\psi)=e/ e_r \left\{\begin{array}{cc}
   \frac{(\q^{l/2}-1)(\q+1)}{\q-1}+1  & \text{if } l\leq 2r\text{ and } l \text{ even}, \\
     \frac{(\q^{(l-1)/2}-1)(\q+1)}{\q-1}+\q^{(l-1)/2}+1  & \text{if } l\leq 2r\text{ and } l \text{ odd},  \\
     \frac{(\q^{r-1}-1)(\q+1)}{\q-1}+\q^{r-1}+(\frac{l+1}{2}-r)e_r +1  & \text{if } l\geq 2r-1.
\end{array}\right.  \]
Secondly suppose that $\cG_r$, $\cG_s$ and $\cG_{s+1}$ are defined over $A$ and that $\psi\in \Oo_{\mathbb H} \setminus H_{r,s}$ for $r\leq s$. Then by \cite[Corollary 5.3]{W}, we have
\[n_{r,s+1}(\pi\psi)=n_{r,s}(\psi)+e/e_{s+1}.\]
We can use induction on $s$ by the above equation. To initiate the induction process, we need the following equation which is a variant of Lemma 7.4 of \cite{KR1}.
\begin{equation}\label{eq:n_r,sofpsi}
  n_{r,s}(\psi)=\frac{1}{2}(e/e_{s})e_r \text{ if } l_{r,s}(\psi)=0 \text{ and } s\geq 2r.
\end{equation}
The whole process of inductive calculation is exactly the same as loc.cit, so we omit it. This finishes the proof of the proposition.
\end{proof}

\subsection{Intersection of special cycles}
We now describe a basis of $\bV$.
Define $i_1:\bY \rightarrow \bX$ and $i_2:\bY \rightarrow \bX$ by the function on $R$-points ($R$ a $k$-algebra)
\begin{equation}\label{i_1 i_2}
   i_1(x)=(x,x),\ i_2(x)=(\delta x,-\delta x).   
\end{equation}
\begin{lemma}
$i_1$ and $i_2$ are $\Oo_F$-morphisms.
\end{lemma}
\begin{proof}
It is a direct calculation. 
\begin{align*}
    i_1(\iota_\bY(\pi)x)=(\pi x,\pi x),\quad & i_2(\iota_\bY(\pi)x)=(\delta\pi x,-\delta\pi x),\\
    \iota_\bX(\pi)i_1(x)=(\pi x,\pi x), \quad & \iota_\bX(\pi)i_2(x)=(-\pi \delta x,\pi \delta x).
\end{align*}
The lemma follows.
\end{proof}

\begin{lemma}
We know that
\begin{equation}
    h(i_1,i_1)=2, h(i_1,i_2)=0, h(i_2,i_2)=-2\delta^2.
\end{equation}
In particular $i_1,i_2$ span a unimodular lattice in $\bV$.
\end{lemma}
\begin{proof}
We can check directly that
\[i_1(e)=v_1,\ i_2(e)=v_2\]
where $v_1,v_2$ are as in \eqref{eq:v_1 v_2}. The lemma now follows from equation \eqref{eq:relationsonhermitianforms}, \eqref{eq:(e,e)} and \eqref{eq:Gram matrix of v_1 v_2}.
\end{proof}

Define $\kappa:=\kappa_{i_1}$ as in \eqref{eq:kappa x}. Then $\kappa_{i_2}=\bar\kappa$.  

\begin{proposition}\label{intersectionnumberofquasicanonicalliftsandspecialcycles} Assume that $\bx=i_1\circ \pi^a$ and $\by=i_2\circ \pi^b$. Then we have
\[\cZ_{s,+}^\kappa\cdot \cZ^\Kra(\by)=\cZ_{s,-}^\kappa\cdot \cZ^\Kra(\by)=\left\{\begin{array}{ccc}
    &\frac{\q^{b+1}-1}{\q-1} & \text{ if } b\leq s, \\
     &\frac{\q^{s}-1}{\q-1}+ (b+1-s) \q^s & \text{ if } b\geq s,
\end{array}\right.\]
\[\cZ_{s,+}^{\bar\kappa}\cdot \cZ^\Kra(\bx)=\cZ_{s,-}^{\bar\kappa}\cdot \cZ^\Kra(\bx)=\left\{\begin{array}{ccc}
    &\frac{\q^{a+1}-1}{\q-1}  & \text{ if } a\leq s ,\\
     &\frac{\q^{s}-1}{\q-1}+ (a+1-s) \q^{s} & \text{ if } a\geq s.
\end{array}\right.\]
\end{proposition}
\begin{proof}
Let $A=W_s$ and $\lambda$ be a uniformizer of $A$. 
By Proposition \ref{prop:quasi canonical divisor on NPap} and \ref{prop:blowupquasicanonicaldivisor}, we know that $\cZ_{s,+}^\kappa\cong \Spf W_s$ and the formal $\Oo_{F_0}$-module over $ \cZ_{s,+}^\kappa$ is $\cG_{m(s)}\times \cG_{s}$.
Then $\cZ_{s,+}^\kappa\cap \cZ^\Kra(\by)$ is the locus of $\cZ_{s,+}^\kappa$ such that the map
\begin{equation}\label{mapfromF0totYs-}
    (\iota_\bY(\delta\pi^b) \times \iota_\bY(-\delta\pi^b))\circ i_1:\bY\rightarrow \bY\times \bY
\end{equation}
lifts to a homomorphism $\cG\rightarrow \cG_{m(s)}\times \cG_{s}$. Let
\[m=\cZ_{s,+}^\kappa\cdot \cZ^\Kra(\by),\]
then $m$ is biggest integer such that the map in equation \eqref{mapfromF0totYs-} can be lifted to $ \cG\times_{\Spf W_0} \Spec A_m \rightarrow (\cG_{m(s)}\times \cG_{s}) \times_{\Spf W_s} \Spec A_m $.

By Proposition \ref{lengthofliftingsinquasicanonicallifts}, the map 
\[\iota_\bY(\delta)\circ \iota_\bY(\pi)^b:\bY\rightarrow \bY \]
can be lifted to $ \cG\times_{\Spf W_0} \Spec A_{m_1} \rightarrow \cG_{s} \times_{\Spf W_s} \Spec A_{m_1} $, but not to $ \cG\times_{\Spf W_0} \Spec A_{m_1+1} \rightarrow \cG_s \times_{\Spf W_s} \Spec A_{m_1+1}$ where
\[m_1=\left\{\begin{array}{ccc}
    &\frac{\q^{b+1}-1}{\q-1} & \text{ if } b<s, \\
     &\frac{\q^{s}-1}{\q-1}+ (b+1-s) \q^s & \text{ if } b\geq s,
\end{array}\right.\]
as $l_{0,s}(\pi^b \circ \delta)=b$.

Similarly,  the map 
\[\iota_\bY(\delta)\circ \iota_\bY(\pi)^b:\bY\rightarrow \bY \]
can be lifted to $ \cG\times_{\Spf W_0} \Spec A_{m_2} \rightarrow \cG_{m(s)} \times_{\Spf W_s} \Spec A_{m_2} $, but not to $ \cG\times_{\Spf W_0} \Spec A_{m_2+1} \rightarrow \cG_{m(s)} \times_{\Spf W_s} \Spec A_{m_2+1}$ where 
\[m_2=\q^{s-m(s)}\cdot\left\{\begin{array}{ccc}
    &\frac{\q^{b+1}-1}{\q-1}  & \text{ if } b<m(s) ,\\
     &\frac{\q^{m(s)}-1}{\q-1}+ (b+1-m(s)) \q^{m(s)} & \text{ if } b\geq m(s),
\end{array}\right.\]
as $l_{0,s-1}(\pi^b \circ \delta)=b$.
Then we know that
\[m=min\{m_1,m_2\}=m_1.\]
Similar arguments produce the formula for the other cases in the proposition.
\end{proof}

\begin{corollary}\label{intersectionnumberofquasicanonicalliftingdivisors}
Suppose $s,t\geq 1$. Then 
\[\cZ_{s,+}^\kappa \cdot \cZ_{t,+}^{\bar\kappa}=\cZ_{s,-}^\kappa \cdot \cZ_{t,-}^{\bar\kappa}=\q^{min\{s,t\}}-1.\]
\end{corollary}
\begin{proof}
We prove
\begin{equation}\label{eq:cZ intersect cZ}
  \cZ_{s,-}^\kappa \cdot \cZ_{b,-}^{\bar\kappa}=\q^{min\{s,b\}}-1   
\end{equation}
The other case is similar and left to the reader. First by Proposition and \ref{prop:cZintersectwithspecialfiber} and \ref{prop:cZ0intersectspecialfiber}, we know for all $s,t\geq 0$
\begin{equation}\label{eq: cZ- intersect barcZ+}
    \cZ_{s,+}^\kappa\cdot \cZ_{t,-}^{\bar\kappa}=\cZ_{s,-}^\kappa\cdot  \cZ_{t,+}^{\bar\kappa}=0.
\end{equation}
When $s=0$ or $b=0$, equation \eqref{eq:cZ intersect cZ} follows directly from this. Now assume $s\geq 1$ and $b\geq 1$.
By Theorem \ref{decompositionofspecialcycles}, we have
\[
   \cZ^\Kra(i_2\circ \pi^b)=\cZ_0^{\bar\kappa}+\sum_{t=1}^b \cZ_{b,-}^{\bar\kappa}+ \sum_{t=1}^b \cZ_{b,+}^{\bar\kappa} +(b+1) \Exc. 
\]
This together with \eqref{eq: cZ- intersect barcZ+} and Proposition \ref{intersectionnumber1withExc} implies
\begin{equation}
   \cZ_{s,-}^\kappa \cdot \cZ^{\bar\kappa}_{b,-}
    =\cZ_{s,-}^\kappa \cdot \cZ^\Kra(i_2\circ \pi^{b})-\cZ_{s,-}^\kappa \cdot \cZ^\Kra(i_2\circ \pi^{b-1})-1. 
\end{equation}
Now apply Proposition \ref{intersectionnumberofquasicanonicalliftsandspecialcycles} to the right hand side of the above. 
If $b\leq s$, then 
\[\cZ_{s,-}^\kappa \cdot \cZ^\Kra(i_2\circ \pi^{b})-\cZ_{s,-}^\kappa \cdot \cZ^\Kra(i_2\circ \pi^{b-1})=\frac{\q^{b+1}-1}{\q-1}-\frac{\q^{b}-1}{\q-1}=\q^b.\]
If $b\geq s+1$, then 
\begin{align*}
    &\cZ_{s,-}^\kappa \cdot \cZ^\Kra(i_2\circ \pi^{b})-\cZ_{s,-}^\kappa \cdot \cZ^\Kra(i_2\circ \pi^{b-1})\\
    =&[\frac{\q^s-1}{\q-1}+(b+1-s)\q^s]-[\frac{\q^s-1}{\q-1}+(b-s)\q^s]\\
    =&\q^s.
\end{align*}
The corollary now follows.
\end{proof}

For a rank two integral lattice $\bL$ in $\bV$, define $\mathrm{Int}(\bL)$ as in equation \eqref{IntL}.
\begin{theorem}\label{thm:maintheorem1}
Suppose $\bV$ has dimension $2$ and $(\bV,h(,))$ is anisotropic.
Let $\bL$ be a hermitian lattice in $\bV$ whose hermitian form is represented by a Gram matrix equivalent to $T=\left(\begin{array}{cc}
    u_1 (-\pi_0)^a & 0 \\
    0 & u_2 (-\pi_0)^b
\end{array}\right)$, where $u_1,u_2\in \Oo_{F_0}^\times$ and $a,b\geq 0$. Then 
\[\mathrm{Int}(\bL)=\mu_\q(T),\]
where
\[\mu_\q(T):=2 \sum_{s=0}^{min\{a,b\}}\q^s(a+b+1-2s)-a-b-2.\]
\end{theorem}
\begin{proof}
By \cite{Ho2}, the intersection number $\mathrm{Int}(\bL)$ is well-defined. In other words, we can choose any basis $\{\bx,\by\}$ of  $\bL$, then we have 
\[\mathrm{Int}(\bL)=\cZ^\Kra(\bx) \cdot \cZ^\Kra(\by).\]
So we can choose $ \{\bx,\by\}$ such that 
\[\left(\begin{array}{cc}
    h(\bx,\bx) & h(\bx,\by) \\
    h(\by,\bx) & h(\by,\by)
\end{array}\right)=\left(\begin{array}{cc}
    u_1 (-\pi)^a & 0 \\
    0 & u_2 (-\pi)^b
\end{array}\right).\]
We can further assume that $\bx=i_1\circ \pi^a$ and $\by=i_2\circ \pi^b$ by changing the frame $\rho_\bX$ by an element in $\rU(C)$ if necessary (it is an isomorphism).

Decompose  $\cZ^\Kra(\bx)$ and $ \cZ^\Kra(\by)$ as in Theorem \ref{decompositionofspecialcycles}. Combine the decomposition with Proposition \ref{prop:cZintersectwithspecialfiber} and \ref{prop:cZ0intersectspecialfiber}, we have
\[\cZ^\Kra(\bx)\cdot \cZ^\Kra(\by)=\mathrm{I}+\mathrm{II}+\mathrm{III}\]
where
\[\mathrm{I}=(a+1)(b+1)\Exc\cdot \Exc,\]
\[\mathrm{II}=(\cZ_0^\kappa+\sum_{s=1}^a \cZ_{s,-}^\kappa+\sum_{s=1}^{a} \cZ_{s,+}^\kappa)\cdot (b+1)\Exc+(\cZ_0^{\bar\kappa}+\sum_{s=1}^b \cZ_{s,-}^{\bar\kappa}+\sum_{s=1}^{b} \cZ_{s,+}^{\bar\kappa})\cdot (a+1)\Exc,\]
\[\mathrm{III}=(\sum_{s=1}^b \cZ_{s,-}^{\bar\kappa})\cdot (\sum_{t=1}^a \cZ_{t,-}^\kappa)+(\sum_{s=1}^b \cZ_{s,+}^{\bar\kappa})\cdot (\sum_{t=1}^a \cZ_{t,+}^\kappa).\]
By Lemma \ref{lem:Excselfintersection},
\[\mathrm{I}=-2(a+1)(b+1).\]
By Proposition \ref{intersectionnumber1withExc},
\[\mathrm{II}=(2a+1)(b+1)+(2b+1)(a+1).\]
By Corollary \ref{intersectionnumberofquasicanonicalliftingdivisors}, 
\[(\sum_{s=1}^b \cZ_{s,-}^{\bar\kappa})\cdot (\sum_{t=1}^a \cZ_{t,-}^\kappa)=(\sum_{s=1}^b \cZ_{s,+}^{\bar\kappa})\cdot (\sum_{t=1}^a \cZ_{t,+}^\kappa)=\sum_{s=1}^b\sum_{t=1}^a \q^{min\{s,t\}}-ab.\]
For $0 \leq e \leq min\{a,b\}$, the set of indices $(s,t)$ that contributes to the summand $\q^e$ in the above summation is 
\[\{(s,t)\mid t=e,e< s \leq b\}\sqcup \{(s,t)\mid s=e,e\leq t \leq a\}.\]
Hence
\[\sum_{s=1}^b\sum_{t=1}^a \q^{min\{s,t\}}=\sum_{e=1}^{min\{a,b\}}(a+b-2e+1)\q^e.\]
Putting the terms $\mathrm{I}$, $\mathrm{II}$ and $\mathrm{III}$ together and simplify. This finishes the proof of the theorem.
\end{proof}

\section{A Kudla-Rapoport type theorem}\label{sec:KRconjecture}
In this section we relate the intersection number $\mu_\q(T)$ defined in Theorem \ref{thm:maintheorem1} to the derivative of certain local density polynomial. 
\subsection{Local density polynomials}
Let $|\cdot|$ be the absolute value on $F$ normalized such that $|\pi|=\frac{1}{\q}$.
For $X,Y\in \Herm_n(F_0)$, set $\langle X,Y\rangle=\mathrm{Tr}(XY)\in F_0$ where $\mathrm{Tr}$ is matrix trace.
For $S\in \Herm_m(F_0)$ and $X\in M_{m,n}(F)$, we denote $S[X]=X^* SX$. 
Define $X_n:=\Herm_n(F_0)\cap \GL_n(F)$.
Fix an additive character $\psi$ of $F_0$ with conductor $\Oo_{F_0}$. Define
\[ \Herm^\vee_n(\Oo_{F_0}):=\{ T=(t_{ij}) \in \Herm_n(F_0)\mid \ord_{\pi}(t_{ii}) \geq 0,  \ord_{\pi}(t_{ij}\partial_F) \geq 0 \text{ if } i\neq j\}\]
where $\partial_F$ is a generator of the different ideal of $F/F_0$. In the current paper we have $(\partial_F)=(\pi)$. When $F/F_0$ is unramified and $p\neq 2$, $(\partial_F)=(1)$. The lattice $\Herm^\vee_n(\Oo_{F_0})$ is the dual of $\Herm_n(\Oo_{F_0})$ under the pairing $\langle,\rangle$.
For $S\in \Herm^\vee_m(\Oo_{F_0})\cap X_m$ and $T\in\Herm^\vee_n(\Oo_{F_0})\cap X_n$, define
\begin{equation}\label{eq:definitionoflocaldensity}
    \alpha(S, T) = \int_{\Herm_n(F_0)} \int_{M_{m,n}(\Oo_F)} \psi(\langle Y,S[X]-T\rangle) dX dY
\end{equation}
where $dX$ and $dY$ are Haar measures on $M_{m,n}(\Oo_F)$ and $\Herm_n(\Oo_{F_0})$ respectively such that $\mathrm{vol}(M_{m,n}(\Oo_F), dX) =\mathrm{vol}(\Herm_n(\Oo_{F_0}), dY)=1$. The quantity $\alpha(S,T)$ only depends on the equivalent classes of $S$ and $T$.
\begin{lemma}\label{lem:definitionoflocaldensity}
\begin{align*}
		\alpha(S, T) 	&=\q^{\ell n (n-2m)} |\{ X \in M_{m,n}(\Oo_{F}/\pi_0^\ell):\,  S[X]-T \in  \pi_0^\ell \cdot \Herm^\vee_n(\Oo_{F_0})\}|\\
		&= |\partial_F|^{-\frac{n(n-1)}2}  \q^{\ell n (n-2m)} |\{ X \in M_{m,n}(\Oo_{F}/\pi_0^\ell):\,  S[X]-T \in  \pi_0^\ell \cdot \Herm_n(\Oo_{F_0}) \}| 
\end{align*}
for sufficiently large $\ell$. In particular the classical local density
\[\alpha^{\mathrm{cl}}(S, T)= \lim_{\ell \rightarrow \infty} \q^{\ell n (n-2m)} |\{ X \in M_{m,n}(\Oo_{F}/\pi_0^\ell):\,  S[X]-T \in  \pi_0^\ell \cdot \Herm_n(\Oo_{F_0}) \}|\]
differs from $\alpha(S, T)$ by the factor $|\partial_F|^{{\frac{n(n-1)}2}}$.
\end{lemma}
\begin{proof}
\begin{align*}
    \alpha(S,T)=& \lim_{\ell\rightarrow \infty}\int_{\frac{1}{\pi_0^\ell} \Herm_n(\Oo_{F_0})} dY \int_{M_{m,n}(\Oo_F)} \psi (\langle Y, S[X]-T\rangle) dX\\
    =& \lim_{\ell\rightarrow \infty}  \int_{M_{m,n}(\Oo_F)} dX \int_{\frac{1}{\pi_0^\ell} \Herm_n(\Oo_{F_0})} \psi (\langle Y, S[X]-T\rangle) dY\\
    =&  \lim_{\ell\rightarrow \infty}  \int_{M_{m,n}(\Oo_F)}  \mathrm{char}\{\pi_0^\ell \Herm^\vee_n(\Oo_{F_0})\}(S[X]-T)\cdot \mathrm{vol}(\frac{1}{\pi_0^\ell} \Herm_n(\Oo_{F_0}),dY)dX.
\end{align*}
Here we have used the orthogonality of characters on $\frac{1}{\pi_0^\ell} \Herm_n(\Oo_{F_0})$ and the duality between $\Herm_n(\Oo_{F_0})$ and $\Herm_n^\vee(\Oo_{F_0})$. Hence
\begin{align*}
    \alpha(S,T)=& \lim_{\ell\rightarrow \infty}  q^{n^2\ell} \int_{M_{m,n}(\Oo_F)}  \mathrm{char}\{\pi_0^\ell \Herm^\vee_n(\Oo_{F_0})\}(S[X]-T) dX\\
    =& \lim_{\ell\rightarrow \infty}  q^{n^2\ell} \sum_{x\in M_{m,n}(\Oo_F/(\pi_0^\ell))}\int_{x+\pi_0^\ell M_{m,n}(\Oo_F)}  \mathrm{char}\{\pi_0^\ell \Herm^\vee_n(\Oo_{F_0})\}(S[X]-T) dX.
\end{align*}
where the summation is over $x\in M_{m,n}(\Oo_F/(\pi_0^\ell))$. It is easy to see that if $S[X]-T\in \pi_0^\ell \Herm^\vee_n(\Oo_{F_0})$ then $S[\tilde X]-T\in \pi_0^\ell \Herm^\vee_n(\Oo_{F_0})$ for all $\tilde{X}$ such that $\tilde{X}-X\in \pi_0^\ell M_{m,n}(\Oo_F)$.  Hence we get the first equality in the lemma. Swap the roles of $\Herm_n(\Oo_{F_0})$ and $\Herm_n^\vee(\Oo_{F_0})$ in the above calculations, we get the second equality in the lemma.
\end{proof}

Let $\cH=\begin{pmatrix}0&\pi^{-1}\\-\pi^{-1}&0\end{pmatrix}$ and we use the same notation to define the rank $2$ hermitian lattice with Gram matrix $\cH$.
Define
\begin{equation}
    S_r=S\oplus \cH^r.
\end{equation}
By the methods of \cite[Corollary 5.4]{Hi} one can prove that
\begin{equation}\label{eq:localdensitypolynomial}
    \alpha(S_r,T)=\alpha(S,T,X)|_{X=\q^{-2r}}
\end{equation}
where $\alpha(S,T,X)\in\Q[X]$ is a polynomial. Let $\chi$ be the quadratic character of the extension $F^\times/F_0^\times$, namely,
\begin{equation}
    \chi(a)=\left\{\begin{array}{cc}
        1 & \text{ if } a\in \mathrm{Nm}_{F/F_0}(F^\times), \\
        -1 & \text{ if } a\in F^\times_0\setminus \mathrm{Nm}_{F/F_0}(F^\times).
    \end{array}\right.
\end{equation}
For $T\in \Herm_n(F_0)$, define 
\[\chi(T)=\chi((-1)^{\frac{n(n-1)}{2}} \mathrm{det}(T)).\]
The following theorem will be use in this paper as well as \cite{HSY}.
\begin{theorem}\label{thm:localdensity calculation}
Assume that $u_1,u_2,v\in  \Oo_{F_0}^\times$. Define $\epsilon_1 =\chi(-u_1 u_2)$ and $\epsilon_2 =\chi(-v)$. Assume $a\geq b$ are non-negative integers.
Let $S=\mathrm{Diag}\{v,1\}$.
\begin{enumerate}
    \item Assume that $T=\mathrm{Diag}\{u_1 (-\pi_0)^a,u_2 (-\pi_0)^b\}$. Then we have
\begin{align}\label{eq:local density first formula}
\alpha(S, T, X) 
 =& (1-X) (1+ \epsilon_2 + \q \epsilon_2) \sum_{e=0}^b (\q X)^e 
   -\epsilon_1 \q^{b+1} X^{a+1} (1-X) \sum_{e=0}^{b} (\q^{-1}X)^e
   \\ \nonumber
   -&\epsilon_1 (1+\q) ( X^{a+ b +2}+ \epsilon_1 \epsilon_2)
   + (1+\epsilon_2) \q^{b+1} X^{b+1} (1+\epsilon_1 X^{a-b}).
\end{align}
And
\begin{equation}\label{eq:local density second formula}
    \alpha(\cH,T,X)=(1-\q^{-2}X)(\sum_{e=0}^b\q^e X^e+\epsilon_1\sum_{e=a+1}^{a+b+1} \q^{a+b+1-e} X^e).
\end{equation}

\item  
Assume that $T= \left(\begin{array}{cc}
    0 & (-\pi)^{2a-1} \\
   (\pi)^{2a-1}  & 0
\end{array}\right)$ where $a\geq 0$.  Then
\begin{align}\label{eq:local density third formula}
     \alpha(S,T,X)=&(1-X)(1+\epsilon_2+\q\epsilon_2)\sum_{e=0}^{a} (\q X)^e -\q^{2a+1}(1-X)\sum_{e=a+1}^{2a}(\q^{-1}X)^e\\ \nonumber
    &-(\q+1)(X^{2a+1}+\epsilon_2)+(\q+1)(1+\epsilon_2)\q^{a}X^{a+1}.
\end{align}
And 
\begin{equation}\label{eq:local density fourth formula}
     \alpha(\cH,T,X)=(1-\q^{-2}X)(\sum_{e=0}^{a} \q^e X^e+\sum_{e=a+1}^{2a} \q^{2a-e} X^e).
\end{equation}
\end{enumerate}
\end{theorem}
We postpone the proof of the theorem to the next section. Now let us focus on its consequence.
\subsection{A Kudla-Rapoport type theorem}
We now go back to the situation of Theorem \ref{thm:maintheorem1}.
We define $\alpha'(S,T)$ as in \eqref{eq:alphaprime}. Hence 
\begin{equation}\label{alpha'andF'}
    \frac{\partial}{\partial r} \alpha(S_r,T)|_{r=0}=2 \log \q\cdot \alpha'(S,T).
\end{equation}
\begin{theorem}\label{thm:maintheorem2} Suppose $\bV$ has dimension $2$ and $(\bV,h(,))$ is anisotropic. 
Let $\bL$ be a hermitian lattice in $\bV$ whose hermitian form is represented by a Gram matrix $T$. Then 
\[\mathrm{Int}(\bL)
   =2 \frac{\alpha'(S,T)}{\alpha(S,S)},\]
where $S=\left(\begin{array}{cc}
    v & 0 \\
    0 & 1
\end{array}\right)$ with $v\in \Oo_{F_0}^\times$ and $\chi(-v)=1$.
\end{theorem}
\begin{proof}
Assume as before that $\bx=(x_1,x_2)$ and $\bL=\mathrm{span}\{x_1,x_2\}$. First assume $T\in \Herm_2(\Oo_{F_0})$. Since both sides of the equation only depends on the equivalent class of $T$, we can assume $T$ is of the form in Theorem \ref{thm:maintheorem1}.
By the assumption that $(\bV,h(,))$ is anisotropic, we know that $\chi(-u_1 u_2)=-1$. Without loss of generality we assume $a\geq b$.
Evaluate the first equation of Theorem \ref{thm:localdensity calculation} at $\epsilon_1=\epsilon_2=1$ and $a=b=0$, we have
\[\alpha(S,S)=\alpha(S,S,1)=(-X^2+(2\q-2)X+1)|_{X=1}=2\q-2.\]
Evaluate the same equation at $\epsilon_1=-1$, $\epsilon_2=1$, we have
\begin{align}\label{eq:alpha(S,T)}
\alpha(S, T, X) 
 =& (1-X) (2 + \q ) \sum_{e=0}^b (\q X)^e 
   + \q^{b+1} X^{a+1} (1-X) \sum_{e=0}^{b} (\q^{-1}X)^e
   \\ \nonumber
   +& (1+\q) ( X^{a+ b +2}-1)
   + 2 \q^{b+1} X^{b+1} (1- X^{a-b}).
\end{align}
Hence
\begin{align*}
    &\alpha'(S,T)\\
    =&(2+\q)\sum_{e=0}^b \q^e+\q^{b+1} \sum_{e=0}^b \q^{-e}-(1+\q)(a+b+2)+2\q^{b+1}(a-b)\\
    =&2\q^{b+1}(a-b+1)+4\sum_{s=1}^b \q^s-2\q-(\q+1)(a+b).
\end{align*}
On the other hand, by theorem \ref{thm:maintheorem1}, we have 
\begin{align*}
    &\frac{1}{2}\alpha(S,S)\mathrm{Int}(\bL)
    =\frac{1}{2}\alpha(S,S)\mu_\q(T)\\
    =&(\q-1)[2 \sum_{s=0}^{b}\q^s(a+b+1-2s)-a-b-2]\\
    =&2 \sum_{s=1}^{b+1}\q^{s}(a+b+3-2s)-2\sum_{s=0}^{b}\q^s(a+b+1-2s)-(\q-1)(a+b+2)\\
    =&2\q^{a+1}(a-b+1)+4\sum_{s=1}^b \q^s-2\q-(\q+1)(a+b).
\end{align*}
This finishes the proof of the theorem when $T\in \Herm_2(\Oo_{F_0})$.

If $T\in \Herm_2(F_0)\setminus\Herm_2(\Oo_{F_0})$, then $\cZ(\bx)$ is empty by \cite[Theorem 1.2]{Shi1}. Meanwhile since $\bL$ is anisotropic, by Jordan decomposition and \cite[Proposition 8.1]{J} there exists an element $v\in \bL$ such that 
\[\mathrm{val}_{\pi_0}(h(v,v))\leq -1.\]
Hence $\alpha(S_r,T)=0$ for all $r$. This shows that $\alpha'(S,T)=0$. So again the equality in the theorem holds.   
\end{proof}

For later use, we also make the following calculations. Let $T$ be as in Theorem \ref{thm:maintheorem1}.
Let $S'=\left(\begin{array}{cc}
    v' & 0 \\
    0 & 1
\end{array}\right)$ with  $ \chi(-v')=-1$.
Then by Theorem \ref{thm:localdensity calculation}, we have
\begin{equation}\label{alphaS'S'}
    \alpha(S',T)=\alpha(S',S')=2(\q+1),
\end{equation}
In particular, it is independent of $a$ and $b$.

\section{Local density of hermitian forms}\label{sec:local density}
In this section we prove Theorem \ref{thm:localdensity calculation}. We denote by $\mathrm{tr}(\cdot)$ the trace function of $F/F_0$. We use the uniformizer $\varpi_0=-\pi_0$ of $F_0$ since it simplifies many expressions. We will use results from \cite{Hi} to prove the first and the second formula of Theorem \ref{thm:localdensity calculation}, and use results from \cite{LL2} to prove the second and fourth formula.
\subsection{Hironaka's algorithm}
For any congruence subgroup $\Gamma \subset \GL_n(\Oo_F)$, define
\begin{equation}
    \alpha(Y;\Gamma):=\lim_{d\rightarrow \infty} q^{-dn^2} N_{d}(Y;\Gamma) ,
\end{equation}
where 
\[N_d(Y;\Gamma):=| \big\{ \gamma\in \Gamma \ \mathrm{mod}(\pi_0^d) \mid \gamma^* Y\gamma \equiv Y \mathrm{mod} (\pi_0^d) \big\}|.\]

\begin{theorem}\label{thm:Hironakaformula} (\cite[Proposition 3.2]{Hi})
Let $\Gamma$ be a congruence subgroup of $\GL_n(\Oo_F)$. For $T\in X_n$, $ S \in X_m$, we have
\[\alpha(S,T)=\sum_{Y \in \Gamma \backslash X_n} \frac{\mathcal{G}_{\Gamma}(Y,T)\cG(Y,S)}{\alpha(Y;\Gamma)},\]
where 
\[ \cG(Y,S)=\int_{M_{m,n}(\Oo_F)} \psi(\langle Y,S[X]\rangle)dX,\]
\[\cG_\Gamma(Y,T)=\int_{\Gamma} \psi(\langle Y[\gamma],-T\rangle)d\gamma=\int_{\Gamma} \psi(\langle Y,-T[\gamma^*]\rangle)d\gamma.\]
Here $d\gamma$ is the Haar measure on $M_{n,n}(\Oo_F)$ such that $M_{n,n}(\Oo_F)$ has measure one and $dX$ is the Haar measure on $M_{m,n}(\Oo_F)$ such that $M_{m,n}(\Oo_F)$ has measure one.
\end{theorem}
\begin{remark}\label{rmk:error of Hironaka}
In \cite[Proposition 3.2]{Hi}, $\cG_\Gamma(Y,T)$ is defined to be
\[\int_{\Gamma} \psi(\langle Y,-T[\gamma]\rangle)d\gamma\]
which is incorrect.
\end{remark}
Now we specialize to the case when $n=2$ and $\Gamma$ is the Iwahori subgroup
\[\Gamma=\{\gamma=(\gamma_{ij})\in \GL_n(\Oo_F)\mid \gamma_{21}\in \pi \Oo_F\}.\]

\subsection{Classification of hermitian lattice under the Iwahori subgroup}\label{subsec:Iwahoriclassify}
By \cite[Theorem 2.8]{Hi}, the set $\Gamma\backslash X_2$ can be represented by elements of the form 
\begin{equation}\label{eq:Gamma_0orbitsofhermitianmatrix}
    Y_{(id,e_1,e_2,\epsilon_1,\epsilon_2)}=\left(\begin{array}{cc}
   \epsilon_1 \varpi_0^{e_1}  & 0 \\
    0 & \epsilon_2 \varpi_0^{e_2}
\end{array}\right),
\text{ and } Y_{((12),e)}=\left(\begin{array}{cc}
   0  & \pi^e \\
    (-\pi)^e & 0
\end{array}\right),
\end{equation}
where $ \epsilon_i=1 $ or $\sigma$ and $e_1,e_2,e\in \Z$. Here $\sigma\in \Oo_{F_0}^\times$ and $\chi(\sigma)=-1$.
\begin{lemma}\label{lem:alpha Gamma}
Define
\[\alpha(e_1,e_2):=\left\{\begin{array}{cc}
   4q^{e_2+3e_1-1}  & \text{ if } e_2\geq e_1  \\
   4q^{e_1+3 e_2}   &  \text{ if } e_2< e_1.
\end{array}\right.\]
Then $\alpha(Y_{(id,e_1,e_2,\epsilon_1,\epsilon_2)};\Gamma)=\alpha(e_1,e_2)$. We also have $\alpha(Y_{((12),e)};\Gamma)=q^{2e-1} (q-1)$.
\end{lemma}
\begin{proof}
This is a special case of \cite[Theorem 2.10]{Hi}.
\end{proof}

\subsection{Gauss integrals}\label{subsec:Gaussintegrals}
We denote by $\bar{\chi}$ the function on $\Oo_F/\pi \Oo_F\cong\Oo_{F_0}/ \varpi_0 \Oo_{F_0}\cong \F_q$ such that 
\begin{equation}
    \bar{\chi}(a)=\left\{\begin{array}{cc}
        1 & \text{ if } a\in (\F_q^\times)^2, \\
        -1 & \text{ if } a\in \F_q^\times\setminus (\F_q^\times)^2, \\
        0 & \text{ if } a= 0.
    \end{array}\right.
\end{equation}
Let $\psi$ be an additive character of $F_0$ with conductor $\Oo_{F_0}$.
Define $\barpsi$ on $\Oo_{F_0}/ \varpi_0 \Oo_{F_0}\cong \F_q$ by 
\begin{equation}
    \barpsi(\bar{a})=\psi(\frac{a}{\varpi_0}).
\end{equation}
where $a\in \Oo_{F_0}$ is any representative of $\bar{a}\in \Oo_{F_0}/ \varpi_0 \Oo_{F_0}$. We need some preliminary results on some Gauss type of integrals.
For $a\in \F_q^\times$ define the Gauss sum on $\F_q$ by 
\begin{equation}
 g_a(\barchi,\barpsi)=\sum_{b\in \F_q}\barchi(b) \barpsi(ab).
\end{equation}
We also denote $g_1(\barchi,\psi)$ simply by $g(\barchi,\psi)$. Then 
\begin{equation}\label{eq:g_aandg}
    g_a(\barchi,\barpsi)=\barchi(a^{-1}) g(\barchi,\barpsi)=\barchi(a) g(\barchi,\barpsi).
\end{equation}
It is easy to see that 
\begin{equation}
    g_a(\barchi,\barpsi)=\sum_{b\in \F_q} \barpsi(ab^2)
\end{equation}
using the fact that $\sum_{b\in R} \barpsi(ab)+\sum_{b\in N} \barpsi(ab)+1=0$ where $R=(\F_q^\times)^2$ and $N=\F_q^\times \setminus (\F_q^\times)^2$. The following is a classical result.
\begin{lemma}\label{lem:guasssumsquare}
\[g_a(\barchi,\barpsi)^2=q\cdot \barchi(-1).\]
\end{lemma}
\begin{remark}
It turns out that in the calculation of local density, only $g_a(\barchi,\barpsi)^2$ instead of $g_a(\barchi,\barpsi)$ shows up in the final expression. This is to be expected since local density is independent of the choice of character $\psi$.
\end{remark}

\begin{lemma}\label{lem:Gaussintegral}
Let $v\in F^\times$ and $w \in F_0^\times$. Assume $w=w_0 \varpi_0^s$ where $s=\mathrm{val}_{\varpi_0}(w)$.
\begin{enumerate}
\item Define $J(w):=\int_{\Oo_F} \psi(w x \bar{x}) dx$, then
\[J(w)=\left\{\begin{array}{cc}
     1 & \text{ if }s\text{ is non-negative}, \\
     q^{s} \chi(w_0) g(\barchi,\barpsi) & \text{ if }s\text{ is negative}.
\end{array}\right.\]
\item Define $J^*(w):=\int_{\Oo_F^\times} \psi(w x \bar{x}) dx$. Then we have 
\[J^*(w)=\left\{\begin{array}{cc}
   0  &  \text{ if } \mathrm{val}_{\varpi_0}(w)<-1,\\
   \frac{1}{q}\chi(w_0) g(\barchi,\barpsi) -\frac{1}{q} & \text{ if } w=w_0 \varpi_0^{-1}, w_0 \in \Oo_F^\times,\\
   1-\frac{1}{q}  &  \text{ if } \mathrm{val}_{\varpi_0}(w)\geq 0.
\end{array}\right.\]
In particular, it is zero if and only if $\mathrm{val}_{\varpi_0}(w)<-1$.
\item 
$\int_{\Oo_F}\int_{\Oo_F} \psi(\mathrm{tr}(\frac{1}{\pi} v\bar{x} y))dx dy=q^{min\{\mathrm{val}_\pi(v),0\}} $.
\item $\int_{\Oo_F}\int_{\Oo_F^\times} \psi(\mathrm{tr}(\frac{1}{\pi} v\bar{x} y))dx dy=q^{min\{\mathrm{val}_\pi(v),0\}}-q^{min\{\mathrm{val}_\pi(v)+1,0\}-1} $. In other words it is zero if $\mathrm{val}_{\pi}(v)<0$ and is $(1-\frac{1}{q})$ if $\mathrm{val}_{\pi}(v)\geq 0$ .
\item \begin{align*}
    &\int_{\Oo_F^\times}\int_{\Oo_F^\times} \psi(\mathrm{tr}(\frac{1}{\pi} v\bar{x} y))dx dy\\
    =&q^{min\{\mathrm{val}_\pi(v),0\}}-2q^{min\{\mathrm{val}_\pi(v)+1,0\}-1} +q^{min\{\mathrm{val}_\pi(v)+2,0\}-2}.
\end{align*}
In other words, it is $(1-q^{-1})^2$ if $\mathrm{val}_\pi(v)\geq 0$, $ -q^{-1}(1-q^{-1})$ when $\mathrm{val}_\pi(v)=-1$ and $0$ if $\mathrm{val}_\pi(v)<0$.
\end{enumerate}
\end{lemma}
\begin{proof}
(1), (2), (3) and (5) are covered in \cite[Lemma 5.1]{Hi}. We write down the proof of (1) for illustration.
If $\mathrm{val}_{\varpi_0}(w) \geq 0$, then $J(w)=1$. Assume now $\mathrm{val}_{\varpi_0}(w)<0$ and let $k=-\mathrm{val}_{\varpi_0}(w)$. Then  
\begin{align*}
     J(w)=& \sum_{z \in \Oo_F/(\pi^k)} \int_{z+\pi^k \Oo_F} \psi(w x\bar{x})dx\\
     =& \q^{-k} \sum_{z \in \Oo_F/(\pi^k)} \int_{\Oo_F} \psi(w (z+\pi^k x)\overline{(z+\pi^k x)})dx \\
     =& \q^{-k} \sum_{z \in \Oo_F/(\pi^k)} \psi(w z\bar{z}) \int_{\Oo_F} \psi[\mathrm{tr}(\pi^k w\bar{z} x)]dx\\
     =& \q^{-k} \sum_{z \in \Oo_F/(\pi^k)} \psi(w z\bar{z})  \mathrm{char}(\Oo_F) (w \bar{z} \pi^{k+1}).
\end{align*}
The nonzero terms in the above summation are those with $z\in (\pi^{k-1})$. So
\begin{align*}
    J(w)=& \q^{-k}\sum_{z\in \Oo_F/(\pi)}\psi(\frac{w_0}{\varpi_0} z \bar{z}) \\
    =& \q^{-k}\sum_{z\in \Oo_{F_0}/(\varpi_0)}\psi(\frac{w_0}{\varpi_0} z^2) \\
    =& \q^{-k} \chi(w_0) g(\barchi,\barpsi).
\end{align*}
Here in the last step we use \eqref{eq:g_aandg}.

(4) follows from (3) as
\begin{align*}
  &\int_{\Oo_F}\int_{\Oo_F^\times} \psi(\mathrm{tr}(\frac{1}{\pi} v\bar{x} y))dx dy\\
  =&\int_{\Oo_F}\int_{\Oo_F} \psi(\mathrm{tr}(\frac{1}{\pi} v\bar{x} y))dx dy-\frac{1}{\q}\int_{\Oo_F}\int_{\pi \Oo_F} \psi(\mathrm{tr}(\frac{1}{\pi} v\bar{x} y))dx dy\\
  =& \int_{\Oo_F}\int_{\Oo_F} \psi(\mathrm{tr}(\frac{1}{\pi} v\bar{x} y))dx dy-\frac{1}{\q}\int_{\Oo_F}\int_{ \Oo_F} \psi(\mathrm{tr}(\frac{1}{\pi} (-v \pi )\bar{x} y))dx dy.
\end{align*}
This finishes the proof of the lemma.
\end{proof}

Due to Remark \ref{rmk:error of Hironaka}, we need to recalculate $\cG_\Gamma(Y,T)$ in the formula of Theorem \ref{thm:Hironakaformula}. Recall that we have defined $Y_{(id,e_1,e_2,\epsilon_1,\epsilon_2)}$ and $Y_{((12),e)}$ in \eqref{eq:Gamma_0orbitsofhermitianmatrix}.
\begin{lemma}\label{lem:G_Gamma T diagonal}
Assume that $T=\mathrm{Diag}\{u_1\varpi_0^a, u_2\varpi_0^b\}$,
where $u_1,u_2\in \Oo_{F_0}^\times$. 
Then we have
\[\cG_\Gamma \left(Y_{(id,e_1,e_2,\epsilon_1,\epsilon_2)},T\right)=\frac{1}{q}  J^*(-\epsilon_1 u_1 \varpi_0^{a+e_1}) J(-\epsilon_1 u_2\varpi_0^{b+e_1}) J(-\epsilon_2 u_1 \varpi_0^{a+e_2+1}) J^*(-\epsilon_2  u_2\varpi_0^{b+e_2}),\]
and
\[\cG_\Gamma \left(Y_{((12),e)},T\right)=\frac{1}{q}(q^{min\{2a+e+2,0\}}-q^{min\{2a+e+3,0\}-1}) (q^{min\{2b+e+1,0\}}-q^{min\{2b+e+2,0\}-1}).\]
\end{lemma}
\begin{proof}
Let $\gamma=\left(\begin{array}{cc}
    x & y \\
    z & w
\end{array}\right)\in \Gamma$. Then 
\[T[\gamma^*]=\left(\begin{array}{cc}
    u_1\varpi_0^a x\bar{x} +u_2\varpi_0^b y \bar{y} & u_1\varpi_0^a x \bar{z} +u_2\varpi_0^b y \bar{w} \\
    u_1\varpi_0^a z \bar{x} +u_2\varpi_0^b \bar{y}w &  u_1\varpi_0^a z\bar{z} +u_2\varpi_0^b w \bar{w}
\end{array}\right).\]
So we have
\begin{align*}
    &\cG_\Gamma (Y_{(id,e_1,e_2,\epsilon_1,\epsilon_2)},T) \\
=&\int_{\Oo_F^\times} \psi(-\epsilon_1 u_1 \varpi_0^{a+e_1} x \bar{x}) dx \int_{ \Oo_F} \psi(-\epsilon_1  u_2 \varpi_0^{b+e_1} y \bar{y}) dy \int_{\pi \Oo_F} \psi(-\epsilon_2 u_1 \varpi_0^{a+e_2} z\bar{z}) dz \int_{\Oo_F^\times} \psi(-\epsilon_2  u_2\varpi_0^{b+e_2} w \bar{w}) dw  \\
=&\frac{1}{\q}\int_{\Oo_F^\times} \psi(-\epsilon_1 u_1 \varpi_0^{a+e_1} x \bar{x}) dx \int_{ \Oo_F} \psi(-\epsilon_1  u_2 \varpi_0^{b+e_1} y \bar{y}) dy \int_{ \Oo_F} \psi(-\epsilon_2 u_1 \varpi_0^{a+e_2+1} z\bar{z}) dz \int_{\Oo_F^\times} \psi(-\epsilon_2  u_2\varpi_0^{b+e_2} w \bar{w}) dw  \\
=& \frac{1}{\q} J^*(-\epsilon_1 u_1 \varpi_0^{a+e_1}) J(-\epsilon_1 u_2\varpi_0^{b+e_1}) J(-\epsilon_2 u_1 \varpi_0^{a+e_2+1}) J^*(-\epsilon_2  u_2\varpi_0^{b+e_2}).
\end{align*}

And we have
\begin{align*}
    &\cG_\Gamma (Y_{((12),e)},T) \\
=& \int_{\pi \Oo_F} \int_{\Oo_F^\times}\psi[\mathrm{tr}(-u_1\varpi_0^a \pi^e z \bar{x})]dx dz \int_{ \Oo_F} \int_{\Oo_F^\times}\psi[\mathrm{tr}(-u_2\varpi_0^b \pi^e w \bar{y})]dw dy\\
=& \frac{1}{\q} \int_{ \Oo_F} \int_{\Oo_F^\times}\psi[\mathrm{tr}(-u_1\varpi_0^{a} \pi^{e+1} z \bar{x})]dx dz \int_{ \Oo_F} \int_{\Oo_F^\times}\psi[\mathrm{tr}(-u_2\varpi_0^b \pi^e w \bar{y})]dw dy\\
=& \frac{1}{\q}(\q^{min\{2a+e+2,0\}}-\q^{min\{2a+e+3,0\}-1}) (\q^{min\{2b+e+1,0\}}-\q^{min\{2b+e+2,0\}-1}),
\end{align*}
where the last equality follows from (4) of Lemma \ref{lem:Gaussintegral}.
\end{proof}

The proof of the following lemma is similar to that of the previous one which follows directly from Lemma \ref{lem:Gaussintegral}, so we omit it.
\begin{lemma}
Assume that $T=Y_{((12),a)}$.
Then we have
\[\cG_\Gamma \left(Y_{(id,e_1,e_2,\epsilon_1,\epsilon_2)},T\right)= \frac{1}{q}(q^{min\{a+2 e_1+1,0\}}-q^{min\{a+2e_1+2,0\}-1})(q^{min\{a+2 e_2+2,0\}}-q^{min\{a+2e_2+3,0\}-1}),\]
and
\[\cG_\Gamma \left(Y_{((12),e)},T\right)=\frac{1}{q} (q^{min\{a+e+1,0\}}-2q^{min\{a+e+2,0\}-1}+q^{min\{a+e+3,0\}-2}) q^{min\{a+e+2,0\}}.\]
\end{lemma}

\begin{lemma}\label{lem:GYmultiplicative}
For any $Y\in X_2$,
\[\cG(Y,S_1\oplus S_2)=\cG(Y,S_1)\cG(Y,S_2).\]
\end{lemma}
\begin{proof}
Assume that $S_i$ has rank $m_i$ ($i=1,2$). Also assume that $x_1$ and $x_2$ are of size $m_1\times 1$, $x_3$ and $x_4$ are of size $m_2\times 1$. Let 
\[X_1=(x_1 \ x_2), \ X_2=(x_3 \ x_4),\ X=\left(\begin{array}{c}
     X_1  \\
     X_2 
\end{array}\right).\]
Then
\begin{align*}
    &\left(\begin{array}{cc}
        x_1^* & x_3^* \\
        x_2^* & x_4^*
    \end{array}\right)
    \left(\begin{array}{cc}
        S_1 & 0 \\
        0 & S_2
    \end{array}\right)
    \left(\begin{array}{cc}
        x_1 & x_2 \\
        x_3 & x_4
    \end{array}\right)\\
    =& \left(\begin{array}{cc}
        x_1^* & x_3^* \\
        x_2^* & x_4^*
    \end{array}\right)
    \left(\begin{array}{cc}
        S_1 & 0 \\
        0 & 0
    \end{array}\right)
    \left(\begin{array}{cc}
        x_1 & x_2 \\
        x_3 & x_4
    \end{array}\right)+\left(\begin{array}{cc}
        x_1^* & x_3^* \\
        x_2^* & x_4^*
    \end{array}\right)
    \left(\begin{array}{cc}
        0 & 0 \\
        0 & S_2
    \end{array}\right)
    \left(\begin{array}{cc}
        x_1 & x_2 \\
        x_3 & x_4
    \end{array}\right)\\
    =&\left(\begin{array}{cc}
        x_1^*S_1 x_1 & x_1^* S_1 x_2 \\
        x_2^* S_1 x_1 & x_2^*S_1 x_2
    \end{array}\right)+\left(\begin{array}{cc}
        x_3^*S_2 x_3 & x_3^* S_2 x_4 \\
        x_4^* S_2 x_3 & x_4^*S_2 x_4
    \end{array}\right)\\
    =& S_1[X_1]+S_2[X_2]
\end{align*}
Hence 
\[\langle Y,(S_1\oplus S_2)[X]\rangle=\langle Y,S_1[X_1]\rangle+ \langle Y,S_2[X_2]\rangle.\]
The lemma follows. 
\end{proof}

\begin{lemma}\label{lem:G(Y,S)}
Assume that $S_r=\mathrm{Diag}\{ v_1 \varpi_0^{f_1}, v_2 \varpi_0^{f_2}\}\oplus \cH^r$,
where $v_1, v_2\in \Oo_{F_0}^\times$.
Then we have
\begin{align*}
    & \cG \left(Y_{(id,e_1,e_2,\epsilon_1,\epsilon_2)},S_r\right)\\
=& J(v_1 \epsilon_1 \varpi_0^{e_1+f_1})J(v_1\epsilon_2 \varpi_0^{e_2+f_1}) J(v_2 \epsilon_1 \varpi_0^{e_1+f_2})J(v_2\epsilon_2 \varpi_0^{e_2+f_2}) (q^r)^{2min\{e_1,0\}+2min\{e_2,0\}},
\end{align*}
and
\[\cG \left(Y_{((12),e)},S_r\right)=q^{min\{e+2f_1+1,0\}} q^{min\{e+2f_2+1,0\}}(q^r)^{2 min\{e,0\}}.\]
\end{lemma}
\begin{proof}
\begin{align*}
    \cG (Y_{(id,e_1,e_2,\epsilon_1,\epsilon_2)},\langle v_1 \varpi_0^{f_1}\rangle)=&\int_{\Oo_F} \psi(\epsilon_1 v_1 \varpi_0^{e_1+f_1} x\bar{x} )dx\int_{\Oo_F} \psi(\epsilon_2 v_1 \varpi_0^{e_2+f_1} x\bar{x} )dx \\
 =& J(v_1 \epsilon_1 \varpi_0^{e_1+f_1})J(v_1\epsilon_2 \varpi_0^{e_2+f_1}).
\end{align*}
Similarly 
\begin{align*}
    \cG (Y_{(id,e_1,e_2,\epsilon_1,\epsilon_2)},\langle v_2 \varpi_0^{f_2}\rangle)
 = J(v_2 \epsilon_1 \varpi_0^{e_1+f_2})J(v_2\epsilon_2 \varpi_0^{e_2+f_2}).
\end{align*}
We also have
\begin{align*}
    \cG (Y_{(id,e_1,e_2,\epsilon_1,\epsilon_2)}, \cH)=& \int_{\Oo_F^2} \psi[\mathrm{tr}(\frac{1}{\pi}\epsilon_1 \varpi_0^{e_1} \bar{x}y)]dxdy  \int_{\Oo_F^2} \psi[\mathrm{tr}(\frac{1}{\pi}\epsilon_2 \varpi_0^{e_2} \bar{x}y)]dxdy \\
=& \q^{2min\{e_1,0\}+2min\{e_2,0\}}, \nonumber
\end{align*}
where the last equality follows from Lemma \ref{lem:Gaussintegral}. The first equation in the lemma now follows from Lemma \ref{lem:GYmultiplicative}. For the second equation in the lemma, we have
\begin{align*}
    \cG (Y_{((12),e)},\langle v_1 \varpi_0^{f_1}\rangle)=& \int_{\Oo_F^2} \psi [\mathrm{tr} (\frac{1}{\pi}v_1 \pi^{e+2f_1+1} \bar{x} y)] dx dy\\
=&\q^{min\{e+2f_1+1,0\}}, 
\end{align*}
where the last equality follows from Lemma \ref{lem:Gaussintegral}. Similarly
\[\cG (Y_{((12),e)},\langle v_2 \varpi_0^{f_2}\rangle)
=\q^{min\{e+2f_2+1,0\}},\]
and
\begin{align*}
    \cG (Y_{((12),e)},\cH)= \big\{\int_{\Oo_F^2} \psi [\mathrm{tr}(\pi^{e-1}\bar{x}y)dxdy]\big\}^2
= \q^{2 min\{e,0\}}. \nonumber
\end{align*}
Then the second equation in the lemma follows from Lemma \ref{lem:GYmultiplicative}.
\end{proof}

Similarly we can prove the following.
\begin{lemma}
Assume that $S_r=Y_{((12),f)}\oplus \cH^r$.
Then we have
\begin{align*}
    & \cG \left(Y_{(id,e_1,e_2,\epsilon_1,\epsilon_2)},S_r\right)
= q^{min\{2e_1+f+1,0\}+min\{2e_2+f+1,0\}} (q^r)^{2min\{e_1,0\}+2min\{e_2,0\}},
\end{align*}
and
\[\cG \left(Y_{((12),e)},S_r\right)=q^{2min\{e+f+1,0\}}(q^r)^{2 min\{e,0\}}.\]
\end{lemma}

\subsection{Proof of \eqref{eq:local density first formula}}
In the following we simply write $\mathfrak{g}=g(\barchi,\barpsi)$. 
By Theorem \ref{thm:Hironakaformula}, Lemma \ref{lem:alpha Gamma}, \ref{lem:G_Gamma T diagonal} and \ref{lem:G(Y,S)}, we have
\begin{align*}
    &\alpha(S_r,T)\\
    =&\sum_{\epsilon_1, \epsilon_2} \sum_{(e_1,e_2)\in \Z^2} \frac{1}{\q}  J^*(-\epsilon_1 u_1 \varpi_0^{a+e_1}) J(-\epsilon_1 u_2\varpi_0^{b+e_1}) J(-\epsilon_2 u_1 \varpi_0^{a+e_2+1}) J^*(-\epsilon_2  u_2\varpi_0^{b+e_2})\\
    &\cdot J(v\epsilon_1 \varpi_0^{e_1})J(v\epsilon_2 \varpi_0^{e_2})J(\epsilon_1 \varpi_0^{e_1})J(\epsilon_2 \varpi_0^{e_2})  (\q^r)^{2min\{e_1,0\}+2min\{e_2,0\}} \cdot \alpha(e_1,e_2)\\
    +&\sum_{e\in Z} \frac{1}{\q}(\q^{min\{2a+e+2,0\}}-\q^{min\{2a+e+3,0\}-1}) (\q^{min\{2b+e+1,0\}}-\q^{min\{2b+e+2,0\}-1})\\
    &\cdot \q^{2min\{e+1,0\}} (\q^r)^{2 min\{e,0\}} \frac{1}{\q^{2e-1}(\q-1)},
\end{align*}
where the summation of $\epsilon_1,\epsilon_2$ is over $\Oo_{F_0}^\times/\mathrm{Nm}(\Oo_F^\times)$. The second summation is easy to compute:
\begin{align}\label{eq: sum of alpha(e) terms}
     &\frac{1}{\q}(1-\frac{1}{\q})^2[\sum_{e=-2b-1}^{-1}\q^{2(e+1)} \q^{2er}\frac{1}{\q^{2e-1}(\q-1)}+\sum_{e=0}^\infty  \frac{1}{\q^{2e-1}(\q-1)}]\\ \nonumber
     =&(\q-1) \sum_{e=-2b-1}^{-1} \q^{2er}+\frac{1}{\q+1}.
\end{align}
For a subset $\Lambda$ of $\Z^2$ define 
\begin{align}\label{eq:I(Lambda)}
    I(\Lambda):=&\sum_{\epsilon_1, \epsilon_2} \sum_{(e_1,e_2)\in \Lambda} \frac{1}{\q} J^*(-\epsilon_1 u_1 \varpi_0^{a+e_1}) J(-\epsilon_1 u_2\varpi_0^{b+e_1}) J(\epsilon_2 u_1 \varpi_0^{a+e_2+1}) J^*(-\epsilon_2  u_2\varpi_0^{b+e_2})\\ \nonumber
    \cdot &  J(v\epsilon_1 \varpi_0^{e_1})J(\epsilon_1 \varpi_0^{e_1})J(v\epsilon_2 \varpi_0^{e_2})J(\epsilon_2 \varpi_0^{e_2})  (\q^r)^{2min\{e_1,0\}+2min\{e_2,0\}} \cdot\frac{1}{\alpha(e_1,e_2)}.
\end{align}
We know from Lemma \ref{lem:Gaussintegral} that $J^*(-\epsilon_1 u_1 \varpi_0^{a+e_1})\neq 0$ if and only if $ e_1\geq -a-1$, and $J^*(-\epsilon_2 u_2 \varpi_0^{b+e_2})\neq 0$ if and only if $ e_2\geq -b-1$. In the following calculations, we will repeatedly using Lemma \ref{lem:guasssumsquare} and \ref{lem:Gaussintegral} sometimes without explicitly referring to them.
We will also use the following obvious facts
\begin{equation}\label{eq: summation of epsilon}
    \sum_{\epsilon \in \Oo_{F_0}^\times/\mathrm{Nm}(\Oo_F^\times)}\chi(\epsilon)=0, \ \sum_{\epsilon \in \Oo_{F_0}^\times/\mathrm{Nm}(\Oo_F^\times)}\chi(\epsilon^2)=2.
\end{equation}
\begin{lemma}\label{lem: I Z -b-1}
\begin{align*}
     &I(\Z\times \{-b-1\})\\
     =& -  (\q^r)^{-2(b+1)} [\chi(-u_1 u_2)  (\q^r)^{-2(a+1)}+(\q-1)\q^b\sum_{e_1=-b}^{-1} \q^{e_1}  (\q^r)^{2e_1}+\chi(-v) \q^b].
\end{align*}
\end{lemma}
\begin{proof}
By the definition of $I(\Z\times \{-b-1\})$ and Lemma \ref{lem:Gaussintegral}, we have
\begin{align*}
    &I(\Z\times \{-b-1\})\\
    =& \sum_{\epsilon_1, \epsilon_2} \sum_{e_1\in \Z}
    \frac{1}{\q} [\frac{1}{\q}\chi(-\epsilon_2 u_2)\mathfrak{g}-\frac{1}{\q}]\chi(-v)\q^{-2b-1} (\q^r)^{-2(b+1)}\\
    &\cdot J^*(-\epsilon_1 u_1 \varpi_0^{a+e_1}) J(-\epsilon_1 u_2\varpi_0^{b+e_1}) J(v\epsilon_1 \varpi_0^{e_1})J(\epsilon_1 \varpi_0^{e_1}) (\q^r)^{2min\{e_1,0\}} \cdot\frac{1}{\alpha(e_1,-b-1)}.
\end{align*}
Since $\sum_{\epsilon_2}\chi(\epsilon_2)=0$, we know the above is equal to
\begin{align}\label{eq:I Z -b-1}
   & - 2 \chi(-v)\q^{-2b-3} (\q^r)^{-2(b+1)} (\sum_{e_1=-a-1}^{-a-1}+\sum_{e_1=-a}^{-b-2}+\sum_{e_1=-b-1}^{-b-1}+\sum_{e_1=-b}^{-1}+\sum_{e_1=0}^{\infty})\\ \nonumber
    \cdot&  \sum_{\epsilon_1} J^*(-\epsilon_1 u_1 \varpi_0^{a+e_1}) J(-\epsilon_1 u_2\varpi_0^{b+e_1}) J(v\epsilon_1 \varpi_0^{e_1})J(\epsilon_1 \varpi_0^{e_1}) (\q^r)^{2min\{e_1,0\}} \cdot\frac{1}{\alpha(e_1,-b-1)}.
\end{align}
Since $\sum_{\epsilon_1}\chi(\epsilon_1)=0$, we know that
\begin{align*}
    &\sum_{e_1=-a-1}^{-a-1} \sum_{\epsilon_1} J^*(-\epsilon_1 u_1 \varpi_0^{a+e_1}) J(-\epsilon_1 u_2\varpi_0^{b+e_1}) J(v\epsilon_1 \varpi_0^{e_1})J(\epsilon_1 \varpi_0^{e_1}) (\q^r)^{2min\{e_1,0\}} \cdot\frac{1}{\alpha(e_1,-b-1)}\\ 
    =&\sum_{\epsilon_1} [\frac{1}{\q}\chi(-\epsilon_1 u_1)\mathfrak{g}-\frac{1}{\q}] \q^{b-a-1}\chi(-\epsilon_1 u_2)\mathfrak{g} \chi(-v) \cdot \q^{-2a-1}\cdot (\q^r)^{-2(a+1)}
    \cdot \frac{1}{4\q^{-(b+1)-3(a+1)-1}}\\ 
    =&\frac{1}{2}\q^{2b+3} \chi(u_1 u_2 v)  (\q^r)^{-2(a+1)}.
\end{align*}
And if $a\geq b+2$ we know that (otherwise the following calculation is not needed)
\begin{align*}
    &\sum_{e_1=-a}^{-b-2} \sum_{\epsilon_1} J^*(-\epsilon_1 u_1 \varpi_0^{a+e_1}) J(-\epsilon_1 u_2\varpi_0^{b+e_1}) J(v\epsilon_1 \varpi_0^{e_1})J(\epsilon_1 \varpi_0^{e_1}) (\q^r)^{2min\{e_1,0\}} \cdot\frac{1}{\alpha(e_1,-b-1)}\\
    =& \sum_{e_1=-a}^{-b-2} \sum_{\epsilon_1} (1-\frac{1}{\q}) \q^{b+e_1}\chi(-\epsilon_1 u_2) \mathfrak{g} \chi(-v) \q^{2e_1+1} (\q^r)^{2e_1} \cdot\frac{1}{\alpha(e_1,-b-1)}\\
    =&0.
\end{align*}
Similarly if $a>b $ we know that (otherwise the following calculation is not needed)
\begin{align*}
    \sum_{e_1=-b-1}^{-b-1} \sum_{\epsilon_1} J^*(-\epsilon_1 u_1 \varpi_0^{a+e_1}) J(-\epsilon_1 u_2\varpi_0^{b+e_1}) J(v\epsilon_1 \varpi_0^{e_1})J(\epsilon_1 \varpi_0^{e_1}) (\q^r)^{2min\{e_1,0\}} \cdot\frac{1}{\alpha(e_1,-b-1)}=0.
\end{align*}
By Lemma \ref{lem:guasssumsquare}, \ref{lem:Gaussintegral} and \eqref{eq: summation of epsilon},
We also know that 
\begin{align*}
    &\sum_{e_1=-b}^{-1} \sum_{\epsilon_1} J^*(-\epsilon_1 u_1 \varpi_0^{a+e_1}) J(-\epsilon_1 u_2\varpi_0^{b+e_1}) J(v\epsilon_1 \varpi_0^{e_1})J(\epsilon_1 \varpi_0^{e_1}) (\q^r)^{2min\{e_1,0\}} \cdot\frac{1}{\alpha(e_1,-b-1)}\\
    =& \sum_{e_1=-b}^{-1} \frac{1}{2}  (1-\frac{1}{\q}) \q^{e_1+3b+4} \chi(-v) (\q^r)^{2e_1},
\end{align*}
and
\begin{align*}
    &\sum_{e_1=0}^{\infty} \sum_{\epsilon_1} J^*(-\epsilon_1 u_1 \varpi_0^{a+e_1}) J(-\epsilon_1 u_2\varpi_0^{b+e_1}) J(v\epsilon_1 \varpi_0^{e_1})J(\epsilon_1 \varpi_0^{e_1}) (\q^r)^{2min\{e_1,0\}} \cdot\frac{1}{\alpha(e_1,-b-1)}\\
    =&\frac{1}{2} \q^{3b+3}.
\end{align*}
Plug the above five terms back into \eqref{eq:I Z -b-1}, the lemma is proved.
\end{proof}

\begin{lemma}\label{lem:I Z e_2 <0}
For $-b\leq e_2\leq -1$, we have 
\begin{align*}
    &I(\Z\times \{e_2\})\\
    =& (\q^r)^{2e_2}[(\q-1)\chi(-u_1 u_2 ) \q^{b+1+e_2} \q^{-2(a+1)r}+ (\q-1)^2\sum_{e_1=-b}^{e_2}  \q^{e_2-e_1} (\q^r)^{2e_1}\\
    +& (\q-1)^2\sum_{e_1=e_2+1}^{-1}  \q^{e_1-e_2-1} (\q^r)^{2e_1}+(\q-1)\chi(-v) \q^{-e_2-1}].
\end{align*}
\end{lemma}
\begin{proof}
Apply Lemma \ref{lem:guasssumsquare}, \ref{lem:Gaussintegral} and \eqref{eq: summation of epsilon} again, we have
\begin{align*}
     &I(\Z\times \{e_2\})\\
    = & 2\frac{1}{\q} (1-\frac{1}{\q}) \chi(-v) \q^{2e_2+1} (\q^r)^{2e_2} (\sum_{e_1=-a-1}^{-a-1}+\sum_{e_1=-a}^{-b-2}+\sum_{e_1=-b-1}^{-b-1}+\sum_{e_1=-b}^{e_2}+\sum_{e_1=e_2+1}^{-1}+\sum_{e_1=0}^{\infty})\\
    &\sum_{\epsilon_1} J^*(-\epsilon_1 u_1 \varpi_0^{a+e_1}) J(-\epsilon_1 u_2\varpi_0^{b+e_1}) J(v\epsilon_1 \varpi_0^{e_1})J(\epsilon_1 \varpi_0^{e_1})(\q^r)^{2min\{e_1,0\}}\cdot\frac{1}{\alpha(e_1,e_2)}.
\end{align*}
Denote the summand of the last equation by $s(e_1,e_2)$. As in the proof of Lemma 
\ref{lem: I Z -b-1}, we know that
\begin{align*}
    \sum_{e_1=-a-1}^{-a-1} \sum_{\epsilon_1} s(e_1,e_2)=&\frac{1}{2} \chi(u_1 u_2 v) \q^{b+2-e_2} \q^{-2(a+1) r},\\
    \sum_{e_1=-a}^{-b-2} \sum_{\epsilon_1} s(e_1,e_2)=&0,\\
    \sum_{e_1=-b-1}^{-b-1} \sum_{\epsilon_1} s(e_1,e_2)=&0,\\
    \sum_{e_1=-b}^{e_2} \sum_{\epsilon_1} s(e_1,e_2)=&\sum_{e_1=-b}^{e_2}\frac{1}{2}(\q-1) \chi(-v) \q^{-e_1-e_2+1} (\q^r)^{2e_1},\\
    \sum_{e_1=e_2+1}^{-1} \sum_{\epsilon_1} s(e_1,e_2)=&\sum_{e_1=e_2+1}^{-1} \frac{1}{2}(\q-1) \chi(-v) \q^{e_1-3e_2} (\q^r)^{2e_1},\\
    \sum_{e_1=0}^{\infty} \sum_{\epsilon_1}s(e_1,e_2)=&\frac{1}{2}\q^{-3e_2}.
\end{align*}
Combining all the terms above finishes the proof of the lemma.
\end{proof}

\begin{lemma}\label{lem:I Z e_2 geq 0}
\begin{align*}
    I(\Z\times \Z_{\geq 0})=\chi(u_1 u_2 v)\q^{b+1} \q^{-2(a+1) r}+\sum_{e_1=-b}^{-1}(\q-1) \chi(-v) \q^{-e_1} (\q^r)^{2e_1}+\frac{\q}{\q+1}.
\end{align*}
\end{lemma}
\begin{proof}
Assume now $e_2\geq 0$.
Apply Lemma \ref{lem:guasssumsquare}, \ref{lem:Gaussintegral} and \eqref{eq: summation of epsilon} again, we have
\begin{align*}
    &I(\Z\times \{e_2\})\\
    =& 2\frac{1}{\q}(1-\frac{1}{\q})  (\sum_{e_1=-a-1}^{-a-1}+\sum_{e_1=-a}^{-b-2}+\sum_{e_1=-b-1}^{-b-1}+\sum_{e_1=-b}^{-1}+\sum_{e_1=0}^{e_2}+\sum_{e_1=e_2+1}^{\infty})\sum_{\epsilon_1}   \\
    &\cdot J^*(-\epsilon_1 u_1 \varpi_0^{a+e_1}) J(-\epsilon_1 u_2\varpi_0^{b+e_1}) 
    J(v\epsilon_1 \varpi_0^{e_1})J(\epsilon_1 \varpi_0^{e_1}) (\q^r)^{2min\{e_1,0\}} \cdot\frac{1}{\alpha(e_1,e_2)}.
\end{align*}
Again denote the summand of the last equation by $s(e_1,e_2)$
The first four terms of the summation have been computed in the previous two lemmas.  We also have
\begin{align*}
    \sum_{e_1=0}^{e_2}  \sum_{\epsilon_1} s(e_1,e_2)=&\frac{1}{2} (\q-1)\q^{-e_2}\frac{1-\q^{-3(e_2+1)}}{1-\q^{-3}},\\
    \sum_{e_1=e_2+1}^{\infty}  \sum_{\epsilon_1} s(e_1,e_2)=&\frac{1}{2}\q^{-4e_2-1}.
\end{align*}
We hence have
\begin{align*}
    I(\Z\times \{e_2\})=&\frac{1}{\q}(1-\frac{1}{\q}) [\chi(u_1 u_2 v) \q^{b+2-e_2} \q^{-2(a+1) r}\\
    +&\sum_{e_1=-b}^{-1}(\q-1) \chi(-v) \q^{-e_1-e_2+1} (\q^r)^{2e_1}
    + (\q-1)\q^{-e_2}\frac{1-\q^{-3(e_2+1)}}{1-\q^{-3}}+\q^{-4e_2-1}].
\end{align*}
Summing over $e_2\in \Z_{\geq 0}$ and a further simplification proves the lemma.
\end{proof}

Combine \eqref{eq: sum of alpha(e) terms}, and Lemma \ref{lem: I Z -b-1}, \ref{lem:I Z e_2 <0} and \ref{lem:I Z e_2 geq 0}, we have
\begin{align}\label{eq:alpha S,T,X initial expression}
    &\alpha(S,T,X)\\ \nonumber
    =&(\q-1) \sum_{e=1}^{2b+1} X^e+1-(\q-1)\q^b\sum_{e=1}^{b} \q^{-e}  X^{e+b+1}\\ \nonumber
    +&\sum_{e_2=1}^{b}(\q-1)^2\sum_{e_1=e_2}^{b}  \q^{e_1-e_2} X^{e_1+e_2}+\sum_{e_2=2}^{b}(\q-1)^2\sum_{e_1=1}^{e_2-1}  \q^{-e_1+e_2-1} X^{e_1+e_2}\\ \nonumber
    -&  \chi(-u_1 u_2)  X^{a+b+2}
    +\sum_{e=1}^{b} (\q-1) \chi(-u_1 u_2 ) \q^{b+1-e} X^{a+1+e}\\ \nonumber
    +&\sum_{e=1}^{b}(\q-\frac{1}{\q}) \chi(-v) \q^{e} X^{e}-\chi(-v) \q^b X^{b+1}
    + \chi(u_1 u_2 v)\q^{b+1} X^{a+1}.
\end{align}
Define
\[A=\sum_{e_2=1}^{b}(\q-1)^2\sum_{e_1=e_2}^{b}  \q^{e_1-e_2} X^{e_1+e_2}, \ B=\sum_{e_2=2}^{b}(\q-1)^2\sum_{e_1=1}^{e_2-1}  \q^{-e_1+e_2-1} X^{e_1+e_2}.\]
\begin{lemma}\label{lem:A+B}
For $2\leq k \leq b+1$, the $X^k$-coefficient of $A+B$ is $(\q-1)(\q^{k-1}-1)$.
For $b+2\leq k \leq 2b$, the $X^k$-coefficient of $A+B$ is $(\q-1)(\q^{2b-k+1}-1)$.
\end{lemma}
\begin{proof}
Define
\[\Lambda_A=\{(e_1,e_2)\in \Z^2 \mid e_2\leq e_1 \leq b, 1\leq e_2 \leq b\},\ 
\Lambda_B=\{(e_1,e_2)\in \Z^2 \mid 1\leq e_1 \leq e_2-1, 2\leq e_2 \leq b\},\]
and 
\[\Lambda_A(k)=\{(e_1,e_2)\in \Z^2 \mid e_1+e_2=k\}\cap \Lambda_A,\ \Lambda_B(k)=\{(e_1,e_2)\in \Z^2 \mid e_1+e_2=k\}\cap \Lambda_B. \]
We then have
\begin{align*}
    \Lambda_A(k)=\{(k-1,1),(k-2,2),\ldots, (\lfloor\frac{k+1}{2}\rfloor,\lfloor\frac{k}{2}\rfloor)\},&  \text{ for } 2\leq k \leq b+1,\\
    \Lambda_A(k)=\{(b,k-b),(b-1,k-b+1),\ldots, (\lfloor\frac{k+1}{2}\rfloor,\lfloor\frac{k}{2}\rfloor)\}, & \text{ for } 
    b+2\leq k \leq 2b,\\
    \Lambda_B(k)=\{(1,k-1),(2,k-2),\ldots,(\lfloor\frac{k-1}{2}\rfloor,\lfloor\frac{k}{2}\rfloor+1)\}, & \text{ for } 3\leq k \leq b+1,\\
    \Lambda_B(k)=\{(k-b,b),(k-b+1,b-1),\ldots,(\lfloor\frac{k-1}{2}\rfloor,\lfloor\frac{k}{2}\rfloor+1)\}, & \text{ for } b+2 \leq k\leq 2b-1.
\end{align*} 
Since 
\[A=\sum_{\Lambda_A} (\q-1)^2 \q^{e_1-e_2} X^{e_1+e_2}, \quad B=\sum_{\Lambda_B} (\q-1)^2 \q^{e_1-e_2} X^{e_1+e_2}, \]
the lemma follows easily.
\end{proof}

Apply the above Lemma and simplify \eqref{eq:alpha S,T,X initial expression}, we get
\begin{align}\label{eq:TDSUformula}
    &\alpha(S,T,X)\\ \nonumber
    =&1+(\q-1)\sum_{e=1}^{b+1} \q^{e-1}X^e
    -\chi(-u_1 u_2)  X^{a+b+2}
    +\sum_{e=1}^{b} (\q-1)\chi(-u_1 u_2 ) \q^{b+1-e} X^{a+1+e}\\ \nonumber
    +&\sum_{e=1}^{b}(\q-\frac{1}{\q}) \chi(-v) \q^{e} X^{e}-\chi(-v) \q^b X^{b+1}+ \chi(u_1 u_2 v)\q^{b+1} X^{a+1}. \nonumber
\end{align}
Using the identity
\begin{equation}\label{eq:difference identity}
   (1-X)\sum_{e=0}^b Y^e X^e=(\q-1)\sum_{e=1}^b Y^{e-1} X^e+1-Y^b X^{b+1} 
\end{equation}
with $Y=\q$ and $\frac{1}{\q}$, it is easy to see that the formula \eqref{eq:TDSUformula} is the same as the first formula in Theorem \ref{thm:localdensity calculation}. This finishes the proof. \qedsymbol

\subsection{Proof of \eqref{eq:local density third formula}}
According to Theorem \ref{thm:Hironakaformula},  Lemma \ref{lem:alpha Gamma}, \ref{lem:G_Gamma T diagonal} and \ref{lem:G(Y,S)}, we have
\begin{align*}
    &\alpha(S_r,T)\\
    =&\sum_{\epsilon_1, \epsilon_2}\sum_{e_1\in \Z} \sum_{e_2\in \Z} \frac{1}{\q}(\q^{min\{a+2 e_1+1,0\}}-\q^{min\{a+2e_1+2,0\}-1})(\q^{min\{a+2 e_2+2,0\}}-\q^{min\{a+2e_2+3,0\}-1})\\
    \cdot & J(v \epsilon_1 \varpi_0^{e_1})J(v\epsilon_2 \varpi_0^{e_2}) J( \epsilon_1 \varpi_0^{e_1})J(\epsilon_2 \varpi_0^{e_2}) (\q^r)^{2min\{e_1,0\}+2min\{e_2,0\}}\cdot \alpha(e_1,e_2)\\
    +& \sum_{e\in \Z} \frac{1}{\q} (\q^{min\{a+e+1,0\}}-2\q^{min\{a+e+2,0\}-1}+\q^{min\{a+e+3,0\}-2}) \\
    \cdot & \q^{2min\{e+1,0\}}(\q^r)^{2 min\{e,0\}}\cdot \frac{1}{\q^{2e-1}(\q-1)}.
\end{align*}
The second summation is easy to compute, it is 
\begin{align*}
    & (\sum_{e=-a-2}^{-a-2}+\sum_{e=-a-1}^{-1}+\sum_{e=0}^\infty) \frac{1}{\q} (\q^{min\{a+e+1,0\}}-2\q^{min\{a+e+2,0\}-1}+\q^{min\{a+e+3,0\}-2}) \\
    \cdot & \q^{2min\{e+1,0\}}(\q^r)^{2 min\{e,0\}}\cdot \frac{1}{\q^{2e-1}(\q-1)}\\
    =&-(\q^r)^{-2a-4}+(\q-1)\sum_{e=-a-1}^{-1} (\q^r)^{2e}+\frac{1}{\q+1}.
\end{align*}

Again for a subset $\Lambda \subseteq \Z^2$ define 
\begin{align*}
    I(\Lambda)
    :=&\sum_{\epsilon_1, \epsilon_2}\sum_{(e_1,e_2)\in \Lambda} \frac{1}{\q}(\q^{min\{a+2 e_1+1,0\}}-\q^{min\{a+2e_1+2,0\}-1})(\q^{min\{a+2 e_2+2,0\}}-\q^{min\{a+2e_2+3,0\}-1})\\
    \cdot & J(v \epsilon_1 \varpi_0^{e_1})J(v\epsilon_2 \varpi_0^{e_2}) J( \epsilon_1 \varpi_0^{e_1})J(\epsilon_2 \varpi_0^{e_2}) (\q^r)^{2min\{e_1,0\}+2min\{e_2,0\}}\cdot \alpha(e_1,e_2).
\end{align*}
The summand is nonzero if and only if $a+2e_2+2\geq 0$ and $a+2e_1+1\geq 0$, or equivalently $e_2\geq -\left \lfloor{\frac{a}{2}}\right \rfloor -1=-\frac{a+1}{2}$ and $e_1\geq -\left \lfloor{\frac{a+1}{2}}\right \rfloor=-\frac{a+1}{2}$ since $a$ is odd.

Assume that $a> 0$, otherwise the following is not needed and can be skipped. 
\begin{align}\label{eq:TASU I a/2-1}
    & I( \{ -\frac{a+1}{2} \}\times \Z)\\ \nonumber
    =&\sum_{e_2\in \Z} \sum_{\epsilon_1, \epsilon_2} \frac{1}{\q}(1-\frac{1}{\q})
    \q^{-\left \lfloor{\frac{a}{2}}\right \rfloor-1}\cdot
    \chi(v\epsilon_1) \mathfrak{g}\cdot
    \q^{-\left \lfloor{\frac{a}{2}}\right \rfloor-1}\cdot
     \chi(\epsilon_1) \mathfrak{g} \cdot(\q^r)^{-2\left \lfloor{\frac{a}{2}}\right \rfloor-2}\\ \nonumber
    \cdot &(\q^{min\{a+2 e_2+2,0\}}-\q^{min\{a+2e_2+3,0\}-1})
     J(v \epsilon_2 \pi_0^{e_2}) J( \epsilon_2 \pi_0^{e_2}) (\q^r)^{2min\{e_2,0\}}\cdot \alpha(-\frac{a+1}{2},e_2)\\ \nonumber
     =& 2 (1-\frac{1}{\q})
    \q^{-a-1} \chi(-v)
     (\q^r)^{-a-1} (\sum_{e_2=-\frac{a+1}{2}}^{-1}+\sum_{e_2=0}^\infty)\\ \nonumber
     & \cdot\sum_{\epsilon_2}(\q^{min\{a+2 e_2+2,0\}}-\q^{min\{a+2e_2+3,0\}-1})
     J(v \epsilon_2 \varpi_0^{e_2}) J( \epsilon_2 \varpi_0^{e_2}) (\q^r)^{2min\{e_2,0\}}\cdot \alpha(-\frac{a+1}{2},e_2)\\ \nonumber
    =&\sum_{e_2=-\frac{a+1}{2}}^{-1} (\q-1)^2 \q^{e_2+\left \lfloor{\frac{a}{2}}\right \rfloor+1} 
     (\q^r)^{2e_2-a-1}+(\q-1)\q^{\left \lfloor{\frac{a}{2}}\right \rfloor+1} \chi(-v) (\q^r)^{-a-1}.
\end{align}

Assume that $a> 0$, otherwise the following calculation is not needed and can be skipped. For $-\left \lfloor{\frac{a}{2}}\right \rfloor \leq e_1\leq -1$, we have
\begin{align}\label{eq:TASU I e_1 Z}
    & I(\{e_1\}\times \Z)\\ \nonumber
    =& 2 \frac{1}{\q}(1-\frac{1}{\q})\q^{2e_1+1} \chi(-v) (\q^r)^{2e_1}  (\sum_{e_2=-\left \lfloor{\frac{a}{2}}\right \rfloor-1}^{e_1-1}+\sum_{e_2=e_1}^{-1}+\sum_{e_2=0}^\infty) \sum_{ \epsilon_2}\\ \nonumber
    \cdot &(\q^{min\{a+2 e_2+2,0\}}-\q^{min\{a+2e_2+3,0\}-1})
    \cdot J(v\epsilon_2 \varpi_0^{e_2}) J(\epsilon_2 \varpi_0^{e_2}) (\q^r)^{2min\{e_2,0\}}\cdot \alpha(e_1,e_2)\\ \nonumber
    =&\sum_{e_2=-\left \lfloor{\frac{a}{2}}\right \rfloor-1}^{e_1-1} (\q-1)^2\q^{e_1-e_2-1} (\q^r)^{2e_1+2e_2}+ \sum_{e_2=e_1}^{-1} (\q-1)^2\q^{e_2-e_1} (\q^r)^{2e_1+2e_2}+(\q-1)\q^{-e_1}  \chi(-v) (\q^r)^{2e_1}.
\end{align}

Now consider $e_1\geq 0$. Then 
\begin{align*}
    &I(\{e_1\}\times \Z)\\ 
    =& 2\frac{1}{\q}(1-\frac{1}{\q}) (\sum_{e_2=-\left \lfloor{\frac{a}{2}}\right \rfloor-1}^{-1}+\sum_{e_2=0}^{e_1-1}+\sum_{e_2=e_1}^\infty)\sum_{ \epsilon_2} \\ 
    \cdot & (\q^{min\{a+2 e_2+2,0\}}-\q^{min\{a+2e_2+3,0\}-1}) J(v\epsilon_2 \varpi_0^{e_2}) J(\epsilon_2 \varpi_0^{e_2}) (\q^r)^{2min\{e_2,0\}}\cdot \alpha(e_1,e_2)\\ 
    =& \frac{1}{\q} (1-\frac{1}{\q})[\sum_{e_2=-\left \lfloor{\frac{a}{2}}\right \rfloor-1}^{-1}  (\q-1) \q^{-e_1-e_2} \chi(-v) (\q^r)^{2e_2}+(1-\frac{1}{\q}) \q^{-e_1}\frac{1-\q^{-3e_1}}{1-\q^{-3}}+ \q^{-4e_1+1}].
\end{align*}
Hence we get 
\begin{equation}\label{eq:TASU Z geq0 Z}
     I(\Z_{\geq 0}\times \Z)=\sum_{e_2=-\left \lfloor{\frac{a}{2}}\right \rfloor-1}^{-1} (\q-1) \q^{-e_2-1} \chi(-v) (\q^r)^{2e_2}+\frac{\q}{\q+1}.
\end{equation}

As in Lemma \ref{lem:A+B}, for $a>0$, we can show that
\begin{align}\label{eq:TASU simplify step}
    &\sum_{e_1=1}^{\left \lfloor{\frac{a}{2}}\right \rfloor}\sum_{e_2=e_1+1}^{\left \lfloor{\frac{a}{2}}\right \rfloor+1}(\q-1) \q^{e_2-e_1-1} X^{e_1+e_2}+\sum_{e_1=1}^{\left \lfloor{\frac{a}{2}}\right \rfloor}\sum_{e_2=1}^{e_1}(\q-1) \q^{e_1-e_2}X^{e_1+e_2}\\ \nonumber
    =&\sum_{e=2}^{\left \lfloor{\frac{a}{2}}\right \rfloor+1} (\q^{e-1}-1) X^e+\sum_{e=\left \lfloor{\frac{a}{2}}\right \rfloor+2}^{a} (\q^{a-e+1}-1) X^e.
\end{align}

Combining \eqref{eq:TASU I a/2-1}, \eqref{eq:TASU I e_1 Z}, \eqref{eq:TASU Z geq0 Z} and \eqref{eq:TASU simplify step}, we have that
for all $a\geq -1$
\[
     \alpha(S,T,X)
    =-X^{a+2}+\sum_{e=\frac{a+1}{2}+1}^{a+1} (\q-1) \q^{a+2-e} X^{e}+(\q-1)\sum_{e=1}^{\frac{a+1}{2}} \q^{e-1} X^e+1+\sum_{e=1}^{ \frac{a+1}{2}}(\q^2-1)\chi(-v)\q^{e-1}X^{e}.
\]
Using the identity \eqref{eq:difference identity} with $Y=\q$ and $\frac{1}{\q}$, we see that the above is the same with the third formula in Theorem \ref{thm:localdensity calculation}. \qedsymbol

\subsection{Proof of \eqref{eq:local density second formula} and \eqref{eq:local density fourth formula}}
For $T\in \Herm_n(F_0), S\in \Herm_m(F_0)$,
we define primitive local densities as in \cite{kitaoka1983note}:
\begin{align}
    \beta(S, T)&=\q^{\ell n (n-2m)}\cdot\\ \nonumber
    & |\{ X \in M_{m,n}(\Oo_{F}/\pi_0^\ell):\mathrm{rank}_{\Oo_F/\pi\Oo_F}(X)=n,  S[X]-T \in  \pi_0^\ell \cdot \Herm^\vee_n(\Oo_{F_0})\}|
\end{align}
for sufficiently large $\ell$. We define the primitive local density polynomial
$\beta(S,T,X)\in \Q[X]$ by 
\begin{equation}
    \beta(S,T,q^{-2r})=\beta(S_r,T).
\end{equation}
The following is the analogue of \cite[Theorem 1]{kitaoka1983note} in the Hermitian form case.
\begin{proposition}\label{prop:induction on fundamental invariants}
Define $U_n=\mathrm{GL}_n(\Oo_F)$ and $ \pi_{n,i}$ to be the double coset $U_n (\pi I_i\oplus I_{n-i}) U_n$ in $\GL_n(F)$. Then we have
\begin{align*}
			\alpha(S, T, X)
			= \sum_{i=1}^{n} (-1)^{i-1} \q^{i(i-1)/2+i(n-m)}X^{i}
		 \cdot \sum_{g\in U_n\backslash \pi_{n,i}} \alpha(S,T[g^{-1}],X) + \beta(S,T,X).
		\end{align*}
\end{proposition}
We now assume again that $m=n=2$.
\begin{lemma} Assume that $T$ is integral or $T=\cH$.
\begin{align*}
	\beta(\cH,T,X)=\begin{cases}
				(1-q^{-2}X),& \text{ if $T$ is equivalent to } \cH,\\
				(1-X)(1-\q^{-2}X), & \text{ otherwise}.
	\end{cases}
\end{align*}
\end{lemma}\label{lem:primitive local density}
\begin{proof}
This is a special case of \cite[Lemma 2.16]{LL2}. We simply evaluate the summand of their formula at $s+1$ and $L=L'$, and make the substitution $X=\q^{-2s}$.
\end{proof}

We define the fundamental invariant of $T$ following \cite[Definition 2.12]{LL2}. If $T$ is equivalent to $\mathrm{Diag}\{u_1 (-\pi_0)^a,u_2 (-\pi_0)^b\}$ with $u_1,u_2\in \Oo_{F_0}^\times$ and $a\geq b$, we define its fundamental invariant to be $i(T)=(2a+1,2b+1)$. If $T$ is equivalent to $\begin{pmatrix}
0 & \pi^{2a-1} \\
(-\pi)^{2a-1} & 0
\end{pmatrix}$, we define its fundamental invariant to be $i(T)=(2a,2a)$. The set of fundamental invariants is partially ordered by the partial order inherited from $\Z^2$.

Our proof is based on induction on $i(T)$. 
Notice that $i(T[g^{-1}])<i(T)$ for any $g\in U_2\backslash \pi_{2,i}$ ($i=1,2$). If $i(T)=(s,t)$ with $s\geq t$ and $t<0$, then $\alpha(S,T,X)=0$. By Proposition \ref{prop:induction on fundamental invariants}, it suffices to prove that
\begin{equation}\label{eq:induction on fi formula}
``\alpha(\cH,T,X)"-\sum_{i=1}^{2} (-1)^{i-1} \q^{i(i-1)/2}X^{i}
		 \cdot \sum_{U_n\backslash \pi_{n,i}} ``\alpha(\cH,T[g^{-1}],X)"= \beta(\cH,T,X)
\end{equation}
for any $T$ such that $i(T)\geq (0,0)$,
where $``\alpha(\cH,T,X)"$ is the claimed formula of $\alpha(\cH,T,X)$ in Theorem \ref{thm:localdensity calculation}.

\noindent 
Case 1. We first prove the case $\chi(T)=-1$. In this case we can assume  $T=\mathrm{Diag}\{u_1 (-\pi_0)^a,u_2 (-\pi_0)^b\}$ with $u_1,u_2\in \Oo_{F_0}^\times$ and $a\geq b$. Define
\[
    \mu(a,b,X)=\frac{``\alpha(\cH,T,X)"}{1-q^{-2}X}=\begin{cases}
    \sum_{e=0}^b\q^e X^e-\sum_{e=a+1}^{a+b+1} \q^{a+b+1-e} X^e,& \text{ if } b\geq 0 \\
    0, & \text{ if } b< 0, 
    \end{cases}
\]
where $T\in \Herm_2(F_0)$ such that $\chi(T)=-1$ and $i(T)=(2a+1,2b+1)$.
The set $U_2\backslash \pi_{2,1}$ is represented by $\q+1$ elements
\begin{equation}
    \gamma_\infty=\begin{pmatrix}
    \pi & 0\\
    0 & 1
    \end{pmatrix},
    \gamma_x=\begin{pmatrix}
    1 & x\\
    0 & \pi
    \end{pmatrix}, \quad x\in \Oo_F/\pi \Oo_F.
\end{equation}
The set $U_2\backslash \pi_{2,2}$ is the singleton represented by $\mathrm{Diag}\{\pi,\pi\}$. According to the preceding argument, it suffices to prove that in each the following sub cases, \eqref{eq:induction on fi formula} holds.
\begin{enumerate}[label=(\alph*)]
    \item $i(T)=(2a+1,2b+1)$ with $a>b>0$.
    \item $i(T)=(2a+1,2b+1)$ with $a>b=0$.
    \item $i(T)=(2a+1,2a+1)$ with $a>0$.
    \item $i(T)=(2a+1,2a+1)$ with $a=0$.
\end{enumerate}
If $i(T)=(2a+1,2b+1)$ with $a>b$, then 
\[i(T[\gamma_\infty^{-1}])=(2a-1,2b+1), \quad i(T[\gamma_x^{-1}])=(2a+1,2b-1).\]
In this case we need to show that 
\[
     \mu(a,b,X)-\q X\mu(a,b-1,X)-X\mu(a-1,b,X)+\q X^2\mu(a-1,b-1,X)=1-X.
\]
If $i(T)=(2a+1,2a+1)$, then 
\[i(T[\gamma_\infty^{-1}])=(2a+1,2a-1), \quad i(T[\gamma_x^{-1}])=(2a+1,2a-1).\]
In this case we need to show that
\[
     \mu(a,a,X)-(\q+1) X\mu(a,a-1,X)+\q X^2\mu(a-1,a-1,X)=1-X.
\]
In sub case (a), we have
\begin{align*}
    &\mu(a,b,X)-\q X\mu(a,b-1,X)-X\mu(a-1,b,X)+\q X^2\mu(a-1,b-1,X)\\
    =&\sum_{e=0}^b \q^e X^e-\sum_{e=a+1}^{a+b+1} \q^{a+b+1-e} X^e 
    -\sum_{e=1}^b \q^e X^e+\sum_{e=a+2}^{a+b+1} \q^{a+b+2-e} X^e \\
    -&\sum_{e=0}^b \q^e X^{e+1}+\sum_{e=a+1}^{a+b+1} \q^{a+b+1-e} X^e
    +\sum_{e=1}^b \q^e X^{e+1}-\sum_{e=a+2}^{a+b+1} \q^{a+b+2-e} X^e\\
    =&1-X.
\end{align*}
In sub case (b), we have
\begin{align*}
    &\mu(a,b,X)-\q X\mu(a,b-1,X)-X\mu(a-1,b,X)+\q X^2\mu(a-1,b-1,X)\\
    =&\mu(a,0,X)-X\mu(a-1,0)\\
    =&1-X.
\end{align*}
Sub case (c) and (d) are similar and left to the reader.

\noindent 
Case 2. We first prove the case $\chi(T)=1$. In this case $T$ can be either diagonalized or anti-digonalized. According to the fundamental invariant, we define
\[
    \nu(s,t,X)=\begin{cases} \sum_{e=0}^b \q^e X^e+\sum_{e=a+1}^{a+b+1} \q^{a+b+1-e} X^e &\text{ if } s=2a+1,t=2b+1,b\geq 0, \\
     \sum_{e=0}^a \q^e X^e+\sum_{e=a+1}^{2a} \q^{2a-e} X^e &\text{ if } s=t=2a, a\geq 0,\\
    0 & \text{ if } t< 0.
    \end{cases}
\]
Again it suffices to prove that in each the following sub cases, \eqref{eq:induction on fi formula} holds.
\begin{enumerate}[label=(\alph*)]
    \item $i(T)=(2a+1,2b+1)$ with $a>b>0$.
    \item $i(T)=(2a+1,2b+1)$ with $a>b=0$.
    \item $i(T)=(2a+1,2a+1)$ with $a>0$.
    \item $i(T)=(1,1)$.
    \item $i(T)=(2a,2a)$ with $a>0$.
    \item $i(T)=(0,0)$.
\end{enumerate}
If $i(T)=(2a+1,2b+1)$ with $a>b$ we assume $T=\mathrm{Diag}\{u_1 (-\pi_0)^a,u_2 (-\pi_0)^b\}$ with $u_1,u_2\in \Oo_{F_0}^\times$. Then 
\[i(T[\gamma_\infty^{-1}])=(2a-1,2b+1), \quad i(T[\gamma_x^{-1}])=(2a+1,2b-1).\]
In this case we need to show that 
\[
     \nu(2a+1,2b+1,X)-\q X\nu(2a+1,2b-1,X)-X\nu(2a-1,2b+1,X)+\q X^2\nu(2a-1,2b-1,X)=1-X.
\]
If $i(T)=(2a+1,2a+1)$, we assume $T=\begin{pmatrix}
0 & \pi^{2a} \\
(-\pi)^{2a} & 0
\end{pmatrix}$.  Then
\[i(T[\gamma_\infty^{-1}])=(2a,2a), \quad i(T[\gamma_x^{-1}])=(2a+1,2a-1) \text{ for } x\neq0, \quad i(T[\gamma_0^{-1}])=(2a,2a) .\]
In this case we need to show that
\[
     \nu(2a+1,2a+1,X)-(\q-1) X\nu(2a+1,2a-1,X)-2X\nu(2a,2a,X)+\q X^2\nu(2a-1,2a-1,X)=1-X.
\]
If $i(T)=(2a,2a)$, we assume $T=\begin{pmatrix}
0 & \pi^{2a-1} \\
(-\pi)^{2a-1} & 0
\end{pmatrix}$.  Then
\[i(T[\gamma_\infty^{-1}])=(2a-1,2a-1), \quad i(T[\gamma_x^{-1}])=(2a-1,2a-1).\]
In this case we need to show that
\[
     \nu(2a,2a,X)-(\q+1) X\nu(2a-1,2a-1,X)+\q X^2\nu(2a-2,2a-2,X)=\begin{cases}
     1-X & \text{ if } a>0, \\
     1 & \text{ if } a=0.
     \end{cases}
\]

In sub case (a), we have
\begin{align*}
    &  \nu(2a+1,2b+1,X)-\q X \nu(2a+1,2b-1,X)-X \nu(2a-1,2b+1,X)+\q X^2  \nu(2a-1,2b-1,X) \\
     =&\sum_{e=0}^b \q^e X^e+\sum_{e=a+1}^{a+b+1} \q^{a+b+1-e} X^e 
    -\sum_{e=1}^b \q^e X^e-\sum_{e=a+2}^{a+b+1} \q^{a+b+2-e} X^e \\
    -&\sum_{e=0}^b \q^e X^{e+1}-\sum_{e=a+1}^{a+b+1} \q^{a+b+1-e} X^e
    +\sum_{e=1}^b \q^e X^{e+1}+\sum_{e=a+2}^{a+b+1} \q^{a+b+2-e} X^e\\
    =&1-X.
\end{align*}
In sub case (b), we have
\begin{align*}
    &  \nu(2a+1,2b+1,X)-\q X \nu(2a+1,2b-1,X)-X \nu(2a-1,2b+1,X)+\q X^2  \nu(2a-1,2b-1,X) \\
     =&  \nu(2a+1,1,X)-X \nu(2a-1,1,X)\\
     =&1-X.
\end{align*}
The rest cases are similar and we leave the calculation to the reader.
\qedsymbol

\part{Global theory}

\section{Intersection of special cycles on unitary Shimura curves}\label{sec:globalintersecionnumber}
We switch our notations in this section. Let $\bk=\Q[\sqrt{-\Delta}]$ be an imaginary quadratic extension of $\Q$ with discriminant $-\Delta$. Let $\chi$ be the character of $\bA^\times$ associated to the extension $F/F_0$. 
Let $\Herm_n$ be the group scheme defined by \eqref{eq:Herm_n global}. For a sub-algebra $R$ of $\R$, we denote by $\Herm_n(R)_{>0}$ the set of positive definite hermitian matrices in  $\Herm_n(R)$.
\subsection{Global moduli problem} In this subsection, we briefly review the definition of an integral model of Shimura variety defined  in \cite{BHKRY1} over $\Spec \Oo_\bk$ . Let 
\[\cM^{\Kra}_{(n-1,1)} \rightarrow \Spec \Oo_\bk\]
be the algebraic stack which assigns to each $\Oo_\bk$-scheme $S$ the groupoid of isomorphism classes of quadruples $(A,\iota,\lambda,\cF_A)$ where 
\begin{enumerate}
    \item $A\rightarrow S$ is an abelian scheme of relative dimension $n$;
    \item $\iota:\Oo_\bk\rightarrow \End(A)$ is an action satisfies the following determinant condition 
    \[\mathrm{char}(T-\iota(\alpha)\mid \Lie A)=(T-s(\alpha))^{n-1}(T-s(\bar{\alpha})) \in \Oo_S[T],\]
    for all $\alpha\in \Oo_\bk$ where $s:\Oo_\bk\rightarrow \Oo_S$ is the structure morphism;
    \item There is a principal polarization $\lambda:A\rightarrow A^\vee $ whose Rosati involution satisfies $\iota(\alpha)^*=\iota(\bar{\alpha})$ for all $\alpha \in \Oo_\bk$;
    \item $\cF_A \subset \Lie A$ is an $\Oo_\bk$-stable $\Oo_S$-module local direct summand of rank $n-1$ satisfying the Kr\"amer condition: $\Oo_\bk$ acts on $\cF_A$ by the structure map $s:\Oo_\bk \rightarrow \Oo_S$ and acts on $\Lie A/\cF_A$ by the complex conjugate of the structure map. 
\end{enumerate}
$\cM^{\Kra}_{(n-1,1)}$ is flat of dimension $n-1$ over $\Spec \Oo_\bk$. It is smooth over $\Spec \Oo_\bk[\frac{1}{\sqrt{-\Delta}}]$ and has semi-stable reduction over primes of $\bk$ dividing $\sqrt{-\Delta}$. This is indicated by the corresponding behaviour of its local model studied in \cite{Kr}. As a special case, $\cM_{(1,0)}$ is the moduli stack of elliptic curve with $\Oo_bk$-action.

Assume that $\F$ is an algebraically closed field of characteristic $p$ over $\bk$. Let
\[(E_0,\iota_0,\lambda_0,A,\iota,\lambda,\cF_A)\in (\cM_{(1,0)}\times \cM^{\Kra}_{(n-1,1)})(\Spec \F).\]
We can define a hermitian form $h(x,y)$ on the relative prime to $p$ Tate module
\begin{equation}
    T^p(E_0,A):=\Hom_{\Oo_\bk}(T^p(E_0),T^p(A))
\end{equation}
as in \cite[Section 2.3]{KR2} using the polarizations $\lambda_0,\lambda$ and Weil pairings on $E_0,A$.

Fix a hermitian space $W$ over $\bk$  of signature $(n-1,1)$ that contains a self-dual lattice $\mathfrak a$ and a hermitian space $W_0$ over $\bk$  of signature $(1,0)$ that contains a self-dual lattice $\mathfrak{a}_0$. Define
\begin{equation}\label{eq:VandL}
    V:=\Hom_\bk (W_0,W), \quad L:=\Hom_{\Oo_\bk}(\mathfrak{a}_0,\mathfrak{a}).
\end{equation}
$V$ and $L$ are equipped with hermitian forms coming from the ones on $W_0$ and $W$. Define $G:=\rU(W)$. Also define the group scheme $\mathrm{GU}(W)$ over $\Q$ by 
\[\mathrm{GU}(W)(R)=\{g\in \GL_R(W\otimes R)\mid (gv,gw)=c(g)(v,w),\forall v,w\in W\otimes R\}\]
where $R$ is any $\Q$-algebra. Also define $Z:=\mathrm{Res}_{\bk/\Q} \bG_m=\mathrm{GU}(W_0)$ and
\begin{equation}\label{eq:tildeG}
    \tilde{G}:=Z\times_{\bG_m} \mathrm{GU}(W)
\end{equation}
where the maps from the factors on the right hand side to $\bG_m$ are $\mathrm{Nm}_{\bk/\Q}$ and the similitude character $c(g)$ respectively. We have an isomorphism of group schemes
\begin{equation}\label{eq:tildeGisomorphism}
   \tilde{G}\rightarrow Z\times \rU(W), (z,g)\mapsto (z,z^{-1}g).
\end{equation}
Let $K_G$ be the compact subgroup of $G(\bA_f)$ that stabilizes the lattice $\mathfrak{a}\otimes \bA_f$ and $K_Z=\hat{\Z}^\times\subset Z(\bA_f)$. Under the isomorphism \eqref{eq:tildeGisomorphism}, define 
\begin{equation}
    K:=K_Z\times K_G.
\end{equation}

Now define
$\cM\subset \cM_{(1,0)}\times \cM^{\Kra}_{(n-1,1)}$
to be the open and closed substack such that $(E_0,\iota_0,\lambda_0,A,\iota,\lambda,\cF_A)\in \cM(S)$ if and only if there is an isomorphism of hermitian $\Oo_\bk\otimes_\Z \bA^p_f$-modules
\begin{equation}
    T^p(E_{0,s},A_s)\cong L\otimes \bA^p_f
\end{equation}
for any geometric point $s\in S$ where $p$ is the characteristic of $s$. If $s$ has characteristic zero, simply use $T(E_{0,s},A_s)$ instead. 
$\cM$ is an integral model of the Shimura variety associated to the group $\tilde{G}$ with level structure defined by $K$.

\subsection{Intersection of special cycles}
Now we review the definition of special cycles. For $(E,\iota_0,\lambda_0, A,\iota, \lambda,\cF_A)\in \cM(S)$ where $S$ is an $\Oo_\bk$-scheme, consider the projective $\Oo_\bk$-module of finite rank
\[V'(E,A)=\Hom_{\Oo_\bk}(E,A).\]
On this module there is a hermitian form $h'(x,y)$ defined by 
\begin{equation}
    h'(x,y)=\iota_0^{-1}(\lambda_0^{-1}\circ y^\vee \circ \lambda \circ x),
\end{equation}
where $y^\vee$ is the dual homomorphism of $y$. It is proved in \cite[Lemma 2.7]{KR2} that $h'(x,y)$ is positive semi-definite. 
\begin{definition}\label{def:globalspecialcycle}
For $T\in \Herm_m(\Z)_{>0}$, the special cycle $\cZ^\Kra(T)$ is the collection  $(E,\iota_0,\lambda_0,A,\iota ,\lambda,\cF_A,\bx)$ where
\[(E,\iota_0,\lambda_0,A,\iota ,\lambda,\cF_A)\in \cM(S)\]
and $\bx=(x_1,\ldots,x_m)\in \Hom_{\Oo_\bk}(E,A)^m$ such that 
\[h'(\bx,\bx)=(h'(x_i,x_j))=T.\]
\end{definition}

For a $T\in \Herm_n(\Z)_{>0}$, let $V_T$ be the hermitian space over $\bk$ with Gram matrix $T$. Define $\mathrm{Diff}(V_T,V)$ as in \eqref{eq:Diff}.
By the Hasse principal, the cardinality of $\mathrm{Diff}(V_T,V)$ is odd.  By \cite[Proposition 2.22]{KR2} or \cite[Proposition 5.4]{Shi1}, $\cZ^\Kra(T)$ is empty if $|\mathrm{Diff}(V_T,V)|>1$ and is supported on $\cM^{ss}_{p}$ if 
$\mathrm{Diff}(V_T,V)=\{p\}$ where $p$ is a prime inert or ramified in $\bk$ and $\cM^{ss}_p$ is the super singular locus of $\cM$ over $\bk_p$.

For a hermitian lattice or vector space $M$, define 
\begin{equation}
    \Omega(T,M)=\{x\in M\mid (x,x)=T\}. 
\end{equation}
For a positive definite hermitian lattice $L$, let $[[L]]$ be its genus. Define the representation number
\begin{equation}\label{eq:rgen}
   r_{\mathrm{gen}}(T,[[L]])=\sum_{M\in [[L]]} |\mathrm{Aut}(M)|^{-1} |\Omega(T,M)| 
\end{equation}
where $\mathrm{Aut}(M)$ is the automorphism group of $M$.

We now specialize to the case when $n=2$. Our main global geometric result is the following theorem.
\begin{theorem}\label{globaintersectionnumber}
Let $p$ be an odd prime of $\Q$ ramified in $\bk$ such that $V_p$ is isotropic. Assume $T\in \Herm_2(\Z)_{>0}$ with $\mathrm{Diff}(V_T,V)=\{p\}$. Let $h_\bk$ be the class number of $\bk$ and $w_\bk$ be the number of units in $\Oo_\bk^\times$. Then 
\[\widehat{\mathrm{deg}}(\cZ^\Kra(T))=\log p\cdot \mu_p(T)\cdot \frac{h_\bk}{ w_\bk} r_{\mathrm{gen}}(T,[[L']]),\]
where $\widehat{\mathrm{deg}}(\cZ^\Kra(T))$ is defined in \eqref{eq:degcZ}, $L'$ is a self-dual lattice in $V_T$ such that $L'\otimes \bA^p_f\cong L\otimes \bA^p_f$ as hermitian spaces and $\mu_p(T)$ is as in Theorem \ref{thm:maintheorem1}.
\end{theorem}
\begin{proof}
The theorem follows from Theorem \ref{thm:maintheorem1} and the standard $p$-adic uniformization theorem for basic locus of Shimura varieties (see for example \cite[Theorem 6.30]{RZ}).   
By our assumption $V_T\otimes_{\Q} \Q_p$ contains a unique integral lattice, hence
the proof of \cite[Theorem 11.8]{KR2} applies to our situation without any change. So we omit the details.
\end{proof}

\section{Derivative of incoherent Eisenstein series}\label{sec:Eisensteinseries}
For the moment, let $(V,(,))$ be any non-degenerate hermitian space over $\bk$ of dimension $m$. Let $G=\rU(V)$ and let $H=\rU(W_{n,n})$ where $(W_{n,n},\langle,\rangle)$ is a split skew hermitian space of dimension $2n$. Let $\mathbb{W}=V\otimes_\bk W_{n,n}$ with the symplectic form 
\[\langle\langle v_1\otimes w_1,v_2 \otimes w_2\rangle\rangle=\mathrm{tr}_{\bk/\Q} ((v_1,v_2)\langle w_1, w_2 \rangle),\]
There is a homomorphism $G\times H\rightarrow \mathrm{Sp}(\mathbb{W})$ and $(G,H)$ is a reductive dual pair.

Fix a character $\eta$ of $\bk_\bA^\times$ whose restriction to $\Q_\bA^\times$ is $\chi^m$ where $\chi$ is the global quadratic character attached to the extension $\bk/\Q$. By \cite{Kudlasplitting}, the choice of $\eta$ determines a homomorphism 
\[G(\bA)\times H(\bA) \rightarrow \mathrm{Mp}(\mathbb{W})(\bA),\]
where $\mathrm{Mp}(\mathbb{W})(\bA)$ is the metaplectic cover of $\mathrm{Sp}(\mathbb{W})(\bA)$, hence a Weil representation $\omega$ of $G(\bA)\times H(\bA) $ on the Schwartz space $\mathcal{S}(V(\bA)^n)$. We normalize this so that the action of $G(\bA)$ is given by 
\[(\omega(g,1)\varphi)(x)=\varphi(g^{-1}x).\]
The Weil representation depends on a choice of a continuous additive character $\psi$ on $ \Q \backslash \bA$. We fix such a choice given by
\begin{equation}
    \psi_\infty(x)=\exp(2\pi i x), \psi_\ell(x)=\exp(-2 \pi i \lambda(x)),
\end{equation}
where $\lambda$ is the canonical map $\Q_\ell \rightarrow \Q_\ell/\Z_\ell \hookrightarrow \Q/\Z$.

We are interested in exactly the same incoherent Eisenstein series considered in \cite[Section 9]{KR2}, whose construction we now briefly recall. From now on assume $V$ is of signature $(n-1,1)$ and contains a self-dual lattice $L$. Let $\varphi_f\in \mathcal{S}(V(\mathbb{A}_f)^n)$ be the characteristic function of $(L\otimes \hat{\Z})^n$. The incoherent Eisenstein series is defined to be $E(h,s,\Phi)$ for $\Phi(h,s)=\Phi_\infty(h,s)\otimes \Phi_f(h,s,L)$ where $\Phi_f(h,s,L)$ is the Siegel-Weil section associated to $\varphi_f$ and $\Phi_\infty(h,s)$ is the Siegel-Weil section attached to the Gaussian of a hermitian space of signature $(n,0)$. We have
\[E(h,s,\Phi)=\sum_{\gamma \in P(\bk)\backslash H(\bk)} \Phi(\gamma h,s),\]
where $P\subset H$ is the Siegel parabolic subgroup. The Eisenstein series has meromorphic continuation to $s\in \C$ and a functional equation relating $s\leftrightarrow -s$. We call $E'(h,0,\Phi)=\frac{\partial}{\partial s} E(h,s,\Phi)|_{s=0}$ its central derivative.

Let us to relate the above Eisenstein series to the classical Eisenstein series. Let $D(W_{n,n})$ be defined by 
\[D(W_{n,n})=\{z\in M_n(\C)\mid v(z):=(2i)^{-1} (z-z^*)>0\}. \]
Write $z=u(z)+i v(z)$ where $u(z)=\frac{1}{2}(z+z^*)$, and let
\[h_z=\left(\begin{array}{cc}
    I_n & u(z) \\
     & I_n
\end{array}\right)\left(\begin{array}{cc}
   a &  \\
     & (a^*)^{-1}
\end{array}\right)\in H(\R),\]
where $a\in \mathrm{GL}_n(\C)$ with $v(z)=a a^*$. Note that $iI_n\in D(W_{n,n})$ and $h_z\cdot iI_n=z$. Now let 
\begin{equation}\label{eq:h(z)}
    h(z)=(h_z,1)\in H(\R)\times H(\bA_f)=H(\bA).
\end{equation}
Define
\begin{equation}\label{classicalEisenstein}
    E(z,s,\Phi)=\eta_\infty(\mathrm{det}(a))^{-1} \mathrm{det}(v(z))^{-\frac{n}{2}} E(h(z),s,\Phi).
\end{equation}
It is the corresponding classical Eisenstein series of $ E(h,s,\Phi)$.

Recall that for a place $v$ of $\Q$ we can define the local Whittaker function 
\begin{equation}\label{eq:W_T,v}
    W_{T,v}(h_v,s,\Phi_v)=\int_{\Herm_n(\Q_v)} \Phi_v(w^{-1}n(b)h_v,s) \psi_v (-\mathrm{tr}(Tb))db,
\end{equation}
where
\[w=\left(\begin{array}{cc}
     & I_n \\
    -I_n & 
\end{array}\right), n(b)=\left(\begin{array}{cc}
    I_n & b \\
     & I_n
\end{array}\right),\]
and $db$ is the self-dual measure with respect to the pairing 
\[\psi_v (\mathrm{tr}(bc))\]
on  $\Herm_n(\Q_v)$.

By a standard unfolding calculation, we know that the $T$-th Fourier coefficient $E_T(h(z),s,\Phi)$ of $E(h(z),s,\Phi)$ for a non-degenerate $T$ is 
\begin{equation}
    E_T(h(z),s,\Phi)=\prod_{v} W_{T,v}(h_v(z),s,\Phi_v),
\end{equation}
where the product is taken over all places of $\Q$, is convergent for large $s$
and has an analytic continuation to $s\in \C$.
\begin{remark}
To ease notations, we write $W_{T,p}(s,\Phi_p)$ for $W_{T,p}(1,s,\Phi_p)$ for a finite prime $p$.
\end{remark}

\subsection{Main global identity} Now we specialize to the case when $n=2$. Let $V$ and $L$ be as in \eqref{eq:VandL}.
We suppose that $T\in \Herm_2(\Z)_{>0}$ with $\mathrm{Diff}(V_T,V)=\{p\}$ for an odd prime $p$ ramified in $\bk$. We assume $V'=V_T$ is anisotropic over $\bk_p$.
Fix an isomorphism $V'(\mathbb{A}_f^p)=V(\mathbb{A}_f^p)$ and let $L'$ be the lattice in $V'$ determined by $L'\otimes \hat{\Z}^p=L\otimes \hat{\Z}^p$ and $L'_p=\Lambda$ where $\Lambda\subset V'_p$ is the unique self-dual lattice. Let $\varphi'_f\in \mathcal{S}((V')^2)$ be the characteristic function of $(L'\otimes \hat{\Z})^2$. Let $\Phi'$ be the Siegel-Weil section associated to $\varphi'_f\otimes\varphi_\infty$. 
 
\begin{proposition}\label{W'PhiandWPhi'}
For $T\in \Herm_2(\Z)_{>0}$ with $\mathrm{Diff}(V_T,V)=\{p\}$ where $p$ is an odd prime ramified in $\bk$. Suppose that under the action of $\mathrm{GL}_2(\Oo_{\bk,p})$, $T$ is equivalent to
\[\left(\begin{array}{cc}
    u_1 (-\Delta)^a &  \\
     & u_2 (-\Delta)^b
\end{array}\right),\]
where $u_1,u_2 \in \Z_p^\times$ with $\chi_p(-u_1 u_2)=-1$. Let 
\[S'=\left(\begin{array}{cc}
    u_1 u_2  &  \\
     & 1
\end{array}\right), S=\left(\begin{array}{cc}
    v  &  \\
     & 1
\end{array}\right),\]
where $v\in \Z_p^\times$ and $ \chi_p(-v)=1$. 
\begin{enumerate}
    \item We have
    \[W'_{T,p}(0,\Phi_p)=\gamma_p(V)^2 p^{-\frac{5}{2}} \alpha_p(S,S) \mu_p(T) \log p,\]
    where the derivative is taken with respect to $s$, the function $\mu_p(T)$ is defined in Theorem \ref{thm:maintheorem1} and $\gamma_p(V)$ is an eighth root of unity defined in equation \eqref{gammapV}. 
    \item We have
    \[W_{T,p}(0,\Phi'_p)=\gamma_p(V)^2 p^{-\frac{5}{2}} \alpha_p(S',S'). \]
\end{enumerate}
\end{proposition} 
We postpone the proof of the above proposition to the next subsection.
Notice that the hermitian form on $L'_p$ can be represented by the Gram matrix $S'$ and the hermitian form on $L_p$ can be represented by the Gram matrix $S$.

Write
\[E(z,s,L)=E(z,s,\Phi)\]
to emphasize the dependence on the choice of the self-dual lattice $L$. 
\begin{proposition}\label{prop:E'_T}
Assume $T\in \Herm_2(\Z)_{>0}$ with $\mathrm{Diff}(V_T,V)=\{p\}$ for an odd prime $p$ ramifieid in $\bk$ and $V_{p}$ is isotropic. Then 
\[E'_T(z,0,L)= C \cdot \mu_p(T) \log p \cdot r_{\mathrm{gen}}(T,[[L']])\cdot q^T, \]
where $L'$ is obtained from $L$ as explained before Proposition \ref{W'PhiandWPhi'}, $q=\exp(2\pi i \mathrm{Tr}(Tz))$ and $C$ is as in \cite[Corollary 9.4]{KR2}. The constant $C$ only depends on $\bk$ and $L$.
\begin{comment}
\[C= \mathrm{vol}(G'(\R),d\nu') \mathrm{vol}(K,d\nu).\]
\end{comment}
\end{proposition}
\begin{proof}
Following Proposition \ref{W'PhiandWPhi'}, the proof is exactly the same as that of \cite[Corollary 9.4]{KR2}.
\end{proof}

We can now state the main global result of the paper.
\begin{theorem}\label{globalmainthm}
Let $p$ be an odd prime ramified in $\bk$ such that $V_{p}$ is isotropic. 
Let $T\in \Herm_2(\Z)_{>0}$.  Assume that $\mathrm{Diff}(V_T,V)=\{p\}$. Then
\[E'_T(z,0,L)=C_1 \cdot \widehat{\mathrm{deg}}(\cZ^\Kra(T)) \cdot q^T, \]
where 
\begin{equation}\label{C_1}
    C_1=\frac{w_k}{h_k} C
\end{equation}
where $C$ is as in Proposition \ref{prop:E'_T}.
\end{theorem}
\begin{remark}
If $T$ is not integral, i.e., $T\in \Herm_2(\Q)\setminus \Herm_2(\Z)$, then both sides of the equation in the above theorem are zero.
\end{remark}
\begin{proof}
The theorem follows easily from Proposition \ref{prop:E'_T} and Theorem \ref{globaintersectionnumber}.
\end{proof}

\subsection{Derivative of Whittaker functions}
The purpose of this subsection is to explain the relation between local Whittaker functions and local densities of hermitian forms for an odd prime $p$ ramified in $\bk$. In particular we prove Proposition \ref{W'PhiandWPhi'}. Let $F:=\bk_p$.

Let $L_{r,r}$ be a $2r$ dimensional hermitian $\Oo_F$-lattice with the hermitian form $(,)_{r,r}$ represented by the Gram matrix
\begin{equation}\label{interpolationmatrix}
    \frac{1}{\sqrt{-\Delta}} \left(\begin{array}{cc}
    0 & I_r \\
    -I_r & 0 
\end{array}\right).
\end{equation}
Let $V_{r,r}:=L_{r,r}\otimes\Q$ and let $\varphi_{r,r}^n$ be the characteristic function of the lattice $L_{r,r}^n\subset V_{r,r}^n$. 
We choose an $\Oo_F$-basis $\{e_1,\ldots,e_n,f_1,\ldots,f_n\}$ of $W_{n,n}$ with respect to which the skew hermitian form on $W_{n,n}$ is represented by the matrix 
$\left(\begin{array}{cc}
    0 & I_n \\
    -I_n & 0 
\end{array}\right)$.
Since $V_{r,r}$ has even dimension, $H:=\rU(W_{n,n})$ acts on $V_{r,r}^n$ by the Weil representation $\omega_r$ determined by the additive character $\psi$ on $\Q\backslash\bA$. Let $P$ be the Siegel parabolic that preserves the flag $\mathrm{span}\{e_1,\ldots,e_n\}\subset W_{n,n}$ and $P=MN$ be its Levi decomposition where 
\[M=\left\{m(a)=\left(\begin{array}{cc}
    a &  \\
     & {}^t \bar{a}^{-1}
\end{array}\right)\mid a\in \mathrm{Res}_{F/F_0}\GL_n\right\},
N=\left\{n(b)=\left(\begin{array}{cc}
    1 & b \\
     & 1
\end{array}\right)\mid b\in \Herm_n\right\}.\]

\begin{lemma}\label{Kinvariance}
The function $\varphi_{r,r}^n$ is $K$-invariant where $K=\rU(W_{n,n})(\Oo_{F_0})$ is a maximal compact subgroup of $\rU(W_{n,n})$.
\end{lemma}
\begin{proof}
By the Bruhat decomposition of $K$, we know that 
\[K=\sqcup_{j=0}^n P(\Oo_{F_0})w_j P(\Oo_{F_0}),\]
where
\[w_j=\left(\begin{array}{cccc}
  I_{n-j} &  &   &  \\
   &  &   & I_j \\
   &  & I_{n-j}  &  \\
  &-I_{j} &    &  
\end{array}\right).\]
Write $x=(x',x'')$ where $x'\in V_{r,r}^{n-j}$ and $x''\in V_{r,r}^j$.
Recall from \cite[Section 5]{Kudlasplitting}:
\begin{equation}\label{eq:w_jaction}
    \omega_r(w_j^{-1}) \varphi_{r,r}^n (x)=\gamma_p(V_{r,r})^{-j} \varphi_{r,r}^{n-j}(x')\int_{V_{r,r}^j} \psi_p(-\mathrm{tr}_{F/\Q_p} (\mathrm{Tr}(x'',y)_{r,r}))\varphi_{r,r}^j(y)dy,
\end{equation}
where the measure $dy$ is defined in a way such that $\mathrm{vol}(L_{r,r}^j,dy)=1$ and $\gamma_p(V_{r,r})$ is defined in the same way as \eqref{gammapV} below. One can show that $\gamma_p(V_{r,r})=1$ in our case.

The key observation is that $L_{r,r}$ is self dual with respect to the symmetric form $\mathrm{tr}_{F/\Q_p}[(,)_{r,r}]$. Hence
\begin{align*}
    \int_{V_{r,r}^j} \psi_p(-\mathrm{tr}_{F/\Q_p}(\mathrm{Tr}(x'',y)_{r,r}))\varphi_{r,r}^j(y)dy= 1_{L_{r,r}^j}(x'').
\end{align*}
In conclusion we have 
\begin{equation}
    \omega_r(w_j) \varphi_{r,r}^n (x)=\varphi_{r,r}^{n-j}(x') \varphi_{r,r}^{j}(x'')= \varphi_{r,r}^n(x).
\end{equation}
The invariance under the action of $P(\Oo_{F_0})$ is easy to check. We have completed the proof.
\end{proof}
Hence we conclude that
\begin{corollary}\label{interpolationtrick}
Let $|\cdot|_p$ be the valuation on $\bk_p$ such that $|\sqrt{-\Delta}|_p=\frac{1}{p}$ for a uniformizer $\pi$ of $\bk_p$. Then
\[
    \omega_r(g) \varphi_{r,r}^n (0)=|\mathrm{det}(a(g))|_p^r,
\]
where $g=a(g)n(g)k(g)$ with $k(g)\in K$, $a(g)\in M$ and $n(g)\in N$.
\end{corollary}

\begin{proposition}\label{prop:Whittakerfunctionandlocaldensity}
Let $S\in \Herm_m(\Q_p)$  and define 
\[S_r=S\oplus \frac{1}{\sqrt{-\Delta}} \left(\begin{array}{cc}
    0 & I_r \\
    -I_r & 0 
\end{array}\right).\]
Let $\varphi_p$ the characteristic function of $L^n$ where $L$ is a hermitian $\Oo_F$-lattice with Gram matrix $S$.
Let $\Phi_p$ be the Siegel-Weil section associated to $\varphi_p$ and $V:=L\otimes \Q$. Then 
\[W_{T,p}(r,\Phi_p)=\gamma_p(V)^n |\mathrm{det}(S)|_p^{n} |\Delta|_p^e \alpha_p(S_r,T),\]
where $e=\frac{1}{2} nm+\frac{1}{4}n(n-1)$ and 
\begin{equation}\label{gammapV}
    \gamma_p(V)=(-\Delta,\mathrm{det}(V_p))_p \gamma_p(\Delta,\psi_{p,\frac{1}{2}})^m \gamma_p(-1,\psi_{p,\frac{1}{2}})^{-m}.
\end{equation}
Here $(,)_p$ is the local Hilbert symbol and $\gamma_p(a,\psi_{p,\frac{1}{2}}) $ is the Weil index with respect to the additive character $\psi_{p,\frac{1}{2}}$ given by $\psi_{p,\frac{1}{2}}(x)=\psi_p(\frac{1}{2}x)$, see for example \cite{Kudlasplitting} and \cite[Chapter I.4]{kudla1996notes}. 
\end{proposition}
\begin{remark}
Over $\bk_p$ for an odd prime $p$ inert in $\bk$, we know that 
\[\frac{1}{\sqrt{-\Delta}} \left(\begin{array}{cc}
    0 & I_r \\
    -I_r & 0 
\end{array}\right)\sim
 \left(\begin{array}{cc}
    0 & I_r \\
    I_r & 0 
\end{array}\right).\]
The latter Gram matrix is used in \cite[Proposition 10.1]{KR2}.  However their calculation does not apply to the case when $p$ is ramified in $\bk$.
\end{remark}
\begin{proof}
Define $V_p^{[r]}=V\oplus V_{r,r}$ and
\[\varphi_p^{[r]}=\varphi_p\otimes\varphi_{r,r}^n\in \mathcal{S}((V_p^{[r]})^n).\]
Corollary \ref{interpolationtrick} and \eqref{eq:W_T,v} imply
\begin{align*}
    W_{T,p}(r,\Phi_p)=\int_{N(F_0)}\psi_p(-\mathrm{Tr}(Tb))\gamma_p(V)^n\int_{(V_p^{[r]})^n} \psi_p(\mathrm{Tr}(b(x,x)))\varphi_p^{[r]}(x)dxdb.
\end{align*}
The proof afterwards is exactly the same as that of \cite[Proposition 10.1]{KR2}. 
\end{proof}

\noindent
\textit{Proof of Proposition \ref{W'PhiandWPhi'}}:
By Proposition \ref{prop:Whittakerfunctionandlocaldensity} and \eqref{alpha'andF'}, we have 
\[W'_{T,p}(0,\Phi_p)=2\gamma_p(V_p)^2 p^{-\frac{5}{2}}\alpha'(S,T)  \log p .\]
Now (1) follows from the fact that $\alpha'(S,T)=\frac{1}{2}\alpha(S,S)\mu_p(T)$, see the proof of Theorem \ref{thm:maintheorem2}.
Again by Proposition \ref{prop:Whittakerfunctionandlocaldensity},
\[W_{T,p}(0,\Phi'_p)=\gamma_p(V'_p)^2 p^{-\frac{5}{2}} \alpha_p(S',T).\]
Notice that $\gamma_p(V'_p)=-\gamma_p(V_p)$. Then (2) follows from equation \eqref{alphaS'S'}.
\qedsymbol

\appendix
\section{Kr\"amer model and blow-up}\label{sec:Kramer model is blow up}
\subsection{Definition of local models}\label{subsec:local models}
The goal of this appendix is to prove Proposition \ref{prop: Kra is blow up of Pappas}. First we describe the local model $N^\Kra$ (resp. $N^\Pap$) of $\cN^\Kra$ (resp. $\cN^\Pap$) in the sense of \cite[Definition 3.27]{RZ}.

Let $\Lambda$ be a fixed rank $n$ free $\Oo_{\breve{F}}$-module with canonical basis $\{e_1,\ldots, e_n\}$ and $V=\Lambda\otimes_{\Z} \Q$. Denote by $\Pi$ the $\Oo_{\breve{F}_0}$-linear map induced by the action of $\pi$ on $\Lambda$. Define an alternating $ \breve{F}_0$-bilinear form $\langle,\rangle$ on $V$ such that 
\begin{equation}
    \langle e_i,e_j\rangle=0,\ \langle e_i,\Pi e_j\rangle =\delta_{ij}.
\end{equation}
Then it is true that
\[\langle ax,y\rangle =\langle x,\bar{a} y\rangle, a\in \Oo_F.\]
$N^\Kra$ is the  functor which associates to each $S\in \Nilp \Oo_{\breve{F}}$ the set of pairs $(\cF,\cF_0)$ where $\cF$ and $\cF_0$ are $\Oo_S$-submodules of $\Lambda\otimes_{\Oo_{\breve{F}_0}} \Oo_S$ such that 
\begin{enumerate}
    \item $\cF$ as an $\Oo_S$-module is locally on $S$ a direct summand of rank $n$,\
    \item $\cF_0$ as an $\Oo_S$-module is locally on $S$ a direct summand of rank $1$,\
    \item $\cF_0\subset \cF$, both are $\Oo_{\breve{F}}$-stable,
    \item $\cF$ is  totally isotropic for the form $\langle ,\rangle\otimes_{\Oo_{\breve{F}_0}} \Oo_S$,\
    \item $(\Pi+\pi)\cF\subseteq \cF_0$, where we think of $\pi$ as in $\Oo_S$ via the structure morphism $\Oo_{\breve{F}}\rightarrow \Oo_S$\ 
    \item $(\Pi-\pi)\cF_0=(0)$.
\end{enumerate}
Meanwhile $N^\Pap$ is the  functor which associates to each $S\in \Nilp \Oo_{\breve{F}}$ the set of $\Oo_S$-submodules $\cF$ of $\Lambda\otimes_{\Oo_{\breve{F}_0}} \Oo_S$ such that
\begin{enumerate}
    \item $\cF$ as an $\Oo_S$-module is locally on $S$ a direct summand of rank $n$,\
    \item $\cF$ is $\Oo_{\breve{F}}$-stable,
    \item $\cF$ is  totally isotropic for the form $\langle ,\rangle\otimes_{\Oo_{\breve{F}_0}} \Oo_S$,\
    \item The characteristic polynomial $\mathrm{char}(\Pi|\cF)$ is $(T+\pi)^{n-1}(T-\pi)$,\ 
    \item We need to impose the following wedge conditions when $n\neq 2$:
    \[\wedge^{n}(\Pi-\pi|\cF)=0,\ \wedge^{2}(\Pi+\pi|\cF)=0.\]
\end{enumerate}
The singular locus of $N^\Pap$ consists of a unique point $z_s$ in its special fiber corresponding to the submodule 
\[\cF_s=\Pi \bar\Lambda,\]
where $\bar\Lambda=(\Lambda\otimes_{\Oo_{\breve{F}_0}} k)$.
By the main theorem of \cite{Kr}, the forgetful functor $(\cF,\cF_0)\mapsto \cF$ defines a morphism $\Phi^\loc:N^\Kra\rightarrow N^\Pap$ which induces an isomorphism outside $z_s$. The exceptional locus $(\Phi^\loc)^{-1}(z_s)$ is a divisor in $N^\Kra$ isomorphic to the projective space $\bP^{n-1}/k$.

\subsection{Equations of local model}
By the argument before Proposition 4.14 of \cite{P}, we know that $N^\Pap$ is the affine scheme $\Spec\Oo_{\breve F}[b_{ij}]_{1\leq i,j \leq n}/I^\Pap$ where $I^\Pap$ is the ideal generated by
\begin{equation}\label{eq:I Pap}
    B-B^t,\ \wedge^2(B+\pi I_n),\ B^2-\pi_0 I_n, \ \mathrm{Tr}(B)+(n-2)\pi
\end{equation}
where $B=(b_{ij})_{1\leq i,j \leq n}$. Let $R^\Pap$ be the quotient ring. In this coordinate, the submodule $\cF$ is generated by the column vectors (see the paragraph before Equation (4.8) of \cite{P})
\[(B,I_n)^t.\]

Let $N^\Blow$ be the blow-up of $N^\Pap$ along $z_s$. Let $T=(t_{ij})$ be the $n\times n$ matrix with entries the variables $t_{ij}$.
Let $\Proj(R^\Pap,[t,t_{ij}])$ be the projective space over $R^\Pap$ whose homogeneous coordinate ring is the set of polynomials $R^\Pap[t,t_{ij}]_{1\leq i,j \leq n}$ homogeneous in the variables $t,t_{ij}, 1\leq i,j \leq n$. By \cite[Proposition 4.14]{P}, $N^\Blow$ is the subscheme of $\Proj(R^\Pap,[t,t_{ij}])$ defined by the homogeneous ideal generated by 
\begin{align}\label{eq: I blow}
     b_{ij} t-\pi t_{ij},\ & b_{ij} t_{pq}- b_{pq} t_{ij}, \\ \nonumber
      T-T^t,\ \wedge^2(T+t I_n),\ & T^2-t^2 I_n, \ \mathrm{Tr}(T)+(n-2)t.
\end{align}
It is observed in the proof of \cite[Proposition 4.14]{P} that $N^\Blow$ can be covered by $n+1$ affine charts on which $t$ and $t_{ii}$, $1\leq i \leq n$ are invertible respectively. We push this observation one step further. Define the following change of variables
\[s_{ij}=\delta_{ij} t+t_{ij}.\]
\begin{lemma}
Let $V_i$ be the affine subscheme $N^\Blow$ on which $s_{ii}$ is invertible. Then $\{V_i\mid 1\leq i \leq n\}$ is a cover of $N^\Blow$.
\end{lemma}
\begin{proof}
Suppose $\{V_i\mid 1\leq i \leq n\}$ does not cover $N^\Blow$. Then there is a geometric point $z$ of $N^\Blow$ such that 
\[t+t_{ii}=0, 1\leq i \leq n\]
in the residue field $k_z$. Then the relation $\mathrm{Tr}(T)+(n-2)t=0$ tells us that $t=0$ in $k_z$ since we assume the residue characteristic is not $2$. This in turn implies that $t_{ii}=0$ in $k_z$. This is a contradiction with the observation made by loc.cit..
\end{proof}
Let $S=(s_{ij})$, then $N^\Blow$ is the subscheme of $\Proj(R^\Pap,[t,s_{ij}])$ defined by the homogeneous ideal generated by 
\begin{align}\label{eq: I blow'}
     (b_{ij}+\pi \delta_{ij}) t-\pi s_{ij},\ & b_{ij} (s_{pq}-\delta_{pq}t)- b_{pq} (s_{ij}-\delta_{ij}t), \\ \nonumber
      S-S^t,\ \wedge^2 S,\ & (S-tI_n)^2-t^2 I_n, \ \mathrm{Tr}(S)-2t.
\end{align}

We now describe the local equations of $N^\Kra$ near $(\Phi^\loc)^{-1}(z_s)$, see step 4 of the proof of \cite[Theorem 4.5]{Kr}. $N^\Kra$ is covered by $n$ affine charts $U_i=\Spec R_i$ ($1\leq i \leq n$) where
\begin{equation}\label{eq:app Kramer model equation}
    R_i=\Oo_{\breve F}[x_i^1,\ldots x_i^n, y_i]/(y_i(\sum_{j=1}^n(x_i^j)^2)-2\pi, x^i_i-1).
\end{equation}
The following transition relations between coordinate functions hold in $U_i\cap U_j$ where $x_i^j$ and $x_j^i$ are units.
\begin{equation}\label{eq:app coordinate change of Kramer model}
    x_i^\ell=x_i^j x_j^\ell,\ y_i=(x_j^i)^2 y_j.
\end{equation}
Define the column vector $\vec{x}_i=(x_i^1,\ldots,x_i^n)^t$. Then the $n\times n$ matrix
\begin{equation}\label{eq:matrix A}
    A=y_i \vec{x}_i \vec{x}_i^t-\pi I_n
\end{equation}
is independent of the chart $U_i$.
With respect to the basis $\{e_1,\ldots, e_n,\Pi e_1,\ldots, \Pi e_n\}$ of $\Lambda$, in the chart $U_i$ the submodule $\cF_0$ and $\cF$ are generated respectively by (see Equation (4.5) of \cite{Kr})
\[(A\vec{x}_i,\vec{x}_i)^t\]
and the column vectors of  (see the second last equation before Equation (4.3) of loc.cit.)
\[(A,I_n)^t.\]
The morphism $\Phi^\loc$ is determined by the $\Oo_{\breve F}$-algebra homomorphisms
\[(\Phi_i^\loc)^\sharp:R^\Pap\rightarrow R_i, \ B \mapsto A.\]
Using equation \eqref{eq:app Kramer model equation} and \eqref{eq:matrix A}, one can easily see that the relations generated by \eqref{eq:I Pap} are preserved under $(\Phi_i^\loc)^\sharp$.
The exceptional locus $(\Phi^\loc)^{-1}(z_s)$ is a divisor defined by the single equation $y_i=0$ in each $U_i$.

\subsection{The proof}
By the universal property of blow-up (see for example \cite[Chapter II, Proposition 7.14]{hartshorne2013algebraic}) there is a unique morphism $N^\Kra\rightarrow N^\Blow$ making the following diagram commute.
\begin{equation}\label{eq:NKra Nblow Npap diagram}
   \begin{tikzcd}
N^\Kra \arrow[r] \arrow[dr, "\Phi^\loc"]
& N^\Blow \arrow[d]\\
& N^\Pap
\end{tikzcd} 
\end{equation}
We would like to explicitly construct this morphism. Define $\Psi_i^\sharp:\Oo_{V_i}\rightarrow R_i$ by 
\begin{equation}\label{eq:Psi sharp}
  \Psi_i^\sharp(t/s_{ii})=\frac{1}{2} \vec{x}_i^t\cdot \vec{x}_i, \ \Psi_i^\sharp(S/s_{ii})=\vec{x}_i\vec{x}_i^t, \ \Psi_i^\sharp(B)=A.  
\end{equation}

\begin{lemma}
The maps $\Psi_i^\sharp:\Oo_{V_i}\rightarrow R_i$ are $\Oo_{\breve F}$-algebra homomorphisms. Moreover they glue along $U_i\cap U_j$ to give a morphism $\Psi:N^\Kra\rightarrow N^\Blow$ which makes the diagram \eqref{eq:NKra Nblow Npap diagram} commute.
\end{lemma}
\begin{proof}
In order to show that this is an $\Oo_{\breve F}$-algebra homomorphism,
we need to check that the relations generated by \eqref{eq: I blow'} are preserved. Notice that \eqref{eq:app Kramer model equation} and \eqref{eq:matrix A} implies that
\[A=y_i(\vec{x}_i\vec{x}_i^t-\frac{1}{2}\vec{x}_i^t\cdot \vec{x}_i I_n)=y_i(\Psi_i^\sharp(S/s_{ii})-\Psi_i^\sharp(t/s_{ii})I_n) .\]
This together with $\Psi_i^\sharp(B)=A$ directly implies that 
\[\Psi_i^\sharp((b_{ij} (s_{pq}-\delta_{pq}t)- b_{pq} (s_{ij}-\delta_{ij}t))/s_{ii})=0.\]
By \eqref{eq:app Kramer model equation} and \eqref{eq:matrix A}, we also have
\begin{align*}
    &\Psi_i^\sharp(\frac{t}{s_{ii}}(B+\pi I_n)-\frac{1}{s_{ii}}\pi S)\\
    =&\frac{1}{2} \vec{x}_i^t\cdot \vec{x}_i (A+\pi I_n)-\pi \vec{x}_i\vec{x}_i^t\\
    =& \frac{1}{2} \vec{x}_i^t\cdot \vec{x}_i (y_i \vec{x}_i\vec{x}_i^t)- \frac{1}{2} (\vec{x}_i^t\cdot \vec{x}_i y_i) \vec{x}_i\vec{x}_i^t\\
    =&0.
\end{align*}
The relations in the second line of \eqref{eq: I blow'} can be checked directly using the definition of $\Psi_i^\sharp$. The equation $\Psi_i^\sharp(B)=A$ indicates that the relations generated by \eqref{eq:I Pap} are preserved since the same relations are satisfied by $A$.

We also need to check that the definition of $\Psi_i^\sharp$ and $\Psi_j^\sharp$ are compatible in $U_i\cap U_j$, so they glue together to give a morphism $\Psi:N^\Kra\rightarrow N^\Blow$. First we need to check that
\begin{equation}\label{eq: Psi sharp change coordinate}
   \Psi_j^\sharp(t/s_{jj})\Psi_i^\sharp(s_{jj}/s_{ii})=\Psi_i^\sharp(t/s_{ii}). 
\end{equation}
The second equation of \eqref{eq:Psi sharp} tells us $\Psi^\sharp_i(s_{jj}/s_{ii})=(x_i^j)^2$. Then \eqref{eq: Psi sharp change coordinate} follows from \eqref{eq:app coordinate change of Kramer model}. Similarly one can verify that 
\[\Psi_j^\sharp(S/s_{jj})\Psi_i^\sharp(s_{jj}/s_{ii})=\Psi_i^\sharp(S/s_{ii}).\]
The relation $\Psi_i^\sharp(B)=A$ guarantees the commutativity of the diagram \eqref{eq:NKra Nblow Npap diagram} if the horizontal morphism is $\Psi$. 
\end{proof}

\begin{lemma}\label{lem:local model isomorphism}
The morphism $\Psi:N^\Kra\rightarrow N^\Blow$ is an isomorphism. 
\end{lemma}
\begin{proof}
It suffices to construct an inverse of each $\Psi_i^\sharp$. Define $\Omega_i^\sharp:R_i \rightarrow \Oo_{V_i}$ by
\begin{equation}
    \Omega_i^\sharp(y_i)=b_{ii}+\pi,\ \Omega_i^\sharp(x_i^j)=s_{ji}/s_{ii}.
\end{equation}
The $(i,i)$-th entry of the relation $(S-tI_n)^2-t^2 I_n=0$ together with the relation $S=S^t$ tells us that
\[(s_{ii}-t)^2+\sum_{j=1}^n s_{ji}^2-t^2=0.\]
In other words
\begin{equation}\label{eq: s_ji square sum}
  2s_{ii}t=\sum_{j=1}^n s_{ji}^2.  
\end{equation}
Hence 
\[\Omega_i^\sharp(y_i (\vec{x}_i^t \cdot \vec{x}_i))
    =(b_{ii}+\pi) \sum_{j=1}^n (s_{ji}/s_{ii})^2
    = \frac{2t(b_{ii}+\pi)}{s_{ii}}
    = 2\pi.\]
In the last equality we used the first relation in \eqref{eq: I blow'}. This shows that $\Omega_i^\sharp$ is an $\Oo_{\breve F}$-algebra homomorphism. 

We need to show that $\Omega_i^\sharp$ and $\Psi_i^\sharp$ are inverse of each other. Let $\vec{s}_i$ be the column vector $(s_{1i}/s_{ii},\ldots,s_{ni}/s_{ii})^t$. We verify the desired relations one by one.
\[\Psi_i^\sharp\circ \Omega_i^\sharp(y_i)=\Psi_i^\sharp(b_{ii}+\pi)=a_{ii}+\pi=y_{i}(x_i^i)^2=y_i.\]
\[\Psi_i^\sharp\circ \Omega_i^\sharp(x_i^j)=\Psi_i^\sharp(s_{ji}/s_{ii})=x_{i}^j x_{i}^i=x_{i}^j.\]
\[\Omega_i^\sharp\circ \Psi_i^\sharp(t/s_{ii})=\frac{1}{2}\Omega_i^\sharp(\vec{x}_i^t\cdot \vec{x}_i)=\frac{1}{2}\vec{s}_i^t\cdot \vec{s}_i=t/s_{ii}.\]
Here in the last equality we use \eqref{eq: s_ji square sum} again.
\[\Omega_i^\sharp\circ \Psi_i^\sharp(S/s_{ii})=\Omega_i^\sharp(\vec{x}_i \vec{x}_i^t)=\vec{s}_i \vec{s}_i^t=S/s_{ii}.\]
Here in the last equality we use the fact that $\wedge^2 S=0$, hence $S/s_{ii}$ must be of the form $\lambda \vec{s}_i \vec{s}_i^t$. Equating the $(i,i)$-th entry tells us that $\lambda=1$.  This finishes the proof of the lemma.
\end{proof}

\noindent\textit{Proof of Proposition \ref{prop: Kra is blow up of Pappas}}:
Let $M$ denote the functor that assigns to each strict $\Oo_{F_0}$-module the $\Oo_{F_0}$-strict eigenspace of the Lie algebra of its universal extension. Define $\tilde{\cN}^\Kra$ (resp. $\tilde{\cN}^\Pap$) to be the functor whose $S$-point ($S\in \Nilp_{\Oo_{\breve F}}$) is the groupoid of the isomorphism classes of the following data
\begin{itemize}
    \item $(X,\iota,\lambda,\rho,\cF_X)\in \cN^\Kra(S)$ (resp. $(X,\iota,\lambda,\rho)\in \cN^\Pap(S)$).
    \item An isomorphism $\gamma$ of $\Oo_{\breve F}\otimes_{\Oo_{\breve{F}_0}} \Oo_S$-modules:
    \[\gamma:M(X)\xrightarrow{\sim} \Lambda\otimes_{\Oo_{\breve{F}_0}} \Oo_S\]
    that respects the alternating forms on both sides.
\end{itemize}
Then we have the following local model diagrams (see \cite[Definition 3.28]{RZ}):
\[
\begin{tikzcd}[column sep=small]
& \tilde{\cN}^\Kra \arrow{dl}[swap]{\mu} \arrow{dr}{\mu'} & \\
\cN^\Kra  & & N^\Kra
\end{tikzcd}
\text{ and }
\begin{tikzcd}[column sep=small]
& \tilde{\cN}^\Pap \arrow{dl}[swap]{\nu} \arrow{dr}{\nu'} & \\
\cN^\Pap  & & N^\Pap
\end{tikzcd}.
\]
Here $\mu$ and $\nu$ are natural forgetful functors and $\mu'$ and $\nu'$
are defined by 
\[\mu'(X,\iota,\lambda,\cF_0)=(\cF,\cF_0),\]
where $\cF=\gamma(\mathrm{Ker}[M(X)\rightarrow\Lie X])$ (see \cite[Section 3.29]{RZ}) and $\cF_0=\gamma(\cF_X)^\bot$, and 
\[\nu'(X,\iota,\lambda,\rho)=\cF.\]
Here the alternating form $\langle,\rangle$ on $\Lambda$ induces a non-degenerate pairing between $\Lambda\otimes_{\Oo_{\breve{F}_0}} \Oo_S/\cF$ and $\cF$, and we use $\gamma(\cF_X)^\bot$ to denote the perpendicular of the image of $\cF_X$ under the map $\Lie X\rightarrow \Lambda\otimes_{\Oo_{\breve{F}_0}} \Oo_S/\cF$ induced by $\gamma$. Since $\Oo_F$ acts on $\cF_X$ by the structural morphism and on $ \cF/\cF_X$ by the conjugate of the structural morphism, the pair $(\cF,\cF_0)$ satisfies Condition (5) and (6) in the definition of $N^\Kra$.  
The forgetful functor induces a morphism from the first local model diagram to the second.

Let $\mathcal U$ be an \'etale neighbourhood  of a point $x$ in the special fiber of $\cN^\Pap$ and $\mathcal{U}^\Kra$ be its base change to $\cN^\Kra$.
A section $s:\mathcal{U}^\Kra\rightarrow \tilde{\cN}^\Kra$ rigid of first order in the sense of \cite[Definition 3.31]{RZ} (whose existence is proved in loc.cit.) induces by the forgetful functor a section $s':\mathcal{U}\rightarrow \tilde{\cN}^\Pap$. Hence we have the following commutative diagram 
\[
\begin{tikzcd}
\mathcal{U}^\Kra \arrow{r}{\mu \circ s} \arrow{d} & N^\Kra \arrow{d}\\
\mathcal{U} \arrow{r}{\mu' \circ s'} & N^\Pap
\end{tikzcd}
\]
where the vertical maps are forgetful functors and the horizontal maps are formally \'etale Zariski locally by \cite[Proposition 3.33]{RZ}.
Let $\cN^\Blow$ be the blow-up of $\cN^\Pap$ along its singular locus. Then by the universal property of blowing up, there is a morphism $\cN^\Kra\rightarrow \cN^\Blow$.
The above commutative diagram induces the following 
commutative diagram (notice that blow up commutes with flat base change)
\[
\begin{tikzcd}
\mathcal{U}^\Kra \arrow{r}{\mu \circ s} \arrow{d} & N^\Kra \arrow{d}\\
\mathcal{U}^\Blow \arrow{r} & N^\Blow
\end{tikzcd}
\]
where $\mathcal{U}^\Blow$ is the base change of $\mathcal{U}$ to $\cN^\Blow$ and the lower horizontal map is again formally \'etale Zariski locally. By Lemma \ref{lem:local model isomorphism}, we know that $\mathcal{U}^\Kra\rightarrow\mathcal{U}^\Blow$ is formally \'etale. Hence $\cN^\Kra\rightarrow\cN^\Blow$ is formally \'etale, hence \'etale as it is locally of finite presentation by construction. 

It suffices to show that the finite morphism $\cN^\Kra\rightarrow\cN^\Blow$ has degree $1$. 
When $n\geq 3$ or $n=2$ and $\bV$ is isotropic, this follows from the fact that both $\cN^\Kra$ and $\cN^\Blow$ agree with $\cN^\Pap$ outside the singular locus. When $n=2$ and $\bV$ is isotropic, we can embed $\cN^\Pap_2$ into $\cN^\Pap_3$ by taking $(X,\iota,\lambda,\rho)$ over $\Spec S$ to $(X\times \cG,\iota\times \iota_\cG,\lambda\times \lambda_\cG,\rho \times \rho_\cG)$  where $(\cG,\iota_\cG,\lambda_\cG,\rho_\cG)$ is the canonical lifting and the subscript in $\cN^\Pap$ indicates dimension. The embedding then induces a commutative diagram 
\[
\begin{tikzcd}
\cN^\Kra_2 \arrow{r} \arrow{d} & \cN^\Kra_3 \arrow{d}\\
\cN^\Blow_2 \arrow{r} & \cN^\Blow_3
\end{tikzcd}.
\]
In particular this shows that $\cN^\Kra_2$ and $\cN^\Blow_2$ have the same $k$-points. This finishes the proof of the claim and the proposition.
\qedsymbol

\noindent
{\bf Conflict of Interest Statement:}
On behalf of all authors, the corresponding author states that there is no conflict of interest.

\bibliographystyle{alpha}
\bibliography{reference}

\end{document}